\documentclass[11pt,oneside]{amsart}

\pdfoutput=1
\usepackage[centering]{geometry}
\usepackage{amssymb,bm,mathrsfs,xcolor,upgreek}

\usepackage[shortlabels]{enumitem}

\usepackage[hyperfootnotes=true,psdextra]{hyperref}
\hypersetup{hypertexnames=false,colorlinks=true,citecolor=couleur_cite,linkcolor=couleur_link,urlcolor=couleur_url,
pdfstartview=FitH, pdfauthor=Valentin Blomer and Gergely Harcos and Peter Maga and Djordje Milicevic, pdftitle=The non-spherical sup-norm problem for GL(n) and generalized spherical functions}

\usepackage[normalem]{ulem}
\newcommand{\stkout}[1]{\ifmmode\text{\sout{\ensuremath{#1}}}\else\sout{#1}\fi}

\definecolor{couleur_cite}{rgb}{0,0,1}
\definecolor{couleur_link}{rgb}{0,0,1}
\definecolor{couleur_url}{rgb}{0,0,1}
\definecolor{dmgreen}{rgb}{0.3, 0.73, 0.09}

\numberwithin{equation}{section}

\allowdisplaybreaks[1]

\theoremstyle{remark}
\newtheorem{remark}{Remark}

\theoremstyle{plain}
\newtheorem{theorem}{Theorem}
\newtheorem{lemma}{Lemma}

\DeclareMathOperator{\ad}{ad}
\DeclareMathOperator{\adj}{adj}
\DeclareMathOperator{\diag}{diag}
\DeclareMathOperator{\dist}{dist}
\DeclareMathOperator{\rk}{rk}
\DeclareMathOperator{\sgn}{sgn}
\DeclareMathOperator{\tr}{tr}
\DeclareMathOperator{\Hom}{Hom}
\DeclareMathOperator{\End}{End}
\DeclareMathOperator{\Ind}{Ind}
\DeclareMathOperator{\Pol}{Pol}
\DeclareMathOperator{\supp}{supp}

\newcommand\GL{\mathrm{GL}}
\newcommand\SL{\mathrm{SL}}
\newcommand\PGL{\mathrm{PGL}}
\newcommand\OO{\mathrm{O}}
\newcommand\UU{\mathrm{U}}
\newcommand\SO{\mathrm{SO}}
\newcommand\PO{\mathrm{PO}}
\newcommand\id{\mathrm{id}}
\newcommand\BB{\mathrm{B}}

\newcommand\NN{\mathbb{N}}
\newcommand\ZZ{\mathbb{Z}}
\newcommand\QQ{\mathbb{Q}}
\newcommand\RR{\mathbb{R}}
\newcommand\CC{\mathbb{C}}

\newcommand{\mcA}{\mathcal{A}}
\newcommand{\mcD}{\mathcal{D}}
\newcommand{\mcE}{\mathcal{E}}
\newcommand{\mcF}{\mathcal{F}}
\newcommand{\mcL}{\mathcal{L}}
\newcommand{\mcQ}{\mathcal{Q}}
\newcommand{\mcP}{\mathcal{P}}
\newcommand{\mcS}{\mathcal{S}}
\newcommand{\mcT}{\mathcal{T}}
\newcommand{\mcY}{\mathcal{Y}}
\newcommand{\mcZ}{\mathcal{Z}}

\newcommand\mfa{\mathfrak{a}}
\newcommand\mfb{\mathfrak{b}}
\newcommand\mfg{\mathfrak{g}}
\newcommand\mfh{\mathfrak{h}}
\newcommand\mfk{\mathfrak{k}}
\newcommand\mft{\mathfrak{t}}
\newcommand\mfB{\mathfrak{B}}

\newcommand{\ba}{\mathbf{a}}
\newcommand{\bb}{\mathbf{b}}
\newcommand{\bk}{\mathbf{k}}
\newcommand{\mbm}{\mathbf{m}}
\newcommand{\br}{\mathbf{r}}
\newcommand{\mbs}{\mathbf{s}}
\newcommand{\bt}{\mathbf{t}}
\newcommand{\bu}{\mathbf{u}}
\newcommand{\bv}{\mathbf{v}}
\newcommand{\bw}{\mathbf{w}}
\newcommand{\bx}{\mathbf{x}}
\newcommand{\by}{\mathbf{y}}
\newcommand{\bz}{\mathbf{z}}

\newcommand\wigd{\mathsf{d}}

\newcommand\ov{\overline}
\newcommand\bs{\backslash}
\newcommand\eps{\varepsilon}

\renewcommand\leq{\leqslant}
\renewcommand\geq{\geqslant}

\newcommand\Gtemp{\widehat{G}_\mathrm{temp}}
\newcommand\Gadm{\widehat{G}_\mathrm{adm}}
\newcommand\Gqsf{\widehat{G}_\mathrm{qsf}}
\newcommand\GGqsf{\widehat{H}_\mathrm{qsf}}

\DeclareFontFamily{U}{mathx}{\hyphenchar\font45}
\DeclareFontShape{U}{mathx}{m}{n}{
      <5> <6> <7> <8> <9> <10>
      <10.95> <12> <14.4> <17.28> <20.74> <24.88>
      mathx10
      }{}
\DeclareSymbolFont{mathx}{U}{mathx}{m}{n}
\DeclareFontSubstitution{U}{mathx}{m}{n}
\DeclareMathAccent{\widecheck}{0}{mathx}{"71}
\DeclareMathAccent{\wideparen}{0}{mathx}{"75}

\makeatletter
\NewCommandCopy\@@pmod\pmod
\DeclareRobustCommand{\pmod}{\@ifstar\@pmods\@@pmod}
\def\@pmods#1{\mkern4mu({\operator@font mod}\mkern 6mu#1)}
\makeatother

\title[The non-spherical sup-norm problem for $\GL(n)$]{The non-spherical sup-norm problem for $\GL(n)$ and generalized spherical functions}

\author{Valentin Blomer}
\address{Mathematisches Institut, Endenicher Allee 60, D-53115 Bonn, Germany}
\email{blomer@math.uni-bonn.de}

\author{Gergely Harcos, P\'eter Maga}
\address{Alfr\'ed R\'enyi Institute of Mathematics, POB 127, Budapest H-1364, Hungary}\email{gharcos@renyi.hu, magapeter@gmail.com}
\address{MTA--HUN-REN RI Lend{\"u}let Automorphic Research Group}\email{gharcos@renyi.hu}
\address{MTA--HUN-REN RI Lend{\"u}let Analytic Number Theory and Representation Theory Research Group}\email{magapeter@gmail.com}

\author{Djordje Mili\'cevi\'c}
\address{Bryn Mawr College, Department of Mathematics, 101 North Merion Avenue, Bryn Mawr, PA 19010, USA}
\email{dmilicevic@brynmawr.edu}

\thanks{VB supported by DFG through SFB-TRR 358/1 2023 - 491392403 and EXC-2047/1 - 390685813 and by ERC
Advanced Grant 101054336. GH supported by the MTA--HUN-REN RI Lend{\"u}let Automorphic Research Group and NKFIH (National Research, Development and Innovation Office) grant K~143876. PM supported by the MTA--HUN-REN RI Lend\"ulet Analytic Number Theory and Representation Theory Research Group. DM supported in part by the Simons Foundation Award MPS-TSM-00008085 and the Charles Simonyi Endowment.}

\AtBeginDocument{%
   \def\MR#1{}
}


\begin{document}

\begin{abstract}
We solve the sup-norm problem for minimal weight vectors in an arbitrary cuspidal representation $\pi $ of $\GL(n,\ZZ)\bs\GL(n,\RR)$ with a uniform power saving bound in terms of the archimedean data of $\pi$, including its spectral parameters and the dimension of its minimal $K$-type. As a key ingredient, we establish new uniform decay bounds for generalized spherical functions.
\end{abstract}

\maketitle

\section{Introduction}
A central problem in analytic number theory is to understand the fundamental asymptotic behavior of automorphic forms, regarded as joint eigenfunctions of invariant operators, as the complexity of the underlying physical system increases. In this paper, we solve the sup-norm problem for minimal weight vectors in any cuspidal representation of $\GL(n,\ZZ)\bs\GL(n,\RR)$ and develop new tools including a versatile spectral localizer and uniform decay bounds for the spherical trace functions that enter the general $\tau$-spherical transform, which open the doors to a broader analysis of non-spherical automorphic forms on $\GL(n,\RR)$.

\subsection{The sup-norm problem}
The sup-norm problem asks to bound the $L^{\infty}$-norm of an $L^2$-normalized eigenfunction $\phi$ on a Riemannian manifold (or orbifold) $X$ in terms of its Laplace eigenvalue $\lambda$. While this is a purely analytic question, a typical situation of arithmetic interest is the case when $X$ is an arithmetic locally symmetric space, i.e. $X = \Gamma \bs G/K$ for some reductive real Lie group $G$, a maximal compact subgroup $K$ and an arithmetic lattice $\Gamma$. This space is equipped with many more operators respecting the $G$-action and commuting with the Laplace operator: the algebra of invariant differential operators on $X$ has $\rk_{\RR}(G)$ generators, and in addition there is an infinite commutative family of Hecke operators acting on $X$. From the point of view of automorphic forms, it is most natural to consider simultaneous eigenfunctions of all these operators, since they generate the underlying irreducible automorphic representation. Applications of various kinds of sup-norm bounds in such arithmetic situations range from subconvexity for $L$-functions \cite{IS}, bounds for the number of nodal domains \cite{GRS}, equidistribution of zeros \cite{Ru} and bounds for Faltings' height function through covers \cite{JK}. We refer to \cite{MR4046009, BHMM, MR4055171, MR4089374, MR4150475, MR4190047, MR4193473, MR4282133, MR4307129, MR4552668, AnambyDas2026, MR4683839, MR4728730, MR4735817, MR4731074, MR4886334, MR4752128, steiner2023thetafunctionsfourthmoments, MR4878397, MR5038903, MR5119486, assing2025orbitmethodnumbertheory,lutsko2026improvedsupnormboundslocally} for a non-exhaustive list of recent achievements on this topic in a variety of different settings.

If $X$ is compact, the generic bound for non-constant $\phi$ is\footnote{See \S\ref{subsec:notations} for asymptotic notations.}
\begin{equation}\label{equation}
\| \phi \|_{\infty} \ll_X \lambda^{(\dim X - \rk X)/4}.
\end{equation}
If $X$ is non-compact, the behavior of Whittaker functions in transitional regions leads to exceptionally high peaks of $\phi$ near the cusps that distort the typical behavior of $\phi$, a phenomenon that has been analyzed by Brumley and Templier~\cite{MR4150475} in the case of $G = \GL(n)$. It is therefore reasonable and customary to fix a compact set $\Omega \subseteq X$ and study the restricted sup-norm $\| \phi|_{\Omega} \|_{\infty}$, which does satisfy \eqref{equation}. The sup-norm problem asks for a power-saving improvement relative to this generic bound.
In \cite{BM}, the authors achieved this in the case $G = \PGL_n(\RR)$, $K = \PO(n)$, $\Gamma = \PGL_n(\ZZ)$, 
where $\dim X = n(n+1)/2 - 1$ and $\rk X = n-1$.

In this paper, we open a new chapter (of which only a special predecessor \cite{BHMM} is available) by dropping the condition that the eigenfunction is spherical, i.e. right $K$-invariant. Instead, we consider eigenfunctions on $G$ with arbitrary $K$-types $\tau$, and our sup-norm bounds in Theorems~\ref{thm1} and~\ref{thm2} are not only uniform but in fact power-saving in $\dim \tau$. Automorphic forms with nontrivial $K$-types include for instance the symmetric powers of $\GL(2)$ holomorphic forms, some of the most prominently studied automorphic forms in higher rank. They are also central players in contexts ranging from Matsushima's formula in low-dimensional topology~\cite{LinLipnwski2022}, equidistribution of holonomy~\cite{SarnakWakayama1999,DeverMilicevic2023}, counting hyperbolic lattice points ordered by arbitrary norms, and more.

The analytic theory of spherical automorphic forms and in particular the spherical sup-norm problem naturally lead to the Fourier transform between functions on $K\bs G/K$ and on $\mfa^{\ast}_{\CC}/W$. The corresponding kernel function is the elementary spherical function $\phi_{\mu}(g)$, studied in great detail since the pioneering work of Harish-Chandra. The corresponding spectral transform of the full automorphic spectrum on $G$ is far less understood. As an analytic tool of independent interest, the present paper develops the analysis of the kernel function in the Fourier transform between certain $K$-central functions $f$ satisfying
\[f = \ov{\chi_{\tau}}\ast f\ast\ov{\chi_{\tau}}\]
for the normalized character $\chi_\tau$ of an arbitrary $K$-type $\tau \in \widehat{K}$ and functions on $\mfa^{\ast}_{\CC}/W$. Theorems~\ref{thm3} and~\ref{thm4} in \S\ref{14}, Theorem~\ref{thm3b} in \S\ref{thm3bsubsection}, and Theorem~\ref{thm4b} in \S\ref{n3-nonminimal-subsec}, feature uniform bounds. On the way, we develop the Paley--Wiener type Theorem~\ref{Thm5} coupled with the spectral localizer provided by Theorems~\ref{theorem:K-type-barrier} and~\ref{theorem:concrete-paley-wiener-function} that should be of use in many other applications of the $\tau$-spherical transform.

\subsection{Non-spherical cusp forms and their archimedean data}
We take some time to prepare the scene and refer to \S\ref{spherical}--\S\ref{sec:generalized-spherical-trace-function-SL} for the details and definitions of all notations. Let $(\pi,V_{\pi})$ be a cuspidal automorphic representation appearing in $L^2(\GL_n(\ZZ) \bs \GL_n(\RR),\omega)$ with some unitary central character $\omega$. Twisting $\pi$ by $|\det|^{it}$ with any $t\in\RR$ does not alter the size of the vectors (functions) in $V_\pi$, hence we may and do assume that $\omega$ is trivial. Then $(\pi,V_{\pi})$ can be thought\footnote{We could equally work with $\PGL_n(\RR)$, but $\SL_n^{\pm}(\RR)$ has some advantages for the general theory.} of as a representation of $G:=\SL_n^{\pm}(\RR)$ left-invariant by $\Gamma := \SL_n^{\pm}(\ZZ)$. The maximal compact subgroup of $G$ is $K:=\OO(n)$, and $V_{\pi}$ admits an orthogonal Hilbert space decomposition
\[
V_{\pi}=\bigoplus_{\tau\in\widehat{K}} \bigoplus_{j=1}^{m_{\pi}(\tau)} V_{\pi,\tau,j},
\]
where each $V_{\pi,\tau,j}$ is isomorphic to $V_{\tau}$ as a representation of $K$. There is a unique minimal $\tau_{\pi}\in\widehat{K}$ (in the sense of \S\ref{sec:generalized-spherical-trace-function-SL}) for which $m_{\pi}(\tau_{\pi})\neq 0$, and then in fact $m_{\pi}(\tau_{\pi})=1$. We write $V_{\pi,\tau_{\pi}}:=V_{\pi,\tau_{\pi},1}$.

Let us assume that $\pi$ is not a discrete series representation\footnote{This assumption is only to ease terminology. We could remove it by allowing $P'=M'=G$ below.} (which is automatic for $n\geq 3$). Then $\pi$ is a generalized principal series representation in the following sense. Let $P' = M'A'N' \subseteq G$ be a parabolic subgroup in its Langlands decomposition containing
a fixed minimal parabolic $P = MAN$. A (generalized) principal $P'$-series representation of $G$ is of the form
\begin{equation}\label{eq:def-U-sigma'-mu'}
U^{\sigma',\mu'}=\Ind_{P'}^G(\sigma'\otimes e^{\mu'}\otimes\mathbf{1}),
\end{equation}
where $\sigma'$ is a discrete series representation of $M'$ and $\mu':\mfa'\to\CC$ is a linear functional. Since discrete series representations for $\SL_m^\pm(\RR)$ only occur when $m \leq 2$, there exists $0 \leq r \leq n/2$ such that $M'$ 
may be written as the product of $r$ factors (``blocks'') isomorphic to $\SL^{\pm}_2(\RR)$, where the restriction of $\sigma'$ to each 
of these $r$ blocks is the discrete series representation $\delta_{\ell}$ for some integer $\ell\geq 2$, and $s:=n- 2r$ is the number of $\{\pm 1\}$ blocks. This provides us with a vector ${\bm \ell} = (\ell_1, \ldots, \ell_r)$, and correspondingly we write $\mu' = (\mu'_1, \ldots, \mu'_r, \mu'_{r+ 1} , \ldots, \mu'_{r+s})$. We may and do assume that
\begin{equation}\label{eq:mu'-sum-0}
2\sum_{j=1}^{r} \mu'_j + \sum_{j=r+1}^{r+s} \mu'_j=0.
\end{equation}
To measure the size of this archimedean data, we introduce\footnote{For a discrete series representation $\pi=\delta_\ell$, we set $\bm\ell:=(\ell)$, $\mu':=(0)$, $\|\pi\|_\infty:=\ell$.}
\begin{equation}\label{eq:unorm}
\|U^{\sigma',\mu'}\|_{\infty} := \| \bm \ell\|_{\infty} + \| \mu' \|_{\infty},
\end{equation}
with the convention that $\|\bm \ell\|_\infty=0$ when $r=0$. By \cite[Th.~14.10.20]{GH}, we can assume that $\pi=U^{\sigma',\mu'}$, where $\sigma'$ is a discrete series representation of a suitable $M'$ and $|\Re\mu'|<1/2$. In fact $|\Re\mu'|\leq 1/2-1/(n^2+1)$ by \cite[Prop.~3.3]{MR2053600}, which implies that $\|\pi\|_{\infty}$ is essentially the size of the infinitesimal character of $\pi$ (cf. \eqref{picharacter}): $\|\pi\|_\infty\asymp_n\|\Lambda_\pi\|$.

We equip the coset space $\Gamma\bs G$ with the unique right $G$-invariant probability measure, and we consider an orthonormal basis 
$\Phi:=(\phi_1,\dotsc,\phi_{\dim \tau_{\pi}})$ of $V_{\pi,\tau_{\pi}}$ viewed as a vector-valued function on $G$. Such a $\Phi$ will be referred to as $L^2$-normalized and of minimal $K$-type. We put
\[
\|\Phi(x)\|:=\sqrt{\sum_{j=1}^{\dim\tau_{\pi}} |\phi_j(x)|^2},\qquad x\in G,
\]
and we fix a compact set $\Omega\subseteq G$. Our ultimate goal is to estimate
\[
\|\Phi|_{\Omega}\|_{\infty}:=\sup_{x\in\Omega} \|\Phi(x)\|
\]
in terms of $\|\pi\|_\infty$.

We start by explaining the baseline bound, often called the trivial bound. Here ``trivial'' does not refer to easiness, but the fact that it does not use arithmeticity. For $\pi=U^{\sigma',\mu'}$ as above, we introduce
\begin{equation}\label{spectralcontent}
\Delta(\pi):=\int_{B(\pi)} d\varpi,
\end{equation}
where $B(\pi)$ is an ``etalon ball'' around $\pi$ in the unitary dual $\widehat{G}$ (see \eqref{eq:def-unit-ball-in-dual} below for a formal definition), and the integral on the right-hand side is meant with respect to the Plancherel measure. Then the trivial bound is
\begin{equation}\label{triv}
\| \Phi|_{\Omega} \|_{\infty} \ll_{\Omega} \left(\Delta(\pi) \cdot \dim \tau_{\pi}\right)^{1/2},
\end{equation}
which is a refinement and an analogue of \eqref{equation}. We note that the right-hand side can be crudely bounded as $\ll_n (1+\|\pi\|_\infty)^{3n^2/8}$; see \eqref{eq:Delta-KS} and \eqref{eq:dim-tau-size} for the details.

\subsection{Non-spherical sup-norm bounds}
Our first main result establishes a power-saving on the right-hand side of \eqref{triv} by exploiting the underlying arithmetic structure.

\begin{theorem}\label{thm1} There exists some $\delta > 0$ depending only on $n$ such that the following holds. For any compact set $\Omega\subseteq G$, and any $L^2$-normalized vector-valued automorphic form $\Phi$ generating $\pi$ and of minimal $K$-type $\tau_{\pi}$, we have
\[
\| \Phi|_{\Omega} \|_{\infty} \ll_{\Omega} (\Delta(\pi) \cdot \dim \tau_{\pi})^{1/2 - \delta}.
\]
\end{theorem}

\begin{remark} As noted earlier, this result solves the non-spherical sup-norm problem for all $L^2$-normalized vector-valued cusp forms on $\GL_n(\ZZ)\bs \GL_n(\RR)$ of minimal $K$-type with a unitary central character. The passage to $\Gamma\bs G$ is provided by a suitable unitary determinant twist.
\end{remark}

\begin{remark}\label{rem2} In Theorem~\ref{thm1} and all subsequent formal statements, we can replace $\Omega$ by $K\Omega K$. Hence we shall assume throughout the paper that $\Omega$ is bi-$K$-invariant.
\end{remark}

The exponent saving $\delta>0$ that our proof produces is very small in this broad generality. Let us therefore focus on the special case $n=3$, $r=s=1$ in the $\dim\tau$ aspect, i.e. assuming that $\mu'\in I$ for a bounded domain $I$. For instance, $\GL_3$ representations associated with the symmetric square of all holomorphic cusp forms of large weight 
$k$ belong to this family, with $\ell=2k-1$ and $\mu'=0$.
In this situation, the right-hand side of \eqref{triv} is $\asymp\ell^2$ (see \eqref{eq:Delta-KS} and \eqref{eq:dimformula_1} below), and our second main result provides a strong power-saving over that.

\begin{theorem}\label{thm2} In the situation described above, we have
$$\| \Phi|_{\Omega} \|_{\infty} \ll_{\eps,\Omega,I} \ell^{2 - \frac{1}{14} + \eps}$$
for all $\eps > 0$. 
\end{theorem}

We highlight an important methodological novelty in the proof of Theorem~\ref{thm2}. As is known, matrix counting constitutes the arithmetic heart of the usual treatments of the sup-norm problem. While the determinant condition for $\SL_2$ is so simple that the counting can be done in an elementary fashion, experience has shown that adequate higher rank matrix counting by arithmetic methods is often extremely difficult. In the proof of Theorem~\ref{thm2}, we take an entirely different approach and apply the ($K$-trivial!) pretrace formula backwards to analyze the counting problem by analytic rather than arithmetic means. This step is not involutory because, on the one hand, we are dealing with two different types of pretrace formula (with and without $K$-types) and, on the other hand, we take absolute values in between. Since the analytic expression involves another (averaged) sup-norm problem, we are led to apply the pretrace formula a third time. A prototype of applying spectral summation formulae in various directions (with Cauchy--Schwarz in between) is the treatment of the $\GL_2$ spectral large sieve by Deshouillers--Iwaniec~\cite{DeIw}. The proof of Theorem~\ref{thm2} implements this in higher rank.

\subsection{Generalized spherical functions}\label{14}
The analytic core of the usual approaches to the sup-norm problem is an asymptotic analysis of spherical functions. This is a classical problem of independent interest when dealing with analytic aspects of Lie groups. The classical bi-$K$-invariant spherical function with spectral parameter $\mu \in \mfa_{\CC}^{\ast}$ is given by
\begin{equation}\label{spher}
\phi_{\mu}(x) = \int_K e^{(\mu-\rho)(H(xk))}\,dk,\qquad x\in G,
\end{equation}
where $\rho\in\mathfrak a^*$ is the half-sum of the positive roots, and $dk$ is the Haar probability measure on $K$. One would expect that for large $\mu$ the integrand is highly oscillatory and hence the integral is small, but of course $\phi_{\mu}(1) = 1$, regardless of the value of $\mu$. Uniform decay properties that also hold for $\mu$ near or on the Weyl chamber walls were obtained in \cite{BP}:
\begin{equation}\label{BPbound}
\phi_{\mu}(k_1 ak_2) \ll (1 + \| \mu \| \cdot \| \log a \|)^{-1/2}
\end{equation}
as long as $\mu$ has bounded real part and $\log a \in \mfa$ remains bounded. The key input in the proof is that the phase $k \mapsto H(ak)$ for fixed $a \in A$ has no degenerate stationary points \cite[Cor.~1.5]{DKV}, so that a second derivative test in one suitable direction gives the desired power saving with exponent $1/2$. Starting with the work of Harish-Chandra~\cite{HC2,HC4}, bounds and asymptotics for spherical functions have a long history, with important contributions by Duistermaat--Kolk--Varadarajan~\cite{DKV} (see also the classical book \cite{MR954385}), and we refer to \cite{MR1870604, MR1873133, MR3144230, MR3506604, MR4297181, MR4931487, brumley2026quantumergodicitybenjaminischrammlimit} for a list of some newer developments.

In this paper, we consider for the first time more general spherical functions that are relevant for a trace formula with non-trivial $K$-types. We proceed to describe the setup. Let $(\tau,V_\tau)$ be an irreducible unitary representation of $K$, and let $(U,V_U)$ be a Hilbert space representation of $G$ whose restriction to $K$ is unitary. Assume that $q:=[U:\tau]$ is finite, and let $(P_1,\dotsc,P_q)$ be an orthonormal basis of $\Hom_K(V_U,V_\tau)$ with respect to the scalar product $\langle P,Q\rangle=(1/d_\tau)\tr(P^*Q)$. We consider the $\End(V_\tau)$-valued function given by
\begin{equation}\label{eq:phidef}
\varphi_\tau^U(x):=\sum_{u=1}^q P_u U(x)P_u^*,\qquad x\in G.
\end{equation}
This function is independent of the basis $(P_1,\dotsc,P_q)$. We shall focus on the situation when $U = U^{\sigma', \mu'}$ and $\tau$ is the minimal $K$-type occurring in $U|_K$ (in the sense of \S\ref{sec:generalized-spherical-trace-function-SL}). Camporesi~\cite{C} gives an explicit formula for this function reminiscent of \eqref{spher} that we recall in \eqref{Eis}--\eqref{Fdef2} below, and we also provide a self-contained derivation of this formula. With notation to be explained in \S\ref{spherical}, it is of the shape
\begin{equation}\label{shape}
\varphi_\tau^{U}(x)=\int_K F_\tau^{\sigma'}(xk)\tau(k^{-1})\,e^{(\mu'-\rho')(H'(xk))}\,dk.
\end{equation}
We are interested in the trace of this function,
\[
\psi^{U}_{\tau} := \tr(\varphi_\tau^U)=\tr(P^U_\tau UP^U_\tau),
\]
where $P^U_\tau$ is the projection of $V_U$ onto its $\tau$-isotypic subspace. The two expressions on the right-hand side are equal by \eqref{eq:psi-expanded-in-P-basis} below. Note that, unlike the classical spherical function \eqref{spher}, the function $\psi^{U}_{\tau}$ is not bi-$K$-invariant, but only $K$-conjugation invariant. In particular, it does not suffice to analyze it on diagonal matrices, but we need to understand it on general block upper triangular matrices with diagonal blocks of size at most $2$. If $U$ is unitary, then by $q=[U:\tau]=1$
(which is proved in \S\ref{sec:generalized-spherical-trace-function-SL}), we have the trivial bound
$|\psi^{U}_{\tau}(x)| \leq \dim \tau$ for all $x\in G$. In general, $U$ is not unitary, but we can still infer from Camporesi's formula the analogous baseline bound\footnote{For a proof, combine \eqref{tracefunctionbound} and \eqref{hatFtausigma2} with the fact that
the $\|\cdot\|_{2,1}$ norm is multiplicative under tensor products. A generalization for arbitrary $K$-types is provided by \eqref{eq:psi-baseline-bound}.}
\begin{equation}\label{eq:psi-trivial-bound}
\psi^{U}_{\tau}(x) \ll_{\Re\mu',\Omega} \dim \tau,\qquad x\in\Omega.
\end{equation}
Again, we expect decay as soon as $x$ ``moves away'' (e.g., from $K$). The results in this section make this quantitative. To this end, we fix a convenient, bi-$K$-invariant distance function on $G$:
\[\dist(x,x'):=\|x-x'\|,\qquad x,x'\in G,\]
where $\|\cdot\|$ is the Frobenius norm on $\RR^{n\times n}$. We also recall the notation $\|U\|_\infty$ introduced in \eqref{eq:unorm}. In analogy with Theorems~\ref{thm1} and~\ref{thm2}, we first establish a bound in full generality, then we move to the special case of $n=3$ and $r=s=1$.

\begin{theorem}\label{thm3} Let $U=U^{\sigma',\mu'}$ be a generalized principal series representation with (unique) minimal $K$-type $\tau$, and let $\Omega\subseteq G$ be a compact set. Then, for any $x\in\Omega$, we have
\[\psi^{U}_{\tau}(x)\ll_{\Re\mu',\Omega}\dim \tau\cdot\left(1 +\|U\|_{\infty}^{3/17}\dist(x, K)\right)^{-1/3}.\]
The implied constant depends continuously on $\Re\mu'$.
\end{theorem}

This should be seen as a generalization of \eqref{BPbound} to minimal $K$-types. While the exact numerical exponents have little significance, the key feature is complete uniformity in all archimedean data and distance from $K$. In fact for the proof of Theorem~\ref{thm1} we need a more general version of Theorem~\ref{thm3} where $\tau$ can be an arbitrary $K$-type sufficiently close to the minimal $K$-type. This is Theorem~\ref{thm3b} in \S\ref{thm3bsubsection}.

The theorem will be proved in two steps. First we apply absolute values in \eqref{shape} and investigate when the (trace of the) integrand is not negligibly small. The function $F_\tau^{\sigma'}$ inherits properties of the $\SL_2$-Bergman kernel 
$$\left(\begin{matrix} a & b\\ c & d\end{matrix}\right) \mapsto \left( \frac{2}{a + i b - i c + d}\right)^{\ell}$$
which for large $\ell$ localizes quite sharply at orthogonal matrices. The idea is to show that if $x$ is away from $K$ and $\ell$ is large, then either $\| F_\tau^{\sigma'}(xk) \|$ for a certain $k\in K$ is already negligibly small, or there is a direction in which it becomes negligibly small. Unlike the case of classical spherical functions where a second derivative test suffices, this requires an analysis up to the third derivative. In this way we obtain a bound that saves a factor 
$\| \bm \ell\|_{\infty}^{-1/6} \dist(x, K)^{-1/3}$ uniformly in $\mu'$. This yields the result
unless $\| \bm \ell\|_{\infty} \leq \| \mu' \|_{\infty}^{6/17}$, in which case we apply a stationary phase argument in \eqref{shape} using a second derivative test.

In the more special case, we have the following.
\begin{theorem}\label{thm4} Let $n=3$ and $r=s=1$, so that $G=\SL_3^\pm(\RR)$ and $K=\OO(3)$. Let $U$, $\tau$, $\Omega$ be as in Theorem~\ref{thm3}, so that
\[
A'=\left\{\diag(y,y,y^{-2}):y\in\RR_{>0}\right\}.
\]
Consider the following two subsets of $G$ invariant under $K$-conjugation:
\begin{align*}
R &:= \left\{ k g k^{-1} : g\in \pm A' \begin{pmatrix} \SO(2) & \\ & 1 \end{pmatrix},\ k\in K \right\}\\
\intertext{and}
S &:= \left\{ k g k^{-1} : g\in \pm A' \begin{pmatrix} \SO(2) & \ast \\& 1 \end{pmatrix},\ k\in K \right\}.
\end{align*}
Then, for all $\eps,A>0$ and $x\in\Omega$, we have
\[\psi^{U}_{\tau} (x)\ll_{\eps, A,\Re\mu',\Omega}
\frac{\ell^{1+\eps}}{1 + \ell \cdot d_{\{\pm\id\}}^2+\ell^{3/2}\cdot c_{\mu'}d_{\{\pm\id\}}^3d_R + (\sqrt{\ell} \cdot d_S)^A}, \]
where $d_{\{\pm\id\}}:=\dist(x,\{\pm\id\})$, $d_R:=\dist(x,R)$, and $d_S:=\dist(x,S)$. Here, the implied constant depends continuously on $\Re\mu'$, and $c_{\mu'}:=\min(1,\sqrt{\ell}/(d_{\{\pm\id\}}\|\Im\mu'\|))$.
\end{theorem}

We note that the bound is uniform in $\Im\mu'$ apart from the factor $c_{\mu'}$ when $\|\Im\mu'\|\geq\sqrt{\ell}/d_{\{\pm\id\}}$, though it also does not exhibit any savings as $\Im\mu'$ gets large. For large $\ell$, Theorem~\ref{thm4} shows that $\psi_\tau^U$ can be of full size $\ell^{1+o(1)}$ only very close to $\{\pm \text{id}\}$, and we can observe decay away from special manifolds, which (along with $K$ to which Theorem~\ref{thm3} localizes) we organize according to their dimensions as follows:
\begin{alignat*}{8}
\{\pm\id\}&\subseteq K&&\subseteq R&&\subseteq S&&\subseteq G\\
\text{dimension:}\qquad 0&<3&&<4&&<6&&<8.
\end{alignat*}
Thus, suppressing the dependencies on $\eps$, $\Re\mu'$, $\Omega$ in the implied constants, and working in the range $\|\Im\mu'\|\leq\sqrt{\ell}$ for simplicity:
\begin{itemize}
\item On the two $1/\sqrt{\ell}$-balls around $\{\pm\id\}$, we have the full-sized bound 
$$\psi_{\tau}^U(x)\ll\ell^{1+\eps}.$$
\item Beyond this, within the $4$-dimensional $R$ (which contains the $3$-dimensional $K$), the bound improves to 
$$\psi_{\tau}^U(x)\ll \frac{\ell^{1+\eps}}{1+\ell\cdot d_{\{\pm \id\}}^2}.$$
This is already much stronger than $\psi_{\tau}^U(x)\ll\ell^{1+\eps}/(1+\ell\cdot\dist(x,K)^2)$, and in fact $\psi_{\tau}^U(x)\ll\ell^\eps$ holds for generic points $x\in R$ for which $d_{\{\pm\id\}}\asymp 1$.
\item Beyond this, within the $6$-dimensional $S$, the bound further improves to
$$\psi_{\tau}^U(x)\ll\frac{\ell^{1+\eps}}{1+\ell\cdot d_{\{\pm\id\}}^2(1+\sqrt{\ell}d_{\{\pm\id\}}d_R)},$$
so that in fact $\psi_{\tau}^U(x)\ll\ell^{-1/2+\eps}$ holds for generic points $x\in S$ for which $d_R\asymp 1$.
\item Beyond this, we have rapid decay outside the $\ell^{-1/2+\eps}$-neighborhood of $S$.
\end{itemize}

As with Theorem~\ref{thm1}, for the proof of Theorem~\ref{thm2} we need a more general version of Theorem~\ref{thm4} where $\tau$ can be an arbitrary $K$-type sufficiently close to the minimal $K$-type. This is Theorem~\ref{thm4b} in \S\ref{n3-nonminimal-subsec}.

\subsection{Plan of the paper and notations}\label{subsec:notations}
Sections~\ref{spherical} to \ref{section_for_n_equals3} are devoted to the analysis of generalized spherical functions. This includes in particular the proofs of Theorems~\ref{thm3} to~\ref{thm4b}. Section~\ref{PWsection} features in Theorem~\ref{Thm5} a Paley--Wiener theorem that generalizes the classical result of Delorme and Flensted-Jensen~\cite[Th.~2]{DFJ} to more general (e.g., non-connected) groups and more general (e.g., non-smooth) test functions, coupled with a spectral localizer provided by Theorems~\ref{theorem:K-type-barrier} and~\ref{theorem:concrete-paley-wiener-function}. Sections~\ref{sec6} and \ref{counting} return to the sup-norm problem and prove Theorems~\ref{thm1} and~\ref{thm2}.

In accordance with Remark~\ref{rem2}, we shall assume throughout the paper that $\Omega\subseteq G$ is a compact set satisfying $K\Omega K=\Omega$. As already used above, $A\ll_D B$ (or, interchangeably, $A=O_D(B)$) means that for some constant $C>0$ depending only on $D$, the bound $|A|\leq CB$ holds. We write $A\asymp B$ if $A\ll B$ and $B\ll A$ simultaneously hold. All implied constants in $\ll$ or $O(\dotsc)$ may depend on $n$ and $\Omega$ without being displayed in the notation. Moreover, whenever we indicate the dependence of an implied constant on $\Re\mu'$, we always understand it as a continuous dependence.

For most of this paper, we set $G:=\SL_n^{\pm}(\RR)$, $\Gamma:=\SL_n^{\pm}(\ZZ)=\GL_n(\ZZ)$, $K=\OO(n)$. Some parts of the theoretical background, however, hold among more general circumstances, such as for Lie groups in the Harish-Chandra class. These exceptional occasions will be indicated clearly, and then we will write $G$ for the more general underlying group and $K$ for its maximal compact subgroup. 

Given irreducible unitary representations $\tau$, $\sigma$, $\omega$ of some compact groups, we occasionally write $d_{\tau}:=\dim\tau$, $d_{\sigma}:=\dim\sigma$, $d_{\omega}:=\dim \omega$.

A general convention is that when we integrate over a group, then it is meant with respect to a fixed Haar measure. In case of compact groups, we always mean the Haar probability measure.

\section{Background on generalized spherical functions}\label{spherical}
In this section we provide the necessary background to define generalized spherical functions and state Camporesi's formula, given in \eqref{Eis} below. Most relevant for us will be the trace of this endomorphism-valued function, discussed in \S\ref{sec27}. For the convenience of the reader, we provide a self-contained proof of \eqref{Eis}. The scope here is more general, so for this section, we allow $G$ to be any Lie group in the Harish-Chandra class (see \cite[\S3]{HC1}, \cite[\S1]{CM}, \cite[Ch.~4, \S3]{KV}). As emphasized by \cite[\S1]{CM}, connected semisimple Lie groups with finite center and groups of $\RR$-points of Zariski-connected reductive algebraic groups defined over $\RR$ belong to the Harish-Chandra class. Moreover, if $G$ belongs to this class, then so do Levi components of parabolic subgroups of $G$. In particular, $\SL_n^\pm(\RR)$ belongs to the Harish-Chandra class, because it is a direct factor of the Levi component of a (maximal) parabolic subgroup of $\SL_{n+1}(\RR)$.

\subsection{Generalized principal series for the minimal parabolic}\label{Pseries}
Let $G=KAN$ be an Iwasawa decomposition. Thus $K\subseteq G$ is a maximal compact subgroup, and $P=MAN$ is a minimal parabolic subgroup, with $M=K\cap P=Z_K(A)$. We shall use without further remark that $MA=Z_G(A)$ normalizes $N$. Let $\mfa$ be the Lie algebra of $A$, and let $\rho\in\mfa^*$ be the half-sum of the positive restricted roots, counted with multiplicity. We shall write the Iwasawa decomposition of group elements as
\begin{equation}\label{iw1}
x=\kappa(x)e^{H(x)}n(x),\qquad x\in G,
\end{equation}
where $\kappa(x)\in K$, $H(x)\in\mfa$, $n(x)\in N$ are uniquely determined. For future reference, we note the identities
\begin{equation}\label{iw5}
\kappa(x^{-1}\kappa(xk))=k,\qquad H(x^{-1}\kappa(xk))=-H(xk),
\end{equation}
which follow from the implication
\[xk\in\kappa(xk)e^{H(xk)}N\qquad\Longrightarrow\qquad x^{-1}\kappa(xk)\in kN e^{-H(xk)}=k e^{-H(xk)}N.\]
A (generalized) principal $P$-series representation of $G$ is of the form\footnote{Also called an \emph{elementary} representation of $G$. Our notation harmonizes with \cite[p.~168]{K} and \cite[p.~660]{KV}. Hence our $U^{\sigma,\mu}$ is $U^{\sigma,\mu-\rho}$ in the notation of \cite[p.~449]{W1} and \cite[p.~271]{C}.}
\[U^{\sigma,\mu}=\Ind_P^G(\sigma\otimes e^{\mu}\otimes\mathbf{1}),\]
where $(\sigma,V_\sigma)$ is an irreducible unitary representation of $M$ and $\mu:\mfa\to\CC$ is a linear functional. We shall think of this as a representation on $L^2(K,\sigma)$, the Hilbert space of square-integrable functions $f:K\to V_{\sigma}$ satisfying
\[f(km)=\sigma(m^{-1})f(k),\qquad k\in K,\quad m\in M.\]
The representation is given by
\[(U^{\sigma,\mu}(x)f)(k):=e^{-(\mu+\rho)(H(x^{-1}k))}f(\kappa(x^{-1}k)),
\qquad x\in G,\ \ f\in L^2(K,\sigma),\ \ k\in K.\]
Restricting this to $K$ yields the left regular representation of $K$ on $L^2(K,\sigma)$, therefore
\begin{equation}\label{indres1}
U^{\sigma,\mu}|_K=\Ind_M^K\sigma.
\end{equation}
The representation $U^{\sigma,\mu}$ embeds unitarily into the representation $U^\mu$ on $L^2(K)$ given by\footnote{Also called a \emph{standard} representation of $G$. Our $U^\mu$ is $U^{\mu-\rho}$ in the notation of Warner~\cite[p.~445]{W1}.}
\[(U^\mu(x)f)(k):=e^{-(\mu+\rho)(H(x^{-1}k))}f(\kappa(x^{-1}k)),\qquad x\in G,\ \ f\in L^2(K),\ \ k\in K.\]
Indeed, as explained\footnote{Warner~\cite[p.~447]{W1} assumes that $G$ is a connected semisimple Lie group, but the argument works equally well for $G$ in the Harish-Chandra class.} in \cite[p.~447]{W1}, the right regular representation of $M$ on $L^2(K)$ commutes with $U^\mu$, and the $\check{\sigma}$-isotypic subspace of $L^2(K)$ is a sum of $d_\sigma$ representations unitarily equivalent to $U^{\sigma,\mu}$.
We note for future reference that the Hilbert space adjoint of $U^\mu(x)$ is $U^{-\bar\mu}(x^{-1})$. 
This can be deduced from \eqref{iw5} and Harish-Chandra's formula
\[\int_K f(k)\,dk=\int_K f(\kappa(xk))\,e^{-2\rho(H(xk))}\,dk,\]
valid for all $f\in L^2(K)$ and $x\in G$. Indeed, using these relations, we obtain for arbitrary $f,g\in L^2(K)$ and $x\in G$ that
\begin{align*}
\langle U^\mu(x)f,g\rangle
&=\int_K e^{-(\mu+\rho)(H(x^{-1}k))}f(\kappa(x^{-1}k))\,\ov{g(k)}\,dk\\
&=\int_K e^{-(\mu+\rho)(H(x^{-1}\kappa(xk)))}f(\kappa(x^{-1}\kappa(xk)))\,\ov{g(\kappa(xk))}\,e^{-2\rho(H(xk))}\,dk\\
&=\int_K e^{(\mu-\rho)(H(xk))}f(k)\,\ov{g(\kappa(xk))}\,dk=\langle f,U^{-\bar\mu}(x^{-1})g\rangle.
\end{align*}

\subsection{Spherical functions}
Let $(\tau,V_\tau)$ be an irreducible unitary representation of $K$, and let $(\bv_1,\dotsc,\bv_{d_\tau})$ be an orthonormal basis of $V_\tau$. Let $(U,V_U)$ be a Hilbert space representation of $G$ whose restriction to $K$ is unitary. Assume that $q:=[U:\tau]$ is finite, and let $(P_1,\dotsc,P_q)$ be an orthonormal basis of $\Hom_K(V_U,V_\tau)$ with respect to the scalar product $\langle P,Q\rangle=(1/d_\tau)\tr(P^*Q)$. Then the vectors $P_u^*\bv_r\in V_U$ form an orthonormal basis of the $\tau$-isotypic subspace of $V_U$. We consider the $\End(V_\tau)$-valued function given by
\[\varphi_\tau^U(x):=\sum_{u=1}^{q} P_u U(x)P_u^*,\qquad x\in G.\]
This function is independent of the orthonormal basis $(P_1,\dotsc,P_q)$. A key feature is that
\begin{equation}\label{eq:psi-expanded-in-P-basis}
\psi_\tau^U(x) := \tr(\varphi_\tau^U(x))
=\sum_{r=1}^{d_\tau}\sum_{u=1}^q\langle P_uU(x)P_u^*\bv_r,\bv_r\rangle
=\sum_{r=1}^{d_\tau}\sum_{u=1}^q\langle U(x)P_u^*\bv_r,P_u^*\bv_r\rangle.
\end{equation}
Another useful property is that
\begin{equation}\label{varphisymmetry}
\varphi_{\tau}^U(k_1xk_2)=\tau(k_1)\varphi_{\tau}^U(x)\tau(k_2),
\qquad k_1\in K,\quad x\in G,\quad k_2\in K.
\end{equation}
We shall prove below that if $U$ is infinitesimally equivalent to a subrepresentation of $U^{\sigma,\mu}$, then there is an explicit projection $P_{\tau,\sigma}^U\in\End_M(V_\tau)$ such that for all $x\in G$,
\begin{equation}\label{varphi1}
\varphi_\tau^U(x)=\frac{d_\tau}{d_\sigma}\int_K\tau(\kappa(xk))P_{\tau,\sigma}^U\tau(k^{-1})\,e^{(\mu-\rho)(H(xk))}\,dk.
\end{equation}
For describing $P_{\tau,\sigma}^U$, let us note that Frobenius reciprocity yields (cf. \eqref{indres1}) that
\begin{equation}\label{frobenius}
\Hom_M(V_\tau,V_\sigma)\cong\Hom_K(V_\tau,L^2(K,\sigma)).
\end{equation}
The isomorphism is given by the following canonical maps $T\leftrightarrow\widetilde T$:
\[\begin{alignedat}{2}
\widetilde T:&\bv\mapsto(k\mapsto T\tau(k^{-1})\bv),&\qquad &\bv\in V_\tau,\quad k\in K,\\
T:&\bv\mapsto(\widetilde T\bv)(1),&\qquad&\bv\in V_\tau.
\end{alignedat}\]
Now we observe that $\Hom_K(V_\tau,V_U)$ is a subspace of $\Hom_K(V_\tau,L^2(K,\sigma))$. We regard $\Hom_M(V_\tau,V_\sigma)$ as a Hilbert space equipped with the scalar product $\langle S,T\rangle=(1/d_\sigma)\tr(S^*T)$, and we take an orthonormal basis $(T_1,\dotsc,T_q)$ of the image of $\Hom_K(V_\tau,V_U)$ in $\Hom_M(V_\tau,\allowbreak V_\sigma)$ under Frobenius reciprocity. Then the sought projection is given by
\begin{equation}\label{varphi2}
P_{\tau,\sigma}^U:=\sum_{u=1}^q T_u^* T_u\in\End_M(V_\tau).
\end{equation}
This projection is independent of the orthonormal basis $(T_1,\dotsc,T_q)$.

\subsection{Proof of \texorpdfstring{\eqref{varphi1}}{(2.6)} and \texorpdfstring{\eqref{varphi2}}{(2.8)}}
We shall prove \eqref{varphi1} and \eqref{varphi2} by embedding $d_\sigma$ copies of $U$ into $U^\mu$, and performing the calculation with the help of the matrix coefficients
\[a_{ij}(k):=\langle\tau(k^{-1})\bv_i,\bv_j\rangle,\qquad k\in K.\]
These functions form an orthogonal basis of the $\tau$-isotypic subspace of $L^2(K)$ under the left regular representation of $K$, and the entries of each column of $(a_{ij})$ generate a copy of $\tau$ under this action:
\begin{equation*}
L(k)a_{ij}=\sum_t a_{it}(k^{-1})a_{tj},\qquad k\in K.
\end{equation*}
The functions $a_{ij}\in L^2(K)$ also form an orthogonal basis of the $\check\tau$-isotypic subspace of $L^2(K)$ under the right regular representation of $K$, and the entries of each row of $(a_{ij})$ generate a copy of $\check\tau$ under this action:
\begin{equation*}
R(k)a_{ij}=\sum_t a_{tj}(k)a_{it},\qquad k\in K.
\end{equation*}
We can assume that the matrix-valued function $(a_{ij}(m^{-1}))$ $(m\in M)$ has a diagonal block decomposition, each block being the matrix of an irreducible unitary representation of $M$. We can further assume that the first block has entries $\langle\sigma(m)\bv_i,\bv_j\rangle$ for $1\leq i,j\leq d_\sigma$, and it repeats $[\tau:\sigma]$ times. Following \cite[p.~447]{W1}, we consider the following intertwiners of $U^\mu$:
\[R_{ij}:=d_\sigma\int_M a_{ij}(m^{-1})R(m)\,dm,\qquad 1\leq i,j\leq d_\sigma.\]
Then $U^\mu$ restricted to $R_{ii}L^2(K)$ is the $i$-th copy of $U^{\sigma,\mu}$ in $U^\mu$, and $R_{ij}$ maps $R_{ii}L^2(K)$ isometrically onto $R_{jj}L^2(K)$. Let us calculate the image of the matrix coefficient $a_{rs}$ under this intertwiner:
\begin{align*}
R_{ij}a_{rs}(k)&=d_\sigma\int_M a_{ij}(m^{-1})a_{rs}(km)\,dm\\
&=d_\sigma\int_M a_{ij}(m^{-1})\sum_t a_{rt}(k)a_{ts}(m)\,dm\\
&=d_\sigma\sum_t a_{rt}(k)\int_M\ov{a_{ji}(m)}a_{ts}(m)\,dm.
\end{align*}
By Schur orthogonality, the inner integral is $d_\sigma^{-1}$ when $(s,t)$ equals $(i,j)$ shifted by $(u-1)d_\sigma$ for some $1\leq u\leq [\tau:\sigma]$; otherwise, it is zero. Hence $R_{ij}$ maps $a_{r,i+(u-1)d_\sigma}$ to $a_{r,j+(u-1)d_\sigma}$, and it annihilates all other matrix coefficients $a_{rs}$. In particular, the left $\tau$-isotypic subspace of $R_{ii}L^2(K)$ is spanned by $a_{r,i+(u-1)d_\sigma}$ with $1\leq r\leq d_\tau$ and $1\leq u\leq [\tau:\sigma]$, and $R_{ij}$ maps this space isometrically onto the left $\tau$-isotypic subspace of $R_{jj}L^2(K)$. It follows that the $(\tau,\check\sigma)$-isotypic subspace of $L^2(K)$ is spanned by the matrix coefficients $a_{rs}$ with $1\leq r\leq d_\tau$ and $1\leq s\leq [\tau:\sigma]d_\sigma$.

Recall the notation $q=[U:\tau]$. By choosing the orthonormal basis $(\bv_1,\dotsc,\bv_{d_\tau})$ appropriately, we may assume that $P_{\tau,\sigma}^U$ considered in \eqref{varphi2} is the projection of $V_\tau$ onto the span of $(\bv_1,\dotsc,\bv_{qd_\sigma})$. Indeed, let us start from an orthonormal basis $(\bw_1,\dotsc,\bw_{d_\sigma})$ of $V_\sigma$, and extend the system $(T_1,\dotsc,T_q)$ used in \eqref{varphi2} to an orthonormal basis $(T_1,\dotsc,T_{[\tau:\sigma]})$ of $\Hom_M(V_\tau,V_\sigma)$. Then, we can take $\bv_{i+(u-1)d_\sigma}:=T_u^*\bw_i$ for $1\leq i\leq d_\sigma$ and $1\leq u\leq [\tau:\sigma]$, and extend this system to an orthonormal basis $(\bv_1,\dotsc,\bv_{d_\tau})$ of $V_\tau$ compatible with a decomposition of $V_\tau$ into irreducible unitary representations of $M$. Now, for any $1\leq i\leq d_\sigma$, the $\tau$-isotypic component of the $i$-th copy of $U$ in $U^\mu$ is spanned by $a_{r,i+(u-1)d_\sigma}$ with $1\leq r\leq d_\tau$ and $1\leq u\leq q$. Hence we can verify \eqref{varphi1} by checking that the two sides have the same matrix coefficients:
\begin{align*}
\langle\varphi_\tau^U(x)\bv_i,\bv_j\rangle
&=\frac{d_\tau}{d_\sigma}\sum_{s=1}^{qd_\sigma}\langle U^{\mu}(x)a_{is},a_{js}\rangle\\
&=\frac{d_\tau}{d_\sigma}\sum_{s=1}^{qd_\sigma}\langle a_{is},U^{-\ov{\mu}}(x^{-1})a_{js}\rangle\\
&=\frac{d_\tau}{d_\sigma}\sum_{s=1}^{qd_\sigma}\int_K a_{is}(k)\,\ov{a_{js}(\kappa(xk))}\,e^{(\mu-\rho)(H(xk))}\,dk\\
&=\frac{d_\tau}{d_\sigma}\sum_{s=1}^{qd_\sigma}\int_K 
\langle\tau(k^{-1})\bv_i,\bv_s\rangle\,\langle\bv_s,\tau(\kappa(xk)^{-1})\bv_j\rangle\,e^{(\mu-\rho)(H(xk))}\,dk\\
&=\frac{d_\tau}{d_\sigma}\sum_{s=1}^{d_\tau}\int_K 
\langle P_{\tau,\sigma}^U\tau(k^{-1})\bv_i,\bv_s\rangle\,\langle\bv_s,\tau(\kappa(xk)^{-1})\bv_j\rangle\,e^{(\mu-\rho)(H(xk))}\,dk\\
&=\frac{d_\tau}{d_\sigma}\int_K 
\langle P_{\tau,\sigma}^U\tau(k^{-1})\bv_i,\tau(\kappa(xk)^{-1})\bv_j\rangle\,e^{(\mu-\rho)(H(xk))}\,dk\\
&=\frac{d_\tau}{d_\sigma}\int_K 
\langle \tau(\kappa(xk))P_{\tau,\sigma}^U\tau(k^{-1})\bv_i,\bv_j\rangle\,e^{(\mu-\rho)(H(xk))}\,dk.
\end{align*}

\subsection{Nonminimal parabolic subgroups} 
We introduce nonminimal parabolic subgroups and the corresponding (non-unique) Iwasawa decompositions. Let $P'=M'A'N'$ be a parabolic subgroup containing $P=MAN$, given in Langlands decomposition. Hence $M'\supseteq M$, $A'\subseteq A$, and $N'\subseteq N$. Then $K'=K\cap M'=Z_K(A')$ is a maximal compact subgroup of $M'$, and we have an Iwasawa decomposition $M'=K'A_1N_1$, where $A=A'A_1$ and $N=N'N_1$. Correspondingly, the Lie algebra of $A$ decomposes as a direct sum $\mfa=\mfa'+\mfa_1$. We denote by $\rho'\in\mfa'^*$ and
$\rho_1\in\mfa_1^*$ the restrictions of $\rho$ to $\mfa'$ and $\mfa_1$, respectively. We shall use without further remark that $M'A'=Z_G(A')$ normalizes $N'$. For the elements of $M'$, we have the Iwasawa decomposition
\begin{equation}\label{iw2}
m'=\kappa_1(m')e^{H_1(m')}n_1(m'),\qquad m'\in M',
\end{equation}
where $\kappa_1(m')\in K'$, $H_1(m')\in\mfa_1$, $n_1(m')\in N_1$ are uniquely determined. By $G=KM'A'N'$, we can also write
\begin{equation}\label{iw3}
x=\kappa'(x)m'(x)e^{H'(x)}n'(x),\qquad x\in G,
\end{equation}
with $\kappa'(x)\in K$, $m'(x)\in M'$, $H'(x)\in\mfa'$, $n'(x)\in N'$. Here $H'(x)$ and $n'(x)$ are uniquely determined, but $\kappa'(x)$ and $m'(x)$ are ambiguous. If we use a given pair $(\kappa'(x),m'(x))$, then for any $k'\in K'$, we can equally use the pair $(\kappa'(x)k'^{-1},k'm'(x))$. The decompositions \eqref{iw1}, \eqref{iw2}, \eqref{iw3} are not independent of each other, namely for all $x\in G$ and $k'\in K'$, we have
\begin{equation}\label{endgame}
\kappa'(x)\kappa_1(m'(x)k')=\kappa(xk'),\qquad H_1(m'(x)k')=H_1(m'(xk')),\qquad H'(x)=H'(xk').
\end{equation}
Note that these relations are meaningful, despite the ambiguity of $\kappa'(x)$ and $m'(x)$. In order to prove \eqref{endgame}, we start from \eqref{iw3}:
\[xk'\in\kappa'(x)m'(x)e^{H'(x)}N'k'=\kappa'(x)m'(x)k'e^{H'(x)}N'.\]
Then, we apply \eqref{iw2} for $m'=m'(x)k'$:
\[xk'\in\kappa'(x)\kappa_1(m'(x)k')e^{H_1(m'(x)k')+H'(x)}N.\]
Finally, a comparison with \eqref{iw1} reveals that
\begin{equation}\label{iw4}
\kappa'(x)\kappa_1(m'(x)k')=\kappa(xk'),\qquad H_1(m'(x)k')+H'(x)=H(xk').
\end{equation}
The first relation in \eqref{iw4} is the first relation in \eqref{endgame}. Using the second relation in \eqref{iw4}, we obtain by the replacement $(x,k')\to (xk',1)$ that
\[H_1(m'(x)k')+H'(x)=H_1(m'(xk'))+H'(xk').\]
Since $\mfa=\mfa_1+\mfa'$ is a direct sum decomposition, the second and third relations in \eqref{endgame} follow.

\subsection{Generalized principal series for nonminimal parabolics}\label{P'series} 
A (generalized) principal $P'$-series representation of $G$ is of the form\footnote{Camporesi~\cite{C} uses this with $\mu'=i\nu'$, where $\nu'\in\mfa'^*$. On the other hand, Camporesi~\cite{C} denotes our $U^{\sigma',i\nu'}$ by $U^{\sigma',\nu'}$.}
\[U^{\sigma',\mu'}=\Ind_{P'}^G(\sigma'\otimes e^{\mu'}\otimes\mathbf{1}),\]
where $(\sigma',V_{\sigma'})$ is in the discrete series of $M'$ and $\mu':\mfa'\to\CC$ is a linear functional. We shall think of this as a representation on $L^2(K,\sigma')$, the Hilbert space of square-integrable functions $f:K\to V_{\sigma'}$ satisfying
\[f(kk')=\sigma'(k'^{-1})f(k),\qquad k\in K,\quad k'\in K'.\]
The representation is given by
\[
(U^{\sigma',\mu'}(x)f)(k):=e^{-(\mu'+\rho')(H'(x^{-1}k))}\sigma'(m'(x^{-1}k))^{-1}f(\kappa'(x^{-1}k))
\]
for $x\in G$, $f\in L^2(K,\sigma')$, and $k\in K$. Restricting this to $K$ yields the left regular representation of $K$ on $L^2(K,\sigma')$, therefore
\begin{equation}\label{indres}
U^{\sigma',\mu'}|_K=\Ind_{K'}^K\sigma'|_{K'}.
\end{equation}

By the subrepresentation theorem \cite[Th.~8.21]{CM}, there exists $\sigma\in\widehat{M}$ and a linear functional $\mu_1:\mfa_1\to\CC$ such that the discrete series representation $\sigma'$ of $M'$ is infinitesimally equivalent with a subrepresentation of $\Ind_{MA_1N_1}^{M'}(\sigma\otimes e^{\mu_1}\otimes\mathbf{1})$. Therefore, writing $\mu=\mu_1+\mu'$ in the obvious sense, the $P'$-principal series representation $U^{\sigma',\mu'}$ is infinitesimally equivalent with a subrepresentation of the $P$-principal series representation $U^{\sigma,\mu}$. Practically speaking, $V_{\sigma'}$ embeds into $L^2(K',\sigma)$ in an $M'$-equivariant fashion (up to infinitesimal equivalence), while $L^2(K,\sigma')$ embeds into $L^2(K,\sigma)$ in a $G$-equivariant fashion (up to infinitesimal equivalence). This induces canonical embeddings (cf. \eqref{frobenius})
\[\begin{alignedat}{3}
&\Hom_{K'}(V_\tau,V_{\sigma'})&&\hookrightarrow\Hom_{K'}(V_\tau,L^2(K',\sigma))\ &\cong\Hom_M(V_\tau,V_\sigma),\\
&\Hom_K(V_\tau,L^2(K,\sigma'))&&\hookrightarrow\Hom_K(V_\tau,L^2(K,\sigma))&\cong\Hom_M(V_\tau,V_\sigma),
\end{alignedat}\]
with equal images on the right-hand side\footnote{More precisely, the resulting embeddings of $\Hom_{K'}(V_\tau,V_{\sigma'})$ and $\Hom_K(V_\tau,L^2(K,\sigma'))$ into $\Hom_M(V_\tau,V_\sigma)$ are connected by Frobenius reciprocity on the left-hand side.}. We see now that \eqref{varphi1} and \eqref{varphi2} hold for $U=U^{\sigma',\mu'}$ as long as $\{T_u\}$ is an orthonormal basis of this particular subspace of $\Hom_M(V_\tau,V_\sigma)$. We can choose $\{T_u\}$ as follows. Let $(\omega,V_\omega)$ run through a set of representatives for $\widehat{K'}$. For each $\omega$, we equip
\begin{alignat*}{2}
&\Hom_{K'}(V_\tau,V_\omega)&\quad\text{with the scalar product}\quad&\langle P,Q\rangle=(1/d_\omega)\tr(P^*Q);\\
&\Hom_M(V_\omega,V_\sigma) &\quad\text{with the scalar product}\quad&\langle S,T\rangle=(1/d_\sigma)\tr(S^*T).
\end{alignat*}
We then choose an orthonormal basis $\{P_j^{\tau,\omega}\}$ of $\Hom_{K'}(V_\tau,V_\omega)$ and an orthonormal basis $\{T_i^{\omega,\sigma'}\}$ of the image of the embedding (cf. \eqref{frobenius}) 
\[\Hom_{K'}(V_\omega,V_{\sigma'})\hookrightarrow\Hom_{K'}(V_\omega,L^2(K',\sigma))\cong\Hom_M(V_\omega,V_\sigma).\]
By Schur's lemma and the $K'$-isotypic decomposition of $V_\tau$, the products $T_i^{\omega,\sigma'}P_j^{\tau,\omega}$ form an orthonormal basis of the image of $\Hom_{K'}(V_\tau,V_{\sigma'})$ in $\Hom_M(V_\tau,V_\sigma)$. That is, if $\{T_u\}=\{T_i^{\omega,\sigma'}P_j^{\tau,\omega}\}$, then \eqref{varphi1} and \eqref{varphi2} hold for $U=U^{\sigma',\mu'}$:
\begin{equation}\label{Eis2}
\varphi_\tau^{U^{\sigma',\mu'}}(x)=\frac{d_\tau}{d_\sigma}\int_K\tau(\kappa(xk))
\sum_{\omega\in\widehat{K'}}\sum_{i,j}P_j^{\tau,\omega*}T_i^{\omega,\sigma'*}T_i^{\omega,\sigma'}P_j^{\tau,\omega}
\tau(k^{-1})\,e^{(\mu-\rho)(H(xk))}\,dk.
\end{equation}

\subsection{Camporesi's formula}
We can now state and prove ``Camporesi's formula'' \cite[(4.8)--(4.9)]{C},
\begin{equation}\label{Eis}
\varphi_\tau^{U^{\sigma',\mu'}}(x)=\int_K F_\tau^{\sigma'}(xk)\tau(k^{-1})\,e^{(\mu'-\rho')(H'(xk))}\,dk.
\end{equation}
In the formula, the function $F_\tau^{\sigma'}:G\to\End(V_\tau)$ is defined by
\begin{equation}\label{Fdef}
F_\tau^{\sigma'}(x):=\tau(\kappa'(x))\hat F_\tau^{\sigma'}(m'(x)),\qquad x\in G,
\end{equation}
where $\hat F_\tau^{\sigma'}:M'\to\End(V_\tau)$ is the ``block matrix'' valued function
\begin{equation}\label{Fdef2}
\hat F_\tau^{\sigma'}(m'):=\sum_{\omega\in\widehat{K'}}\frac{d_\tau}{d_\omega}
\sum_jP_j^{\tau,\omega\ast}\varphi_{\omega}^{\sigma'}(m')P_j^{\tau,\omega},\qquad m'\in M'.
\end{equation}
We note that only those $\omega$'s contribute to this sum that occur both in $\sigma'$ and $\tau$. Using that $P_j^{\tau,\omega}\in\Hom_{K'}(V_\tau,V_\omega)$ and (cf. \eqref{varphisymmetry})
\[\varphi_{\omega}^{\sigma'}(k_1'm'k_2')=\omega(k_1')\varphi_{\omega}^{\sigma'}(m')\omega(k_2'),
\qquad k_1'\in K',\quad m'\in M',\quad k_2'\in K',\]
we obtain readily the invariance property
\[\hat F_\tau^{\sigma'}(k_1'm'k_2')=\tau(k_1')\hat F_\tau^{\sigma'}(m')\tau(k_2')
\qquad k_1'\in K',\quad m'\in M',\quad k_2'\in K'.\]
In fact the definition \eqref{Fdef} implicitly assumes this invariance property, because $\kappa'(x)$ and $m'(x)$ are not uniquely determined: if we use a given pair $(\kappa'(x),m'(x))$, then for any $k'\in K'$, we can equally use the pair $(\kappa'(x)k'^{-1},k'm'(x))$. It is straightforward to see that
\[\kappa'(kxk')m'(kxk')=k\kappa'(x)m'(x)k',\qquad k\in K,\quad x\in G,\quad k'\in K',\]
whence also
\begin{equation}\label{inv}
F_\tau^{\sigma'}(kxk')=\tau(k)F_\tau^{\sigma'}(x)\tau(k'),\qquad k\in K,\quad x\in G,\quad k'\in K'.
\end{equation}

We need to prove that the right-hand sides of \eqref{Eis2} and \eqref{Eis} are equal. First, we apply \eqref{varphi1} and \eqref{varphi2} with $(\sigma',\omega)$ in the role of $(U,\tau)$, and with $\{T_i^{\omega,\sigma'}\}$ in the role of $\{T_u\}$. We infer that, for all $m'\in M'$,
\[\varphi_\omega^{\sigma'}(m')=\frac{d_\omega}{d_\sigma}\int_{K'}\omega(\kappa_1(m'k'))
\sum_i T_i^{\omega,\sigma'*}T_i^{\omega,\sigma'}\omega(k'^{-1})\,e^{(\mu_1-\rho_1)(H_1(m'k'))}\,dk'.\]
Plugging this into the definitions \eqref{Fdef}--\eqref{Fdef2}, it suffices to verify that
\begin{equation}\label{Eis3}
\int_K\tau(\kappa(xk))P_j^{\tau,\omega*}T_i^{\omega,\sigma'*}T_i^{\omega,\sigma'}P_j^{\tau,\omega}
\tau(k^{-1})\,e^{(\mu-\rho)(H(xk))}\,dk
\end{equation}
equals
\begin{multline*}\int_K\int_{K'}\tau(\kappa'(xk))P_j^{\tau,\omega*}\omega(\kappa_1(m'(xk)k'))T_i^{\omega,\sigma'*}T_i^{\omega,\sigma'}
\omega(k'^{-1})P_j^{\tau,\omega}\tau(k^{-1})\\
\times e^{(\mu_1-\rho_1)(H_1(m'(xk)k'))+(\mu'-\rho')(H'(xk))}\,dk'\,dk.
\end{multline*}
We can simplify the exponential factor by applying the second part of \eqref{iw4} with $xk$ in the role of $x$:
\[H_1(m'(xk)k')+H'(xk)=H(xkk'),\]
so that
\[(\mu_1-\rho_1)(H_1(m'(xk)k'))+(\mu'-\rho')(H'(xk))=(\mu-\rho)(H(xkk')).\]
Since $P_j^{\tau,\omega}\in\Hom_{K'}(V_\tau,V_\omega)$, we also have that
\[\omega(k'^{-1})P_j^{\tau,\omega}=P_j^{\tau,\omega}\tau(k'^{-1}),\qquad
P_j^{\tau,\omega*}\omega(\kappa_1(m'(xk)k'))=\tau(\kappa_1(m'(xk)k'))P_j^{\tau,\omega*}.\]
Hence we are left with showing that \eqref{Eis3} equals
\[\int_K\int_{K'}\tau(\kappa'(xk)\kappa_1(m'(xk)k'))P_j^{\tau,\omega*}T_i^{\omega,\sigma'*}T_i^{\omega,\sigma'}
P_j^{\tau,\omega}\tau(k'^{-1}k^{-1})\,e^{(\mu-\rho)(H(xkk'))}\,dk'\,dk.\]
Applying the first part of \eqref{iw4} with $xk$ in the role of $x$, we are left with showing that \eqref{Eis3} equals
\[\int_K\int_{K'}\tau(\kappa(xkk'))P_j^{\tau,\omega*}T_i^{\omega,\sigma'*}T_i^{\omega,\sigma'}
P_j^{\tau,\omega}\tau((kk')^{-1})\,e^{(\mu-\rho)(H(xkk'))}\,dk'\,dk.\]
However, this becomes clear if we pull the $k$-integration inside and make the change of variable $k\leftarrow kk'$
for all $k'$. The proof of \eqref{Eis} is complete.

\subsection{The trace function}\label{sec27}
Taking the trace of \eqref{Eis} and using \eqref{inv}, we obtain a formula for the generalized spherical trace function:
\begin{align}
\notag \psi_\tau^{U^{\sigma',\mu'}}(x)
&=\int_K \tr\bigl(F_\tau^{\sigma'}(xk)\tau(k^{-1})\bigr)\,e^{(\mu'-\rho')(H'(xk))}\,dk\\
\notag &=\int_K \tr\bigl(\tau(k^{-1})F_\tau^{\sigma'}(xk)\bigr)\,e^{(\mu'-\rho')(H'(xk))}\,dk\\
\label{eq:formula_nonminimal_spherical_trace_function} &=\int_K \tr\bigl(F_\tau^{\sigma'}(k^{-1}xk)\bigr)\,e^{(\mu'-\rho')(H'(xk))}\,dk.
\end{align}
This leads us to analyze the trace function
\begin{equation}\label{eq:T-sigma'}
T_\tau^{\sigma'}(x):=\tr\bigl(F_\tau^{\sigma'}(x)\bigr)=\tr\bigl(\tau(\kappa'(x))\hat F_\tau^{\sigma'}(m'(x))\bigr),\qquad x\in G.
\end{equation}
Since $\tau(\kappa'(x))$ is a unitary matrix,
\begin{equation}\label{tracebound}
\bigl|T_\tau^{\sigma'}(x)\bigr|\leq \bigl\| F_\tau^{\sigma'}(x)\bigr\|_{2,1} = \bigl\|\hat F_\tau^{\sigma'}(m'(x))\bigr\|_{2,1},
\end{equation}
where ${\|\cdot\|}_{2,1}$ stands for the sum of the Euclidean norms of the columns of a square matrix. 
Note that even though $m'(x)$ is only defined up to left-multiplication by $K'$, the norm 
$\|\hat F_\tau^{\sigma'}(m'(x))\bigr\|_{2,1}$ is well-defined. Moreover, this function is left $K$-invariant. Indeed, for any $x\in G$ and $k\in K$, the well-defined product $\kappa'(kx)m'(kx)$ equals $k\kappa'(x)m'(x)$, hence it is legitimate to use $\kappa'(kx)=k\kappa'(x)$ and $m'(kx)=m'(x)$. We conclude the useful bound
\begin{equation}\label{tracefunctionbound}
\Bigl|\psi_\tau^{U^{\sigma',\mu'}}(x)\Bigr|\leq
\int_K \bigl\|\hat F_\tau^{\sigma'}(m'(xk))\bigr\|_{2,1}\,e^{(\Re\mu'-\rho')(H'(xk))}\,dk.
\end{equation}

\section{Generalized spherical functions for \texorpdfstring{$\SL_n^{\pm}(\RR)$}{SLn+-(R)}}\label{sec:generalized-spherical-trace-function-SL}
In this section, we calculate $\varphi_\tau^U(x)$ explicitly in the case when $U$ is a discrete series representation or a generalized principal series representation of
\[G(n):=\SL_n^{\pm}(\RR)\]
and $\tau$ is the minimal $K$-type of $U$ (as defined later in this section). Here of course $K=\OO(n)$. As explained in a moment, this covers the generic unitary dual of $G(n)$, hence also the tempered unitary dual of $G(n)$. We shall see that $\tau$ is unique and occurs with multiplicity one in $U$. We shall also touch briefly on nonminimal $K$-types, to set the stage for the elaborate calculations in \S\ref{thm3bsubsection} and \S\ref{n3-nonminimal-subsec}.

The case of $n=1$ is trivial, so let us focus on $n\geq 2$. If $U$ is generic, then by \cite[Th.~14.10.20]{GH}, it is either a discrete series representation $\delta_\ell$ of $G(2)$ or a generalized principal series representation $U^{\sigma',\mu'}$ of $G(n)$ with $|\Re\mu'|<1/2$. If $U$ is tempered, then by \cite[Cor.~6.2]{Trombi} and the main theorem of \cite{Tadic} (see also (S1) and (U0) there), it is either a discrete series representation $\delta_\ell$ of $G(2)$ or a generalized principal series representation $U^{\sigma',\mu'}$ of $G(n)$ with $\Re\mu'=0$.

\subsection{Parametrizing \texorpdfstring{$\widehat{K}$}{the unitary dual of K}.}\label{parametrizingK}
The theory of highest weights allows us to conveniently label the irreducible representations of $K=\OO(n)$ and $L=\SO(n)$. Our main references are \cite[\S114]{Zh}, \cite[Ch.~VII]{Bo}, \cite[Ch.~VI, \S7]{BD}, \cite[\S5.5.5]{GW}. As is customary, we say that an irreducible representation of $K$ is of Type I (resp.\ Type II) if its restriction to $L$ is irreducible (resp. reducible). Similarly, we say that an irreducible representation of $L$ is of Type I (resp.\ Type II) if its induction to $K$ is reducible (resp. irreducible).

Let $\bm{\lambda}=(\lambda_1,\dotsc,\lambda_{\lfloor{n/2}\rfloor})$ be any (non-strictly) decreasing, nonnegative integer vector. First assume that $n$ is odd. There are two representations $\tau_{\bm{\lambda}}^\pm\in\widehat{K}$ of highest weight $\bm{\lambda}$ (satisfying $\tau_{\bm{\lambda}}^-=\tau_{\bm{\lambda}}^+\otimes\det$); each of them restricts to the unique element of $\widehat{L}$ of highest weight $\bm{\lambda}$. These are all the irreducible representations of $K$ and $L$, and they are all of Type I. The sign $\pm$ in the notation $\tau_{\bm{\lambda}}^\pm$ tells us how the negative identity acts on the representation. By Weyl's dimension formula \cite[Th.~7.1.9]{GW},
\begin{equation}\label{eq:dimformula_1}
\dim\tau_{\bm{\lambda}}^\pm =
\prod_{i<j} \frac{\lambda_i-\lambda_j+j-i}{j-i} \prod_{i\leq j} \frac{\lambda_i+\lambda_j+n-i-j}{n-i-j},
\qquad\text{$n$ odd}.
\end{equation}
Now assume that $n$ is even. If $\lambda_{n/2}=0$, then there are two representations $\tau_{\bm{\lambda}}^\pm\in\widehat{K}$ of highest weight $\bm{\lambda}$ (satisfying $\tau_{\bm{\lambda}}^-=\tau_{\bm{\lambda}}^+\otimes\det$); each of them restricts to the unique element of $\widehat{L}$ of highest weight $\bm{\lambda}$. These are the Type I representations of $K$ and $L$. The sign $\pm$ in the notation $\tau_{\bm{\lambda}}^\pm$ tells us how the reflection $\diag(1,\dotsc,1,-1)$ acts on the highest weight space of the representation. Weyl's dimension formula gives for this case that
\begin{equation}\label{eq:dimformula_2}
\dim\tau_{\bm{\lambda}}^\pm =
\prod_{i<j} \frac{(\lambda_i-\lambda_j+j-i) (\lambda_i+\lambda_j+n-i-j)} {(j-i)(n-i-j)},
\qquad\lambda_{n/2}=0.
\end{equation}
If $\lambda_{n/2}>0$, then there is a unique representation $\tau_{\bm{\lambda}}\in\widehat{K}$ of highest weight $\bm{\lambda}$ (satisfying $\tau_{\bm{\lambda}}=\tau_{\bm{\lambda}}\otimes\det$), and its restriction to $L$ is the direct sum of the two elements of $\widehat{L}$ whose highest weights are $(\lambda_1,\dotsc,\lambda_{n/2-1},\pm\lambda_{n/2})$. These are the Type II representations of $K$ and $L$, and Weyl's dimension formula yields
\begin{equation}\label{eq:dimformula_3}
\dim\tau_{\bm{\lambda}} =
2\prod_{i<j} \frac{(\lambda_i-\lambda_j+j-i) (\lambda_i+\lambda_j+n-i-j)} {(j-i)(n-i-j)},
\qquad\lambda_{n/2}>0.
\end{equation}

Our notion of ``minimal $K$-type'' for an admissible representation of $G(n)$ is based on the standard dominance order of highest weights. In general, if $\bm{\mu},\bm{\mu}'\in\ZZ^{\lfloor n/2\rfloor}$ are two weights for $K$, we define
\begin{equation}\label{eq:standard-dominance}
\bm{\mu}\preceq\bm{\mu}'\qquad\overset{\text{def}}{\Longleftrightarrow}\qquad
\mu_1+\dotsb+\mu_k\leq\mu'_1+\dotsb+\mu'_k\quad\text{for all $1\leq k\leq\lfloor n/2\rfloor$}.
\end{equation}
Then, if $\tau\in\widehat{K}$ has highest weight $\bm{\lambda}$, and $\tau'\in\widehat{K}$ has highest weight $\bm{\lambda}'$, we define
\[\tau\preceq\tau'\qquad\overset{\text{def}}{\Longleftrightarrow}\qquad
\tau=\tau'\quad\text{or}\quad\text{$\bm{\lambda}\preceq\bm{\lambda}'$ and $\bm{\lambda}\neq\bm{\lambda}'$}.\]
In particular, if $\tau\preceq\tau'$ and $\tau\neq\tau'$, then the corresponding Vogan norms \cite[Def.~5.4.18]{Vogan} satisfy $\|\tau\|<\|\tau'\|$:
\[\sum_{i=1}^{\lfloor n/2\rfloor}(\lambda_i+n-2i)^2<\sum_{i=1}^{\lfloor n/2\rfloor}(\lambda'_i+n-2i)^2.\]
This follows readily by rewriting the difference of the two sides via Abel summation.

We shall use below that every weight $\bm{\mu}$ of a representation $\tau\in\widehat{K}$ is dominated by the highest weight $\bm{\lambda}$. More precisely, $\bm{\mu}$ lies in the convex hull of the permutations of the weights $(\pm\lambda_1,\dotsc,\pm\lambda_{\lfloor{n/2}\rfloor})$ as follows from \cite[Th.~7.41]{Hall} or \cite[Prop.~2.2]{KLV} applied to the irreducible components of $\tau|_L$.

\subsection{Discrete series}\label{discrete-series-subsection}
Let $\ell\geq 2$ be an integer, and let $U=\delta_{\ell}$ be the $\ell$-th discrete series representation\footnote{The calculations of this subsection are also valid for $\ell=1$, in which case $\delta_\ell$ is the unique limit of discrete series representation of $G(2)$.} of $G(2)$. Then $U$ has a unique minimal $K$-type $\tau=\omega_{\ell}$ (of highest weight $\ell$), which occurs with multiplicity one. (We keep both notations for ease of reference across multiple sources.) Let us use the notations
\begin{equation}\label{rotationmatrix}
k(t):=\begin{pmatrix} \cos t & \sin t \\ -\sin t & \cos t \end{pmatrix},\qquad
w:=\begin{pmatrix} -1 & \\ & 1 \end{pmatrix}.
\end{equation}
There is an orthonormal basis $\{\bv_\ell,\bv_{-\ell}\}$ of $V_\tau$ such that
\begin{equation}
\label{bvj-j}
\tau(k(t))\bv_j=e^{ijt}\bv_j,\qquad\tau(w)\bv_j=\bv_{-j}
\end{equation}
for $j \in \{\ell, -\ell\}$. It follows from \cite[Th.~2.5.3]{Bump} that $U$ embeds into the minimal principal series representation $U^{\sigma,\mu}$ when $\sigma(\epsilon_1,\epsilon_2)=\epsilon_1^a\epsilon_2^b$ with $a+b\equiv\ell\pmod{2}$ and $\mu(y,-y):=(\ell-1)y$. So there are two choices for $\sigma\in\widehat{M}$, and we opt for $\sigma(\epsilon_1,\epsilon_2):=\epsilon_2^\ell$. We shall use the integral representation \eqref{varphi1} for $\varphi_\tau^U(x)$. In this formula, the map $P:=P_{\tau,\sigma}^U$ given by \eqref{varphi2} is the projection of $V_\tau$ onto its $\sigma$-isotypic subspace spanned by $\bv_\ell+\bv_{-\ell}$. In particular,
\[P\bv_\ell=P\bv_{-\ell}=\frac{\bv_\ell+\bv_{-\ell}}{2}.\]
Since every $m\in M$ normalizes $N$ and centralizes $A$, we see that $\kappa(xkm)=\kappa(xk)m$ and $H(xkm)=H(xk)$. Using also $\tau(m)P\tau(m^{-1})=P$, we infer that the integrand in \eqref{varphi1} is invariant under $k\mapsto km$. Hence it suffices to integrate over $K/M$ instead of $K$, yielding
\[\varphi_\tau^{U}(x)=\frac{2}{\pi}\int_{-\pi/2}^{\pi/2}\tau(\kappa(xk(t)))P\tau(k(-t))\,e^{(\mu-\rho)(H(xk(t)))}\,dt.\]

First assume that $x=\bigl(\begin{smallmatrix} a & b \\ c & d \end{smallmatrix}\bigr)\in\SL_2(\RR)$. Writing
\[A:=a\cos t - b\sin t,\qquad C:=c\cos t - d\sin t,\]
we have the Iwasawa decomposition
\[xk(t)=\begin{pmatrix}A&\ast\\C&\ast\end{pmatrix}=
\begin{pmatrix} \frac{A}{\sqrt{A^2+C^2}} & \frac{-C}{\sqrt{A^2+C^2}} \\ \frac{C}{\sqrt{A^2+C^2}} & \frac{A}{\sqrt{A^2+C^2}} \end{pmatrix}\begin{pmatrix}\scriptstyle\sqrt{A^2+C^2}&\\&\frac{1}{\sqrt{A^2+C^2}}\end{pmatrix}
\begin{pmatrix} 1 & \ast \\ & 1 \end{pmatrix},\]
whence
\[\kappa(xk(t))=\begin{pmatrix} \frac{A}{\sqrt{A^2+C^2}} & \frac{-C}{\sqrt{A^2+C^2}} \\ \frac{C}{\sqrt{A^2+C^2}} & \frac{A}{\sqrt{A^2+C^2}} \end{pmatrix} ,\qquad e^{H(xk(t))}=\begin{pmatrix}\scriptstyle\sqrt{A^2+C^2}&\\&\frac{1}{\sqrt{A^2+C^2}}\end{pmatrix}.\]
It follows that
\begin{align*}
\tau(\kappa(xk(t)))P\tau(k(-t))\bv_j\,e^{(\mu-\rho)(H(xk(t)))}
&=\tau(\kappa(xk(t)))e^{-ijt}\frac{\bv_\ell+\bv_{-\ell}}{2}\cdot\sqrt{A^2+C^2}^{\,\ell-2}\\
&=\frac{e^{-ijt}}{2}\cdot\frac{(A-iC)^\ell \bv_\ell+(A+iC)^\ell \bv_{-\ell}}{A^2+C^2},
\end{align*}
therefore
\[\varphi_\tau^{U}(x)\bv_j=\frac{1}{\pi}\int_{-\pi/2}^{\pi/2}
e^{-ijt}\frac{(A-iC)^\ell \bv_\ell+(A+iC)^\ell \bv_{-\ell}}{A^2+C^2}\,dt.\]
We are interested in the matrix coefficients $\langle\varphi_\tau^{U}(x)\bv_{\pm\ell},\bv_{\pm\ell}\rangle$, which form complex conjugate pairs by the above formula:
\[\ov{\langle\varphi_\tau^{U}(x)\bv_\ell,\bv_{\pm\ell}\rangle}=\langle\varphi_\tau^{U}(x)\bv_{-\ell},\bv_{\mp\ell}\rangle.\]
We make the substitution $r=\tan t$, so that
\[dt=\frac{dr}{1+r^2},\qquad \cos t=\frac{1}{\sqrt{1+r^2}},\qquad \sin t=\frac{r}{\sqrt{1+r^2}}
,\qquad e^{it}=\frac{1+ir}{\sqrt{1+r^2}}.\]
With this substitution, the sought matrix coefficients can be expressed as
\[\langle\varphi_\tau^{U}(x)\bv_\ell,\bv_\ell\rangle=\frac{1}{\pi}\int_{-\infty}^\infty\frac{(a-br-ic+idr)^{\ell-1}}{(a-br+ic-idr)(1+ir)^\ell}\,dr\]
and
\[\langle\varphi_\tau^{U}(x)\bv_\ell,\bv_{-\ell}\rangle
=\frac{1}{\pi}\int_{-\infty}^\infty\frac{(a-br+ic-idr)^{\ell-1}}{(a-br-ic+idr)(1+ir)^\ell}\,dr.\]
In both integrals, we shift the integration contour (which is the directed real axis) indefinitely below. In the lower half-plane, the first integrand has the single pole
\[\rho:=\frac{a+ic}{b+id},\]
therefore the residue theorem yields that
\[\langle\varphi_\tau^{U}(x)\bv_\ell,\bv_\ell\rangle
=-2i\ \underset{r=\rho}{\mathrm{res}}\frac{(a-br-ic+idr)^{\ell-1}}{(a-br+ic-idr)(1+ir)^\ell}
=\left(\frac{2}{a-ib+ic+d}\right)^\ell.\]
The second integrand is holomorphic in the lower half-plane, hence
\[\langle\varphi_\tau^{U}(x)\bv_\ell,\bv_{-\ell}\rangle=0.\]
Now assume that $x=\bigl(\begin{smallmatrix} a & b \\ c & d \end{smallmatrix}\bigr)\in\SL_2^-(\RR)$. Then we use that $\varphi_\tau^{U}(x)=\tau(w)\varphi_\tau^{U}(\widetilde x)$, where $\widetilde x:=wx=\bigl(\begin{smallmatrix} -a & -b \\ c & d \end{smallmatrix}\bigr)\in\SL_2(\RR)$. It follows that
\[\langle\varphi_\tau^{U}(x)\bv_\ell,\bv_{-\ell}\rangle=\langle\varphi_\tau^{U}(\widetilde x)\bv_\ell,\bv_\ell\rangle
=\left(\frac{2}{-a+ib+ic+d}\right)^\ell,\]
and $\langle\varphi_\tau^{U}(x)\bv_{-\ell},\bv_\ell\rangle=\langle\varphi_\tau^{U}(\widetilde x)\bv_{-\ell},\bv_{-\ell}\rangle$ is the complex conjugate of this. Similarly,
\[\langle\varphi_\tau^{U}(x)\bv_\ell,\bv_\ell\rangle=\langle\varphi_\tau^{U}(\widetilde x)\bv_\ell,\bv_{-\ell}\rangle=0,\qquad
\langle\varphi_\tau^{U}(x)\bv_{-\ell},\bv_{-\ell}\rangle=\langle\varphi_\tau^{U}(\widetilde x)\bv_{-\ell},\bv_\ell\rangle=0.\]

To summarize, if we identify $\varphi_\tau^{U}(x)=\varphi_{\omega_{\ell}}^{\delta_{\ell}}(x)$ with its matrix with respect to the basis $\{\bv_\ell,\bv_{-\ell}\}$, then
\begin{equation}\label{basicphi}
\varphi_{\omega_{\ell}}^{\delta_{\ell}}\left(\begin{pmatrix} a & b \\ c & d\end{pmatrix}\right)=\begin{cases} \begin{pmatrix} \left(\frac{2}{a-ib+ic+d}\right)^{\ell} & 0 \\ 0 & \left(\frac{2}{a+ib-ic+d}\right)^{\ell} \end{pmatrix},& ad-bc=1; \\ \begin{pmatrix} 0 & \left(\frac{2}{-a-ib-ic+d}\right)^{\ell} \\ \left(\frac{2}{-a+ib+ic+d}\right)^{\ell} & 0 \end{pmatrix}, & ad-bc=-1. \end{cases}
\end{equation}

For later reference, we remark that a similar but slightly more complicated formula exists for all $K$-types $\tau=\omega_m$ contained in $U=\delta_\ell$. That is, let $m\geq\ell$ be of the same parity as $\ell$, and choose an orthonormal basis $\{\bv_m,\bv_{-m}\}$ of $V_\tau$ such that \eqref{bvj-j} holds for $j\in\{m,-m\}$. Then, for $x=\bigl(\begin{smallmatrix} a & b \\ c & d \end{smallmatrix}\bigr)\in\SL_2(\RR)$, we obtain by an analogous calculation as above that the diagonal matrix entry
$\langle\varphi_\tau^{U}(x)\bv_m,\bv_m\rangle$ equals
\begin{align}
\notag&-2i\ \underset{r=\rho}{\mathrm{res}}\frac{(a-br-ic+idr)^{(m+\ell)/2-1}(1-ir)^{(m-\ell)/2}}{(a-br+ic-idr)^{(m-\ell)/2+1}(1+ir)^{(m+\ell)/2}}\\
\label{eq:genmatrixcoeff}&=\left(\frac{2}{a-ib+ic+d}\right)^m
\sum_{j=0}^{\frac{m-\ell}{2}}
\binom{\frac{m+\ell}{2}-1}{j}\binom{\frac{m-\ell}{2}}{j}
\left(\frac{2-a^2-b^2-c^2-d^2}{4}\right)^j,
\end{align}
while the anti-diagonal matrix entries vanish. For $x=\bigl(\begin{smallmatrix} a & b \\ c & d \end{smallmatrix}\bigr)\in\SL_2^-(\RR)$, the diagonal matrix entries vanish, while the lower left matrix entry is as in the display above, but with $(-a,-b)$ in the role of $(a,b)$. In particular, the triangle inequality coupled with the binomial theorem yields that
\[|\langle\varphi_\tau^{U}(x)\bv_{\pm m},\bv_{\pm m}\rangle|
\leq m^{(m-\ell)/2}\left(\frac{4}{a^2+b^2+c^2+d^2+2}\right)^{\ell/2}.\]
We shall use this bound in a certain range $m\in[\ell,\ell+\Delta]$, where $\Delta\geq 2$. Since the left-hand side is also bounded by one (cf. next paragraph), we can raise both sides to exponent $2/(\Delta\log(\ell+\Delta))<1$ and conclude that
\begin{equation}\label{eq:goodbound}
|\langle\varphi_\tau^{U}(x)\bv_{\pm m},\bv_{\pm m}\rangle|
\leq e\left(\frac{4}{a^2+b^2+c^2+d^2+2}\right)^{\ell/(\Delta\log(\ell+\Delta))}.
\end{equation}

It is clear from the definition \eqref{eq:phidef} that
\begin{equation}\label{eq:varphi_vs_U}
\langle\varphi_\tau^{U}(x)\bv_{\pm m},\bv_{\pm m}\rangle=\langle U(x)\bv_{\pm m},\bv_{\pm m}\rangle,
\end{equation}
hence these functions are bounded by $1$ in absolute value. It is less obvious that the Sobolev norm of order $r$ of these functions (with respect to any orthonormal basis of $\mathfrak{sl}_2(\CC)$) is bounded by $O_r(m^r)$. This follows by a simple induction argument from the following basic fact. Consider the orthogonal basis
\[
H:=\begin{pmatrix}0&1\\-1&0\end{pmatrix},\qquad
R:=\begin{pmatrix}1&i\\i&-1\end{pmatrix},\qquad
L:=\begin{pmatrix}1&-i\\-i&-1\end{pmatrix},
\]
and let $\bv\in V_{U}$ be a unit vector of weight $j$ (hence $|j|\geq\ell$). Then 
\begin{itemize}
\item $dU(H)\bv=ij\bv$,
\item $dU(R)\bv$ has weight $j+2$ and length $\sqrt{(j+\ell)(j-\ell+2)}$,
\item $dU(L)\bv$ has weight $j-2$ and length $\sqrt{(j-\ell)(j+\ell-2)}$.
\end{itemize}

\subsection{Generalized principal series}\label{minimalKtype}
Let $U=U^{\sigma',\mu'}$, where $\sigma'$ is a discrete series representation of
\begin{equation}\label{Mprimedecomposition}
M':=G(2)^r\times G(1)^s\qquad\text{with}\qquad n=2r+s,
\end{equation}
and $\mu':\mfa'\to\CC$ is an arbitrary linear functional. Correspondingly, we set
\[K':=\OO(2)^r\times\OO(1)^s.\]
We also set $t:=\lfloor s/2\rfloor$, so that $\lfloor n/2\rfloor = r+t$. The generalized spherical function $\varphi_\tau^U(x)$ admits the integral representation \eqref{Eis}, so we need to identify the minimal $\tau\in\widehat{K}$ occurring in $U$ and the function $\hat F_\tau^{\sigma'}(m')$ given by \eqref{Fdef2}. As it turns out, $\tau$ is uniquely determined, and it occurs in $U$ with multiplicity one.

Without loss of generality,
\begin{equation}\label{sigma'}
\sigma'=\delta_{\ell_1}\boxtimes\dotsb\boxtimes\delta_{\ell_r}
\boxtimes\underbrace{\sgn\boxtimes\dotsb\boxtimes\sgn}_{\text{$u$ times}}
\boxtimes\underbrace{1\boxtimes\dotsb\boxtimes1}_{\text{$s-u$ times}},
\end{equation}
where
\begin{equation}\label{eq:ordering-ell-s}
\ell_1\geq\dotsb\geq\ell_r\geq 2
\end{equation}
are integers, and $\delta_\ell$ is the $\ell$-th discrete series representation of $G(2)$. Twisting $\sigma'$ by $\det$ has the effect of twisting $U$ by $\det$ and replacing $u$ by $s-u$, hence we can also assume that $u\leq t$. Using \eqref{indres}, Frobenius reciprocity, and Schur's lemma, we obtain for any $\tau\in\widehat{K}$ that
\begin{align}
\Hom_K(\tau,U|_K)
\notag&=\Hom_K(\tau,\Ind_{K'}^K\sigma'|_{K'})\\
\notag&\cong\Hom_{K'}(\tau|_{K'},\sigma'|_{K'})\\
\label{schur-Hom}&\cong\bigoplus_{\omega\in\widehat{K'}}\Hom_{K'}(\tau|_{K'},\omega)\otimes\Hom_{K'}(\omega,\sigma'|_{K'}).
\end{align}
So we need to find the minimal $\tau$'s that share a common $K'$-type with $\sigma'$. The $K'$-types occurring in $\sigma'$ are of the form
\begin{equation}\label{eq:omega-tensor}
\omega=\omega_{m_1}\boxtimes\dotsb\boxtimes\omega_{m_r}
\boxtimes\underbrace{\sgn\boxtimes\dotsb\boxtimes\sgn}_{\text{$u$ times}}
\boxtimes\underbrace{1\boxtimes\dotsb\boxtimes1}_{\text{$s-u$ times}},
\end{equation}
where each $m_p$ is from $\{\ell_p,\ell_p+2,\dotsc\}$ and $\omega_m$ is the irreducible representation of $\OO(2)$ of highest weight $m\in\ZZ_{\geq 1}$. Rearranging the last $s$ factors of $\omega$, we can form $u$ subproducts $\sgn\boxtimes 1$, which 
appear in the representations of $\OO(2)$ of odd highest weight but do not appear in the representations of $\OO(2)$ of even highest weight. This shows that $\tau$ has a weight that dominates
\begin{equation}\label{eq:highest-weight}
(\ell_1,\dotsc,\ell_r,\underbrace{1,\dotsc,1}_{\text{$u$ times}},\underbrace{0,\dotsc,0}_{\text{$t-u$ times}})
\end{equation}
pointwise, and hence the highest weight of $\tau$ dominates \eqref{eq:highest-weight} in the standard dominance order $\preceq$. So let us focus on the $\tau$'s whose highest weight equals \eqref{eq:highest-weight}. Then we only need to work with $(m_1,\dotsc,m_r)=(\ell_1,\dotsc,\ell_r)$, that is,
\begin{equation}\label{eq:unique-omega}
\omega=\omega_{\ell_1}\boxtimes\dotsb\boxtimes\omega_{\ell_r}
\boxtimes\underbrace{\sgn\boxtimes\dotsb\boxtimes\sgn}_{\text{$u$ times}}
\boxtimes\underbrace{1\boxtimes\dotsb\boxtimes1}_{\text{$s-u$ times}}.
\end{equation}
The restriction of $\tau$ to $\OO(2r)\times\OO(s)$ has a unique irreducible component $\tau_1\boxtimes\tau_2$, where $\tau_1$ has highest weight $(\ell_1,\dotsc,\ell_r)$, and $\tau_2$ has highest weight
\[(\underbrace{1,\dotsc,1}_{\text{$u$ times}},\underbrace{0,\dotsc,0}_{\text{$t-u$ times}}).\]
That is, $\tau_2$ is either $\wedge^u V$ or $\wedge^{s-u}V$, where $V$ is the standard representation of $\OO(s)$. Moreover, $\omega$ only occurs in the named irreducible component $\tau_1\boxtimes\tau_2$, which implies that $\tau_2$ contains
\begin{equation}\label{omega}
\underbrace{\sgn\boxtimes\dotsb\boxtimes\sgn}_{\text{$u$ times}}
\boxtimes\underbrace{1\boxtimes\dotsb\boxtimes1}_{\text{$s-u$ times}}.
\end{equation}
This forces $\tau_2=\wedge^u V$, and then $\tau_2$ contains the character \eqref{omega} exactly once. Similarly, $\tau_1$ contains the representation $\omega_{\ell_1}\boxtimes\dotsb\boxtimes\omega_{\ell_r}$ exactly once. The sign of $\tau_2$ is $(-1)^u$ when $s$ is odd, and $+1$ when $s$ is even and $u<t$. (In the remaining case when $s=2u$, there is no sign associated with $\tau_2$.)

In summary, given a representation $U=U^{\sigma',\mu'}$ with $\sigma'$ as in \eqref{sigma'} and $u\leq t$, the minimal $\tau\in\widehat{K}$ occurring in $U$ is uniquely determined, namely it has highest weight \eqref{eq:highest-weight} and sign $(-1)^{\ell_1+\dotsb+\ell_r+u}$ when $n$ is odd, and sign $+1$ when $n$ is even and $u<t$. (In the remaining case when $n=2r+2u$, there is no sign associated with $\tau$.) Moreover, the lowest $K'$-type $\omega$ that appears in $\sigma'$, shown in \eqref{eq:unique-omega}, occurs both in $\sigma'$ and $\tau$ with multiplicity one. So we have identified the unique minimal $K$-type of $U$ (under the partial order $\preceq$ defined in \S\ref{parametrizingK}), and we see from \eqref{schur-Hom} that it occurs in $U$ with multiplicity one. As a by-product, we also see that \eqref{Fdef2} simplifies in the current situation as
\begin{equation}\label{hatFtausigma2}
\hat F_\tau^{\sigma'}\left(\diag(x_1,\dotsc,x_r,\epsilon_1,\dotsc,\epsilon_s)\right)=\frac{d_\tau}{2^r}\bigotimes_{p=1}^r \varphi_{\omega_{\ell_p}}^{\delta_{\ell_p}} (x_p) \cdot \prod_{k=1}^u \epsilon_k,
\end{equation}
where $\varphi_{\omega_{\ell_p}}^{\delta_{\ell_p}}(x_p)$ is given by \eqref{basicphi}, and the right-hand side is understood as a $2^r\times 2^r$ matrix forming a diagonal block of a $d_\tau \times d_\tau$ matrix, with the $2^r\leq d_{\tau}$ rows and columns corresponding to the unique representation $\omega\in\widehat{K'}$ that occurs both in $\sigma'$ and $\tau$, which is given by \eqref{eq:unique-omega} and occurs both in $\sigma'$ and $\tau$ with multiplicity one.

The complementary case $u>t$ in \eqref{sigma'} can be reduced to the previous case $u\leq t$ by twisting $\sigma'$ and $U$ by 
$\det$. Under this change, the $K$-types occurring in $U$ get twisted by $\det$, hence we arrive at the following conclusion after ``untwisting''. The minimal $\tau\in\widehat{K}$ occurring in $U$ is uniquely determined, namely it has weight 
\[
(\ell_1,\dotsc,\ell_r,\underbrace{1,\dotsc,1}_{\text{$s-u$ times}},\underbrace{0,\dotsc,0}_{\text{$t-s+u$ times}})
\]
and sign $(-1)^{\ell_1+\dotsb+\ell_r+u}$ when $n$ is odd, and sign $-1$ when $n$ is even. Moreover, we have $[U:\tau]=1$ and \eqref{hatFtausigma2} as before.

\subsection{Explicit expressions for \texorpdfstring{$n=3$}{n=3}}\label{oldsubsection}
For the proof of Theorem~\ref{thm4}, we spell out the matrix-valued function $\hat F_\tau^{\sigma'}$ defined by \eqref{Fdef2} and the trace function $T_{\tau}^{\sigma'}$ defined by \eqref{eq:T-sigma'} completely explicitly in the case of $n=3$ and $r=s=1$. In this case, $M'=G(2)\times G(1)$ is embedded into $G=G(3)$, with maximal compact subgroup $K'=\OO(2)\times\OO(1)$ embedded into $K=\OO(3)$. For ease of reading, we begin by considering the case of minimal $K$-type addressed in \S\ref{minimalKtype}, which is of critical importance.
Thus $\sigma'=\delta_{\ell}\boxtimes\sgn^{u}$ with some $\ell\in\ZZ_{\geq 2}$ and $u\in\{0,1\}$, and  $\tau=\tau_{\ell}^{\iota}$ with $\iota:=(-1)^{\ell+u}$.

To make everything explicit, we begin by fixing a Wigner basis in the representation space $V_{\tau}=V_{\tau_{\ell}^{\iota}}$ as follows.
Using the normalization from \cite[(2.3)]{Buttcane2018}, we may parametrize $L=\SO(3)$ by Euler angles as
\begin{equation}
\label{euler-angles}
k[\alpha,\beta,\gamma]:=\begin{pmatrix}\cos\alpha&-\sin\alpha&0\\\sin\alpha&\cos\alpha&0\\0&0&1\end{pmatrix}\begin{pmatrix}\cos\beta&0&\sin\beta\\0&1&0\\-\sin\beta&0&\cos\beta\end{pmatrix}\begin{pmatrix}\cos\gamma&-\sin\gamma&0\\\sin\gamma&\cos\gamma&0\\0&0&1\end{pmatrix}
\end{equation}
with $0\leq\alpha,\gamma<2\pi$ and $0\leq\beta\leq\pi$. Recalling the discussion in \S\ref{parametrizingK}, the restriction $\tau|_L$ is the unique irreducible representation of $L$ of highest weight $\ell$ and dimension $2\ell+1$, which has a Wigner $\mathscr{D}$-matrix with entries indexed by $|m|,|m'|\leq\ell$ and given as in \cite[(2.4)]{Buttcane2018} by
\[\mathscr{D}^{\ell}_{m',m}(k[\alpha,\beta,\gamma]):=e^{-im'\alpha}\wigd^{\ell}_{m',m}(\cos\beta)e^{-im\gamma}\]
in terms of the Wigner polynomial $\wigd^{\ell}_{m',m}(\cos\beta)$ and the Jacobi polynomial $P^{(a,b)}_n(x)$ (for $n\in\ZZ_{\geq 0}$) \cite[(3.72), (3.67)]{BiedenharnLouck1981}:
\begin{align*}
\wigd^{\ell}_{m',m}(\cos\beta)&:=\left[\frac{(\ell+m)!(\ell-m)!}{(\ell+m')!(\ell-m')!}\right]^{1/2}\left(\sin\frac{\beta}2\right)^{m-m'}\left(\cos\frac{\beta}2\right)^{m'+m}P^{(m-m',m+m')}_{\ell-m}(\cos\beta),\\
P_n^{(a,b)}(x)&:=\sum_{\nu=0}^n\binom{n+a}{\nu}\binom{n+b}{n-\nu}\left(\frac{x+1}{2}\right)^{\nu}\left(\frac{x-1}{2}\right)^{n-\nu}.
\end{align*}
We extend $\mathscr{D}^{\ell}_{m',m}$ from $L$ to $K$, so that it serves as a Wigner matrix entry of $\tau$:
\[\mathscr{D}^{\ell}_{m',m}(-k[\alpha,\beta,\gamma]):=\iota\mathscr{D}^{\ell}_{m',m}(k[\alpha,\beta,\gamma]).\]
Further, we identify $k\in\OO(2)$ with $k{}^{\sharp}=\big(\begin{smallmatrix}k&0\\0&1\end{smallmatrix}\big)\in K'$. Then, combining the definitions \eqref{rotationmatrix} and \eqref{euler-angles} with the results of \cite[p.~688]{Buttcane2018}, we compute
\begin{align*}
\mathscr{D}^{\ell}_{m',m}(k(t)^{\sharp})&=\mathscr{D}^{\ell}_{m',m}(k[2\pi-t,0,0])=\mathbf{1}_{m=m'}e^{imt},\\
\mathscr{D}^{\ell}_{m',m}(w^{\sharp})&=\mathscr{D}^{\ell}_{m',m}(-k[\pi,\pi,0])=\mathbf{1}_{m=-m'}(-1)^u.
\end{align*}
Thus, denoting by
\[\mathfrak{W}_{\ell}:=(\bv_{\ell},\dotsc,\bv_0,(-1)^u\bv_{-1},\dotsc,(-1)^u\bv_{-\ell})\]
the orthonormal basis of $V_{\tau}$ used in the Wigner matrix representation $\mathscr{D}^{\ell}:K\to\UU(2\ell+1)$, the bases $(\bv_{j},\bv_{-j})$ ($0<j\leq\ell$) agree with those used in \S\ref{discrete-series-subsection} in the sense that $(\bv_{j},\bv_{-j})$ is an orthonormal 
basis of $V_{\omega_j}$ embedded (with multiplicity one) inside $V_{\tau}$ satisfying \eqref{bvj-j}. For $k=\epsilon l\in K$ with $l=k[\alpha,\beta,\gamma]\in L$, $\epsilon\in\{\pm 1\}$, identifying $\tau(k)$ with its matrix with respect to the basis $\mathfrak{W}_{\ell}$ we have that
\[ \tau(k)=\tau_{\ell}^{\iota}(\epsilon l) =\epsilon^{\ell+u}\mathscr{D}^{\ell}(l)=\epsilon^{\ell+u} \left(\begin{matrix}e^{-i\ell(\alpha + \gamma)} \cos(\frac{1}{2} \beta)^{2\ell} & * &e^{i\ell(-\alpha + \gamma)} \sin(\frac{1}{2} \beta)^{2\ell}\\ * & * & *\\ e^{i\ell(\alpha - \gamma)} \sin(\frac{1}{2} \beta)^{2\ell} & * &e^{i\ell(\alpha + \gamma)} \cos(\frac{1}{2} \beta)^{2\ell} \end{matrix}\right).\]

Now, working with the same orthonormal basis $\mathfrak{W}_{\ell}$, \eqref{basicphi} and \eqref{hatFtausigma2} yield that
\[\hat F_\tau^{\sigma'}\!\left(\!\begin{pmatrix}a&b&\\c&d&\\&&\epsilon\end{pmatrix}\!\right)=\frac{2\ell+1}{2}\epsilon^u\cdot
\begin{cases}
\begingroup
\renewcommand{\arraystretch}{0.9}
\setlength{\arraycolsep}{3pt}
\begin{pmatrix} \left(\frac{2}{a-ib+ic+d}\right)^{\ell} & \mathbf{0} & 0 \\ \mathbf{0} & \mathbf{0} & \mathbf{0} \\ 0 & \mathbf{0} & \left(\frac{2}{a+ib-ic+d}\right)^{\ell}
\end{pmatrix}
\endgroup, &\hspace{-0.7em} ad-bc=1,\\ 
(-1)^{u}\!
\begingroup
\renewcommand{\arraystretch}{0.9}
\setlength{\arraycolsep}{3pt}
\begin{pmatrix} 0 & \mathbf{0} & \left(\frac{2}{-a-ib-ic+d}\right)^{\ell} \\ \mathbf{0} & \mathbf{0} & \mathbf{0} \\ \left(\frac{2}{-a+ib+ic+d}\right)^{\ell} & \mathbf{0} & 0
\end{pmatrix}
\endgroup, &\hspace{-0.7em} ad-bc=-1.
\end{cases}\]
We conclude that with the above parametrization, for $k=\epsilon_1k[\alpha,\beta,\gamma]\in K$ and $m'=\epsilon_2\big(\begin{smallmatrix}m''&0\\0&1\end{smallmatrix}\big)\in M'$, $\epsilon_i\in\{\pm 1\}$, $m''=\big(\begin{smallmatrix}a&b\\c&d\end{smallmatrix}\big)\in G(2)$, the trace defined by \eqref{eq:T-sigma'} equals
\begin{equation}
\label{n-3-tr-explicit}
\begin{aligned}
&T_\tau^{\sigma'}(km')=\tr\big(F_{\tau}^{\sigma'}(km')\big)=\tr\bigl(\tau(k)\hat F_\tau^{\sigma'}(m')\bigr)\\
& = (2\ell + 1)(\det(km'))^{\ell+u}\cdot \begin{cases}\Re\Big[ e^{i\ell(\alpha + \gamma)} \cos\big(\frac{1}{2} \beta\big)^{2\ell}\left( \frac{2}{a+ib-ic+d}\right)^{\ell}\Big], &\det m'' = 1,\\[6pt]
\Re\Big[ e^{i\ell(\alpha-\gamma)} \sin\big(\frac{1}{2} \beta\big)^{2\ell}\left( \frac{2}{a+ib+ic-d}\right)^{\ell}\Big], &\det m''= -1.\end{cases}
\end{aligned}
\end{equation}
We note that the two cases are in harmony in light of the elementary property
\begin{equation}
\label{transformation-harmony}
k[\alpha,\beta,\gamma]\begin{pmatrix}a&b&\\c&d&\\&&1\end{pmatrix}=-k[\pi+\alpha,\pi-\beta,2\pi-\gamma]\begin{pmatrix} -a&-b&\\c&d&\\&&1\end{pmatrix}.
\end{equation}

We conclude by noting that the expressions $e^{i(\alpha+\gamma)}\cos(\frac12\beta)^2$ and $e^{i(\alpha-\gamma)}\sin(\frac12\beta)^2$ arising in \eqref{n-3-tr-explicit} may also be conveniently expressed directly in terms of the matrix entries of $l=k[\alpha,\beta,\gamma]$ in the Euler angles parametrization \eqref{euler-angles}. Indeed, writing $l=(l_{ij})_{1\leq i,j\leq 3}$ as
\[ l=\begin{pmatrix} \cos\alpha\cos\beta\cos\gamma-\sin\alpha\sin\gamma&-\cos\alpha\cos\beta\sin\gamma-\sin\alpha\cos\gamma&\cos\alpha\sin\beta\\
\sin\alpha\cos\beta\cos\gamma+\cos\alpha\sin\gamma&-\sin\alpha\cos\beta\sin\gamma+\cos\alpha\cos\gamma&\sin\alpha\sin\beta\\
-\sin\beta\cos\gamma&\sin\beta\sin\gamma&\cos\beta\end{pmatrix}, \]
we find that
\begin{alignat*}{2}
 l_{11}+l_{22}&=\cos(\alpha+\gamma)(1+\cos\beta),&\qquad  l_{21}-l_{12}&=\sin(\alpha+\gamma)(1+\cos\beta),\\
-l_{11}+l_{22}&=\cos(\alpha-\gamma)(1-\cos\beta),&\qquad -l_{21}-l_{12}&=\sin(\alpha-\gamma)(1-\cos\beta),
\end{alignat*}
whence
\begin{equation}
\label{n3-alphaplusgamma}
\begin{split}
 l_{11}+l_{22}+i(l_{21}-l_{12})& =2e^{i(\alpha+\gamma)}\cos(\tfrac12\beta)^2,\\
-l_{11}+l_{22}-i(l_{21}+l_{12})& =2e^{i(\alpha-\gamma)}\sin(\tfrac12\beta)^2.
\end{split}
\end{equation}

Now, we proceed to describe the modifications to the above description in the nonminimal case, along the lines of
\S\ref{thm3bsubsection}. As before, $n=3$, $r=s=1$, $\sigma'=\delta_{\ell}\boxtimes\sgn^{u}$ with some $\ell\in\ZZ_{\geq 2}$ and $u\in\{0,1\}$. Moreover, $\tau=\tau_{\lambda}^{\iota}$ with $\iota:=(-1)^{\ell+u}$ and $\lambda\geq\ell$, and the representations $\omega\in\widehat{K'}$ occurring in both $\sigma'$ and $\tau$ are precisely $\omega=\omega_m\boxtimes\sgn^u$ with $\ell\leq m\leq\lambda$ and $m\equiv\ell\pmod 2$, and each of them occurs in both $\sigma'$ and $\tau$ with multiplicity one.

The description of the Wigner matrices $\mathscr{D}_{m',m}^{\lambda}(k)$ remains verbatim the same, with $\lambda$ in place of $\ell$ and with the evaluation $\mathscr{D}_{m',m}^{\lambda}(w^{\sharp})=\mathbf{1}_{m'=-m}(-1)^{u'}$ with $u':=u+\ell-\lambda$ in place of $u$. Let us denote by
\[\mathfrak{W}_{\lambda}':=(\bv_{\lambda},\dots,\bv_0,(-1)^{u'}\bv_{-1},\dots,(-1)^{u'}\bv_{-\lambda})\] the orthonormal basis of $V_{\tau}$ used in the Wigner matrix representation $\mathscr{D}^{\lambda}:K\to\UU(2\lambda+1)$, and let us identify $\tau(k)$ (for $k=\epsilon l$, $\epsilon\in\{\pm 1\}$, $l=k[\alpha,\beta,\gamma]\in L$) with its matrix with respect to $\mathfrak{W}_{\lambda}'$. Then we have
\[ \tau(k)=\tau_{\lambda}^{\iota}(\epsilon l)=\epsilon^{\ell+u}\mathscr{D}^{\lambda}(l), \]
with diagonal and anti-diagonal entries given for every $-\lambda\leq m\leq\lambda$ by
\begin{equation}
\label{tau-coefficients-general}
\begin{aligned}
\tau(k)_{m,m}&=\epsilon^{\ell+u}e^{-im(\alpha+\gamma)}R_{\lambda,|m|}(-\sin^2\tfrac12\beta,\cos^2\tfrac12\beta),\\ 
\tau(k)_{-m,m}&=\epsilon^{\ell+u}e^{im(\alpha-\gamma)}(-1)^mR_{\lambda,|m|}(\cos^2\tfrac12\beta,-\sin^2\tfrac12\beta),
\end{aligned}
\end{equation}
where
\[ R_{\lambda,m}(u,v):=\sum_{\kappa=0}^{\lambda-m}\binom{\lambda-m}{\kappa}\binom{\lambda+m}{\kappa}u^{\kappa}v^{\lambda-\kappa}. \]
Arguing as in \eqref{eq:genmatrixcoeff}--\eqref{eq:goodbound}, we deduce by the triangle inequality and the binomial theorem that
\[ |\tau(k)_{\pm m,m}|\leq (2\lambda)^{\lambda-|m|}(\cos\tfrac12\beta^{\ast})^{2|m|}, \]
where we denote $\beta^{\ast}=\beta$ or $\pi-\beta$ according to the sign indicated in $\pm m$. For $|m|\in[\lambda-\Delta,\lambda]$ with some $\Delta\geq 2$, raising this bound to power $1/(\Delta\log(2\lambda))<1$ and combining with the trivial matrix coefficient bound $|\tau(k)_{\pm m,m}|\leq 1$, we conclude that
\begin{equation}
\label{eq:goodbound-tau}
|\tau(k)_{\pm m,m}|=\big|R_{\lambda,|m|}(\mp\sin^2\tfrac12\beta^{\ast},\pm\cos^2\tfrac12\beta^{\ast})\big|
\leq e\cdot(\cos\tfrac12\beta^{\ast})^{2|m|/(\Delta\log(2\lambda))}.
\end{equation}

Following our subsequent general discussion around \eqref{superblocks-eq}, we see that the matrix of $\hat{F}_{\tau}^{\sigma'}$ relative to the basis $\mathfrak{W}'_{\lambda}$ has zero $(i,j)$-entry unless $i,j\in\{\pm m\}$ for an integer $\ell\leq m\leq \lambda$ satisfying $m\equiv \ell \pmod 2$. If $m$ is such an integer, then by \eqref{eq:genmatrixcoeff} and the discussion below it, the $(\pm m,\pm m)$-entries form the diagonal $2\times 2$ block
\begin{align*}
\hat F_{\tau,(\pm m,\pm m)}^{\sigma'}&\left(\begin{pmatrix}a&b&\\c&d&\\&&\epsilon\end{pmatrix}\right)=\frac{2\lambda+1}2\epsilon^uQ_{m,\ell}\left(\begin{pmatrix}a&b\\c&d\end{pmatrix}\right) \\ 
& \times
\begin{cases}
\begin{pmatrix} \left(\frac{2}{a-ib+ic+d}\right)^{m} & 0  \\ 0 & \left(\frac{2}{a+ib-ic+d}\right)^{m} \end{pmatrix}, & ad-bc=1,\\ 
(-1)^{u'}\begin{pmatrix} 0  & \left(\frac{2}{-a-ib-ic+d}\right)^{m}  \\ \left(\frac{2}{-a+ib+ic+d}\right)^{m} & 0 \end{pmatrix}, & ad-bc=-1,
\end{cases}
\end{align*}
where
\[ Q_{m,\ell}\left(\begin{pmatrix}a&b\\c&d\end{pmatrix}\right):=\sum_{j=0}^{\frac{m-\ell}{2}}\binom{\tfrac{m+\ell}2-1}j\binom{\tfrac{m-\ell}2}j\left(\frac{2-a^2-b^2-c^2-d^2}4\right)^j. \]

Combining the above inputs, we find that for $k=\epsilon_1 l\in K$, $l=k[\alpha,\beta,\gamma]\in L$, $m'=\epsilon_2\big(\begin{smallmatrix}m''&0\\0&1\end{smallmatrix}\big)\in M'$, $\epsilon_i\in\{\pm 1\}$, $m''=\big(\begin{smallmatrix}a&b\\c&d\end{smallmatrix}\big)\in G(2)$, and 
$\varsigma:=\det(m'')\in\{\pm 1\}$, the trace function
\begin{equation}
\label{n-3-tr-explicit-nonminimal}
\begin{aligned}
T_{\tau}^{\sigma'}(km')=&\tr\bigl(\tau(k) \hat{F}_{\tau}^{\sigma'}(m')\bigr)=\frac{2\lambda+1}2(\det(km'))^{\ell+u}\\
&\times \varsigma^{\lambda}\sum_{\substack{\ell\leq m\leq\lambda\\m\equiv\ell\pmod*{2}}}\sum_{\pm}\tau(l)_{\pm m,\pm\varsigma m}\cdot\langle\varphi_{\omega_m}^{\delta_{\ell}}(m'')\bv_{\pm m},\bv_{\pm \varsigma m}\rangle
\end{aligned}
\end{equation}
is explicitly given by
\begin{equation}
\label{n-3-tr-explicit-nonminimal-2}
\begin{aligned}
&T_{\tau}^{\sigma'}(km')=(2\lambda+1)(\det(km'))^{\ell+u}\sum_{\substack{\ell\leq m\leq\lambda\\m\equiv\ell\pmod*{2}}}Q_{m,\ell}(m'')\\
&\,\,\times\begin{cases}\Re\Big[e^{im(\alpha+\gamma)}\left(\tfrac{2}{a+ib-ic+d}\right)^m\Big]R_{\lambda,m}(-\sin^2\tfrac12\beta,\cos^2\tfrac12\beta),&\det(m'')=1,\\
(-1)^{\lambda}\Re\Big[e^{im(\alpha-\gamma)}\left(\tfrac{2}{a+ib+ic-d}\right)^m\Big]R_{\lambda,m}(\cos^2\tfrac12\beta,-\sin^2\tfrac12\beta),&\det(m'')=-1,\end{cases}
\end{aligned}
\end{equation}
and the two cases are again in harmony in view of \eqref{transformation-harmony}.

\section{Bounds for spherical functions I: the general case}
\label{thm3-section}
In this section, we prove Theorems~\ref{thm3} and \ref{thm3b}. As in the previous section, we shall work with
\[G = \SL_n^{\pm}(\RR),\qquad K=\OO(n).\]
We recall that $U=U^{\sigma',\mu'}$ is a generalized principal series representation with (unique) minimal $K$-type $\tau$, and $\Omega\subseteq G$ is a bi-$K$-invariant compact set. In \S\ref{thm3bsubsection}, we shall allow more general $K$-types $\tau$.

\subsection{Notations}
Let $\mfk$ denote the Lie algebra of $K$. It is spanned by the matrices 
\[X_{i,j}\in\RR^{n\times n},\qquad 1\leq i<j\leq n,\]
given by
\[
(X_{i,j})_{i',j'}=\begin{cases}1,\qquad&\text{if $(i',j')=(i,j)$},
\\ -1,&\text{if $(i',j')=(j,i)$}, 
\\ 0, &\text{otherwise}.
\end{cases}
\]
For $n=2$, set $\|X_{1,2}\|:=1$. For $n\geq 3$, consider the normalized Killing form
\[
\langle X,Y\rangle :=\frac{1}{2(2-n)} \tr(\ad X \circ \ad Y)=-\frac{1}{2}\tr(XY), \qquad X,Y\in\mfk,
\]
where $\ad$ stands for the adjoint action on $\mfk$, and set $\|X\|:=\sqrt{\langle X, X\rangle}$. In any case, $(X_{i,j})_{1\leq i<j\leq n}$ is an orthonormal basis of $\mfk$. Accordingly, we shall often think of $\mfk$ as a $\dim K$-dimensional Euclidean space, and call its unit vectors \emph{direction}s.

We introduce the following basic neighborhoods of the identity in $K$:
\begin{equation}\label{defU}
U(r):=\left\{\exp X: X\in\mfk,\ \|X\|<r\right\}, \qquad r>0.
\end{equation}
We fix some $R>0$ which is small enough to guarantee that the exponential map is a bijection from $\{X\in\mfk:\|X\|<R\}$ onto $U(R)$
and
\begin{equation}\label{eq:lebesgue-haar}
\int_{U(R)} f(k) \,dk = \int_{\|X\|<R} f(\exp X)\, c(X)\,d\lambda(X),\qquad f\in C(U(R)),
\end{equation}
where $\lambda$ is the Lebesgue measure and $c\asymp 1$ is a smooth function on $\mfk$.

We recall that $\Omega \subseteq G$ is a fixed bi-$K$-invariant compact set. Unless noted otherwise, all implicit and explicit constants may depend on $\Omega$.

\subsection{Mixing properties}
The key results in this section are Lemmata~\ref{lemma:D-derivatives} and~\ref{lemma:Delta-derivatives} saying that certain functions on $G$ have at least some derivatives that are not too small. They will be used in \S\ref{sec43} and \S\ref{sec45}, respectively.

For any matrix $M\in\RR^{n\times n}$, denote by $M_{i,\cdot}\in\RR^{1\times n}$ its $i$-th row and by $M_{\cdot,j}\in\RR^{n\times 1}$ its $j$-th column. The left action of $K$ on the column space $\RR^{n\times 1}$ is given by isometries:
\[
\langle kM_{\cdot,j},kM_{\cdot,j'} \rangle = \langle M_{\cdot,j},M_{\cdot,j'}\rangle, \qquad k\in K,\quad 1\leq j,j'\leq n.
\]
In particular, the $m$-dimensional volume of the parallelepiped spanned by any $m$ columns of $M$ equals the same quantity for $kM$. However, these quantities might be very different for $M$ and $Mk$. Now we formulate two incarnations of this phenomenon.

In order to measure the relative difference of the lengths of the first two columns of $x$, we introduce
\begin{equation}\label{eq:D-def}
D(x):=\frac{D_{\text{n}}(x)}{D_{\text{d}}(x)}:=\frac{\|x_{.,1}\|^2-\|x_{.,2}\|^2}{\|x_{.,1}\|^2+\|x_{.,2}\|^2}, \qquad x\in G.
\end{equation}
We shall also use, for $x\in G$, $Y\in\mfk$, $t\in\RR$, the notation
\begin{align*}
D_{x,Y}(t)&:=D(x \exp(tY)),\\
D_{x,Y,\text{n}}(t)&:=D_{\text{n}}(x \exp(tY)),\\
D_{x,Y,\text{d}}(t)&:=D_{\text{d}}(x \exp(tY)).
\end{align*}

\begin{lemma}\label{lemma:D-derivatives-upper-bound}
For any $x\in \Omega$, any direction $Y\in\mfk$ and any integer $j\geq 0$, we have
\[
D_{x,Y}^{(j)}(t)|_{t=0} \ll_{j} \dist(x,K).
\]
\end{lemma}

\begin{proof}
Since both sides of the stated inequality are left $K$-invariant in $x$, we may assume that
\begin{equation}\label{m}
x=\begin{pmatrix} y_1 & & (x_{i,i'})_{1\leq i<i'\leq n} \\ & \ddots & \\ & & y_n \end{pmatrix}
\qquad\text{with}\qquad y_i\asymp 1,\ x_{i,i'}\ll 1.
\end{equation}
By \eqref{eq:D-def} and the fact that $D_{\text{d}}(x)\asymp 1$ for $x\in\Omega$, it suffices to prove that
\[
D_{x,Y,\text{n}}^{(j)}(t)|_{t=0} \ll_{j} \dist(x,K), \qquad j \geq 0,
\]
and
\[
D_{x,Y,\text{d}}^{(j)}(t)|_{t=0} \ll_{j} \dist(x,K), \qquad j \geq 1. 
\]
One can readily see that the left-hand sides of these can be written as
\[
\Pol_{Y,j}\bigl(y_1-1,\dotsc,y_{n}-1,x_{1,2},\dotsc,x_{n-1,n}\bigr),
\]
with some polynomials $\Pol_{Y,j}$ of bounded degree (in terms of $j$) and bounded coefficients (in terms of $j$ and $\Omega$) and vanishing constant coefficient.
For the latter claim, note that, indeed, if each $y_i$ equals $1$ and each $x_{i, i'}$ equals $0$, then $x \in K$, hence $D_{x,Y,\text{n}}$ is identically zero and $D_{x,Y,\text{d}}$ is constant.

We finish the proof by showing that
\begin{equation}\label{distanceproperty}
\max\left(\max_{1 \leq i \leq n} |y_i - 1|,\max_{1\leq i<i'\leq n}|x_{i,i'}|\right)\asymp\dist(x,K).
\end{equation}
To see this, we write $x=k_1ak_2$ with $a=\diag(a_1,\dotsc,a_n)$ a positive diagonal matrix and $k_1,k_2\in K$. Then
\[\dist(x,K)=\dist(a,K)\leq\|a-\id_n\|\ll\max_{1\leq i\leq n}(a_i-1)\leq\dist(a,K).\]
Therefore, $\dist(x,K)\asymp\|\log a\|$, and \eqref{distanceproperty} follows from \cite[Lem.~3]{BlomerHarcosMaga2019}.
\end{proof}

\begin{lemma}\label{lemma:D-derivatives}
There exists a constant $0<c_1:=c_1(\Omega)<R$ 
with the following property. For any $x\in\Omega\setminus K$, 
there exist $j\in\{1,2,3\}$ and an orthonormal basis $(Y_1,\dotsc,Y_{\dim K})$ of $\mfk$ such that for any $t_1,\dotsc,t_{\dim K}\in [-c_1,c_1]$ we have 
\[ \left| \frac{\partial^j}{\partial t_1^j} D\left(x \exp\left(\sum_{l=1}^{\dim K} t_l Y_l\right) \right) \right|\geq c_1\cdot\dist(x,K). \]
\end{lemma}
\begin{proof} 
Since both sides of the stated inequality are left $K$-invariant in $x$, we may assume \eqref{m} as before.
Let 
$$\mcY: = \{X_{1,i} : 2 \leq i \leq n\} \cup \{X_{2,i} : 3 \leq i \leq n\} \cup \{X_{1,i} + X_{i,i'} : 3\leq i < i' \leq n\} \subseteq \mfk.$$
By Lemma~\ref{lemma:D-derivatives-upper-bound}, the first four derivatives of $D_{\xi,Y}$ are $\ll\dist(x,K)$ uniformly over $\xi\in xK \subseteq \Omega$ and all directions $Y\in\mfk$. Hence it suffices to prove that
for each $x \in \Omega$ we have
\begin{equation}\label{eq:D-derivatives}
\max_{1 \leq j \leq 3} \max_{Y \in \mcY} \bigl| D_{x,Y}^{(j)}(0) \bigr|
\gg\dist(x, K).
\end{equation}
We compute the derivatives in question by brute force, and assume throughout that $t$ is sufficiently small in terms of $\Omega$.

For $2 \leq i \leq n$, we have
$$\exp(t X_{1,i}) = \left(\begin{matrix}1 - \frac{1}{2} t^2 & & t & \\ & \boxed{\id_{i-2}} & \\ - t & & 1 - \frac{1}{2}t^2 & \\ &&&\boxed{ \id_{n-i} } \end{matrix}\right) + O(t^3)$$
where all missing entries are 0, so that we have the column decompositions
\[
x \exp(t X_{1,2}) = \left( \begin{matrix} x_{., 1} - x_{.,2} t - \frac{1}{2} x_{.,1}t^2 &|& x_{., 2} + x_{., 1} t - \frac{1}{2} x_{., 2} t^2 &|& \cdots \ \; \end{matrix}\right) + O(t^3)
\]
and
\[
x \exp(t X_{1,i}) = \left( \begin{matrix} x_{., 1} - x_{.,i} t - \frac{1}{2} x_{.,1}t^2 &|& x_{., 2} &|& \cdots \ \; \end{matrix}\right) + O(t^3), \qquad i\geq 3.
\]
It follows that
\begin{align}
\notag D_{x,X_{1,2}}(t)& = 
\frac{(y_1 - x_{1,2} t - \frac{1}{2} y_1t^2)^2 - (x_{1,2} + y_1 t - \frac{1}{2} x_{1,2} t^2)^2-y_2^2+2y_2^2t^2 +O(t^3)}{(y_1 - x_{1,2} t - \frac{1}{2} y_1t^2)^2 + (x_{1,2} + y_1 t - \frac{1}{2} x_{1,2} t^2)^2 +y_2^2+O(t^3)}\\
\label{x12} & = \frac{y_1^2 - y_2^2 - x_{1,2}^2}{y_1^2 + y_2^2 + x_{1,2}^2} - \frac{4 y_1 x_{1,2}}{y_1^2 + y_2^2 + x_{1,2}^2} t - \frac{2(y_1^2 - y_2^2 - x_{1,2}^2)}{y_1^2 + y_2^2 + x_{1,2}^2}t^2 + O(t^3),
\end{align}
while for $3\leq i\leq n$,
\begin{align}
\notag D_{x,X_{1,i}}(t) = \ &\frac{y_1^2 - y_2^2 - x_{1,2}^2 - 2 y_1 x_{1,i} t + (\|x_{.,i}\|^2 - y_1^2 )t^2+O(t^3)}{y_1^2 + y_2^2 + x_{1,2}^2 - 2 y_1 x_{1,i} t + (\|x_{.,i}\|^2 - y_1^2 )t^2+O(t^3)}\\ 
\label{x1j} = \ &\frac{y_1^2 - y_2^2 - x_{1,2}^2}{y_1^2 + y_2^2 + x_{1,2}^2} - \frac{4(y_2^2 + x_{1,2}^2) y_1 x_{1,i}}{(y_1^2 + y_2^2 + x_{1,2}^2)^2} t \\
&\notag - \frac{2( y_2^2 + x_{1,2}^2) (4 y^2_1 x^2_{1,i} + (y_1^2 + y_2^2 + x_{1,2}^2) ( y_1^2 - \|x_{.,i}\|^2 ) )}{(y_1^2 + y_2^2 + x_{1,2}^2)^3} t^2 + O(t^3).
\end{align}

Similarly, for $3\leq i \leq n$, we obtain the following approximations:
$$\exp(t X_{2,i}) = \left(\begin{matrix}1 & & & & \\ & 1 - \frac{1}{2} t^2 & & t & \\ & & \boxed{\id_{i-3}} & \\ & - t & & 1 - \frac{1}{2}t^2 & \\ &&&&\boxed{ \id_{n-i} } \end{matrix}\right) + O(t^3)$$
and
$$x \exp(t X_{2,i}) = \left( \begin{matrix} x_{., 1} &|& x_{., 2} - x_{., i} t - \frac{1}{2} x_{., 2}t^2 &|& \cdots \ \;\end{matrix}\right) + O(t^3),$$
whence
\begin{align}
\notag D_{x,X_{2,i}}(t) =\ & \frac{y_1^2-y_2^2-x_{1,2}^2+2(x_{1,2}x_{1,i}+y_2x_{2,i})t+(y_2^2+x_{1,2}^2-\|x_{.,i}\|^2)t^2+O(t^3)}{y_1^2+y_2^2+x_{1,2}^2-2(x_{1,2}x_{1,i}+y_2x_{2,i})t-(y_2^2+x_{1,2}^2-\|x_{.,i}\|^2)t^2+O(t^3)}\\
\label{x2j} = \ & \frac{y_1^2 - y_2^2 - x_{1,2}^2}{y_1^2 + y_2^2 + x_{1,2}^2} + \frac{ 4y_1^2 ( x_{1,2}x_{1,i} + y_2 x_{2,i})}{(y_1^2 + y_2^2 + x_{1,2}^2)^2} t \\
\notag & + \frac{2 y_1^2 (4( x_{1,2}x_{1,i} + y_2 x_{2,i})^2 + (y_1^2 + y_2^2 + x_{1,2}^2)
(y_2^2 + x_{1,2}^2 - \|x_{.,i}\|^2))}{(y_1^2 + y_2^2 + x_{1,2}^2)^3} t^2 + O(t^3).
\end{align}

Finally, for $3 \leq i < i' \leq n$, we have
$$\exp(t (X_{1,i} + X_{i,i'})) = \left(\begin{matrix}1 - \frac{1}{2} t^2 & & t - \frac{1}{3} t^3& & \frac{1}{2} t^2\\ & \boxed{\id_{i-2}} & \\ - t + \frac{1}{3} t^3& & 1 - t^2 & & t - \frac{1}{3} t^3 \\ &&&\boxed{ \id_{i'-i-1} }\\ \frac{1}{2}t^2 & & - t + \frac{1}{3} t^3 & & 1 - \frac{1}{2}t^2 & \\ &&&&& \boxed{ \id_{n-i'} } \end{matrix}\right) + O(t^4)$$
and
$$x \exp(t (X_{1,i} + X_{i,i'})) = \left( \begin{matrix} x_{., 1} - x_{., i} t + \frac{1}{2} (x_{., i'} - x_{., 1})t^2 + \frac{1}{3} x_{., i}t^3 &|& x_{., 2} &|& \cdots \ \; \end{matrix}\right)+ O(t^4).$$
We shall only use this formula when the non-diagonal entries of the first $i-1$ rows of $x$ are sufficiently small in terms of $\Omega$. Accordingly, $\Pol_i(t)\in\RR[t]$ will temporarily denote any polynomial of degree at most three (not necessarily the same one at each occurrence) whose coefficients are
\[
\ll \max_{1\leq k<i} \max_{k<l\leq n} |x_{k,l}|.
\]
We conclude that
\begin{align}
\notag D_{x,X_{1,i} + X_{i,i'}}(t) &= 
\frac{y_1^2 - y_2^2 - (y_1^2 - y_i^2)t^2 - y_i x_{i,i'} t^3 + \Pol_i(t)+ O(t^4)}{y_1^2 + y_2^2 - (y_1^2 - y_i^2) t^2 - y_i x_{i,i'} t^3 + \Pol_i(t)+ O(t^4)}\\
\label{xdouble}& = \frac{y_1^2 - y_2^2}{y_1^2 + y_2^2} - \frac{2 y_2^2 (y_1^2- y_i^2)}{(y_1^2 + y_2^2)^2} t^2 - \frac{2 y_2^2 y_i x_{i, i'}}{(y_1^2 + y_2^2)^2} t^3 + \Pol_i(t) + O(t^4).
\end{align}

Armed with these computations, we can now complete the proof of \eqref{eq:D-derivatives}. Let us introduce (cf. \eqref{distanceproperty})
$$\Delta(x) := \max \left(\max_{2 \leq i \leq n} |y_1 - y_i|, \max_{1 \leq i < i' \leq n} |x_{i,i'}|\right) \asymp \dist(x, K).$$
We consider consecutively the following derivatives using \eqref{x12}--\eqref{xdouble}. 
From \eqref{x12} and \eqref{x1j} we infer
$$\max_{Y \in \mcY} \left| D_{x,Y}' (0) \right| \gg \max_{2 \leq i \leq n} |x_{1,i}|.$$
If the right-hand side is $\asymp \Delta(x)$, we are done. Otherwise, we use \eqref{x2j} to conclude
$$\max_{Y \in \mcY} \left| D_{x,Y}' (0) \right| \gg \max_{3 \leq i \leq n} |x_{2,i}|.$$
Either we are done at this point, or otherwise we consider the third Taylor term in \eqref{xdouble} consecutively for $i = 3, 4, \dotsc,n-1$ and $i'=i+1,\dots,n$ to obtain eventually
$$\max_{Y \in \mcY} \left| D_{x,Y}''' (0) \right| \gg \max_{3 \leq i < i' \leq n} |x_{i,i'}|.$$
Again, either we are done at this point, or otherwise the second Taylor term in \eqref{x12} and \eqref{x1j} gives
$$\max_{Y \in \mcY} \left| D_{x,Y}'' (0) \right| \gg \max_{2 \leq i \leq n} |y_1 - y_i|.$$
This completes the proof of \eqref{eq:D-derivatives} and hence that of Lemma~\ref{lemma:D-derivatives}. 
\end{proof}

For $1\leq m\leq n$, we denote by $x_{(m)}$ the $n\times m$ matrix given by the first $m$ columns of $x$. Then 
\begin{equation}\label{deltam}
\Delta_m(x):= \det \left(x_{(m)}^T x_{(m)}\right), \qquad x\in G,
\end{equation}
is the square of the $m$-dimensional volume of the parallelepiped spanned by the first $m$ columns of $x$. This quantity is invariant under the right action of the group
\[
K_{m}:=\left\{\begin{pmatrix} k' & \mathbf{0} \\ \mathbf{0} & \id_{n-m} \end{pmatrix}:k'\in\OO(m)\right\},
\]
because
\begin{equation}\label{eq:xkm-xmk'}
(x k)_{(m)} = x_{(m)}k',\qquad k = \begin{pmatrix} k' & \mathbf{0} \\ \mathbf{0} & \id_{n-m} \end{pmatrix}\in K_m.
\end{equation}
A related quantity somewhat similar to $D$ is
\begin{equation}\label{eq:Em-def}
E_{m}(x):= \bigl\|x_{(m)}^T x_{(m)} - \Delta_m(x)^{1/m} \id_m\bigr\|_{\infty},\qquad x\in G,
\end{equation}
where $\| \cdot \|_{\infty}$ denotes the largest entry in absolute value of the matrix in question. This quantity is left $K$-invariant. Moreover, it is essentially right $K_m$-invariant, by which we mean that
\begin{equation}\label{eq:E_m-K_m-invariance}
E_{m}(x)\asymp E_{m}(xk),\qquad k\in K_{m}
\end{equation}
holds with an implied constant depending only on $n$ (not on $\Omega$). 
Indeed, for $k$ and $k'$ as in \eqref{eq:xkm-xmk'}, we see that
\begin{align*}
E_{m}(xk) & = \bigl\| k'^Tx_{(m)}^T x_{(m)}k' - \Delta_m(x)^{1/m} \id_m \bigr\|_{\infty} \\ & = \bigl\| k'^{-1} ( x_{(m)}^T x_{(m)} - \Delta_m(x)^{1/m} \id_m) k' \bigr\|_{\infty} .
\end{align*}
For later reference, we record the following consequence of the triangle inequality:
\begin{equation}\label{eq:E_m-related-to-column-lengths}
\|x_{.,j}\|^2 - \|x_{.,j'}\|^2 \ll E_m(x),\qquad 1 \leq j < j' \leq m.
\end{equation}

For convenience, we introduce $\mfk_{m}\subseteq \mfk$ for the Lie algebra of $K_{m}$, which is spanned by
\[
\{X_{i,j}:1\leq i<j\leq m\}.
\]
Our next lemma is similar to Lemma~\ref{lemma:D-derivatives-upper-bound}.
\begin{lemma}\label{lemma:E-m-derivatives-upper-bound} Let $2\leq m \leq n$ and $Y\in\mfk_{m}$ be any direction, and $1\leq i\leq m-1$. Then for $x\in\Omega$ and any $j\geq 1$, we have
\begin{equation}\label{eq:E-m-derivatives-upper-bound}
\left.\frac{d^j}{dt^j} \log \Delta_i(x\exp(tY))\right|_{t=0} \ll_{j} E_{m}(x).
\end{equation}
\end{lemma}

\begin{proof}
For $t \in \RR$ we define $F_Y(t) \in \RR^{m\times i}$ to be the top-left $m\times i$ block of $\exp(tY)$, and for any symmetric positive definite matrix $H \in \text{Pos}_m(\RR)$ we define 
\[
P_{Y,j}(H):= \left.\frac{d^j}{dt^j} \log \det(F_Y(t)^T H F_Y(t))\right|_{t=0}.
\]
Then
\[
P_{Y,j}(x_{(m)}^T x_{(m)}) = \left.\frac{d^j}{dt^j} \log \Delta_i(x\exp(tY))\right|_{t=0}
\]
is the quantity on the left-hand side of the statement. Moreover, the functions $P_{Y,j}$ are trivially invariant under positive scalar multiplication on the open cone $\text{Pos}_m(\RR)$. For future reference, we observe that every $m\times m$ real symmetric matrix $H$ can be written as
\[
H=\sum_{1\leq u\leq v\leq m} H_{u,v} S_{u,v}, \qquad H_{u,v}\in\RR,
\]
where
\[
(S_{u,v})_{u',v'}:=\begin{cases}1,\qquad&\text{if $(u',v')$ equals $(u,v)$ or $(v,u)$},
\\ 0, &\text{otherwise}.
\end{cases}
\]
Correspondingly, smooth functions $P$ on the open cone $\text{Pos}_m(\RR)\subseteq\text{Sym}_m(\RR)$ may be locally written as smooth functions $P((H_{u,v})_{1\leq u\leq v\leq m})$ and their derivatives computed as $\partial P/\partial H_{u,v}$ and combined into the gradient $\nabla P=(\partial P/\partial H_{u,v})_{1\leq u\leq v\leq m}$.

In order to prove the lemma, assume first that $E_{m}(x)=0$, that is, $x_{(m)}^T x_{(m)}=\lambda\cdot \id_m$, where $\lambda:=\Delta_m(x)^{1/m}$. Then, independently of $t\in\RR$, we have
\[
F_Y(t)^T x_{(m)}^T x_{(m)} F_Y(t) = \lambda\cdot F_Y(t)^T F_Y(t) = \lambda\cdot \id_i.
\]
Consequently, $P_{Y,j}(x_{(m)}^T x_{(m)})=0$ for all $j\geq 1$. For a general $x \in \Omega$,
we can write $x_{(m)}^T x_{(m)} = \lambda \cdot \id_m + \mcE(x)$, where $\lambda: = \Delta_m(x)^{1/m}\in [\lambda_0,\lambda_1]$ for some $\lambda_0,\lambda_1\asymp 1$.
Observe that for any $c>0$ depending only on $\Omega$, \eqref{eq:E-m-derivatives-upper-bound} is trivial in the region $\|\mcE(x)\|\geq c\lambda_0$ by the compactness of $\Omega$ and that $\max_{u,v} |\mcE(x)_{u,v}| = E_{m}(x)$ by definition. We choose this $c>0$ so that the convex closed ball $\overline{B}(\id_m,c)$ lies in $\text{Pos}_m(\RR)$. Then, using the scaling invariance of $P_{Y,j}$, the fact that $\lambda\asymp 1$, and the Mean Value Theorem within this ball, we conclude that
\[
P_{Y,j}(x_{(m)}^T x_{(m)}) = P_{Y,j}(\lambda \cdot \id_m) +O\left(\big\|\nabla P_{Y,j}\big|_{\overline{B}(\id_m,c)}\big\|_{\infty}\cdot\|\mcE(x)\|_{\infty}\right). \]
This completes the proof since $P_{Y,j}(\lambda \cdot \id_m) = 0$ and $\overline{B}(\id_m,c)$ is compact.
\end{proof}

\begin{lemma}\label{lemma:Delta-derivatives} There exists a constant $0<c_2:=c_2(\Omega)<R$ with the following property. For any $x\in\Omega$, there exist $j\in\{1,2\}$ and an orthonormal basis $(Y_1,\dotsc,Y_{\dim K_{m}})$ of $\mfk_{m}$, such that for any $t_1,\dotsc,t_{\dim K_{m}}\in [-c_2,c_2]$ we have
\[
\left| \frac{\partial^j}{\partial t_1^j} \log \Delta_{m-1}\left(x \exp\left(\sum_{l=1}^{\dim K_{m}} t_l Y_l\right) \right) \right| \geq c_2 \cdot E_{m}(x).
\]
\end{lemma}
\begin{proof}
By Lemma~\ref{lemma:E-m-derivatives-upper-bound}, the first three derivatives of $\log \Delta_{m-1}(\xi\exp(tY))$ are $\ll E_{m}(x)$ uniformly over $\xi\in xK_{m} \subseteq \Omega$ and all directions $Y\in\mfk_{m}$. Hence it suffices to prove that
for each $x \in \Omega$ we have
\begin{equation}\label{eq:maxmax}
\max_{1 \leq j \leq 2} \max_{\substack{Y\in\mfk_{m}\\ \|Y\|=1}} \left| \left.\frac{\partial^j}{\partial t^j} \log \Delta_{m-1}(x \exp(tY))\right|_{t=0} \right| \gg E_{m}(x).
\end{equation}
Also, since $\Delta_{m-1}(x)\asymp 1$ uniformly over $x\in\Omega$, we can drop the logarithm. Indeed, for a smooth function $f:\RR\to\RR$ with values $\asymp 1$, if $f'$ is large at a point, then $(\log f)'=f'/f$ is also large, while if $f'$ is small, but $f''$ is large, then $(\log f)''=f''/f-(f'/f)^2$ is large. 

Since both sides are invariant under $x\mapsto kx$ for any $k\in K$, 
we may assume that the last $n-m$ rows of $x_{(m)}$ are $0$. 
In this way, we can restrict our attention to the top-left $m\times m$ block of $x$ that we denote by
\[
V:=(V_{i,i'})_{1\leq i,i'\leq m}.
\]
Let further
\[
R:= \begin{pmatrix}
0 & \dots & 0 & -r_1 \\ \vdots & \ddots & \vdots & \vdots \\  0 & \dots & 0 & -r_{m-1} \\ r_1 & \dots & r_{m-1} & 0
\end{pmatrix} \in \mathfrak{so}(m)
\]
with $r_1,\dotsc,r_{m-1}\in\RR$ satisfying $\sum_{j=1}^{m-1} r_j^2=1$. Let us assume that $t\in\RR$ is sufficiently small in terms of $\Omega$. We define
\[
H(t):=(V \exp(tR))_{(m-1)}^T (V \exp(tR))_{(m-1)},
\]
so that $$\det H(t)=\Delta_{m-1}(V\exp (tR)).$$
We write
\[
H(t)=H_0 + H_1 t + H_2 t^2 + O(t^3),\qquad \det H(t)= C_0 + C_1 t + C_2 t^2 + O(t^3),
\]
where
\[H_i:= H^{(i)}(0)/i!,\qquad C_i:= (\det H)^{(i)}(0)/i!\]
and the implied constants depend on $\Omega$ only by compactness. The upshot is that \eqref{eq:maxmax} is equivalent to the following statement: there exists an $R\in\mathfrak{so}(m)$ normalized as above such that
\begin{equation}\label{eq:maxC1C2}
\max(|C_1|,|C_2|) \gg E_{m}(x).
\end{equation}
The rest of the proof is devoted to verifying this statement.

We shall calculate $C_1$ and $C_2$ with the help of Jacobi's formula: for any differentiable function $A:\RR\to\RR^{m\times m}$,
\begin{equation}\label{eq:jacobi-formula}
(\det(A(t)))' = \det(A(t)) \cdot \tr\left(A(t)^{-1}\cdot A'(t) \right) = \tr \left(\adj(A(t)) \cdot A'(t)\right),
\end{equation}
where we occasionally indicate matrix multiplication with $\cdot$ for improved readability.
As a first application, we see that
\begin{equation}\label{eq:C_1}
C_1 = \tr(\adj(H_0)\cdot H_1).
\end{equation}
We claim that the row vectors
\[\br:=\begin{pmatrix}
r_1 & \cdots & r_{m-1}
\end{pmatrix}
\qquad \text{and} \qquad \mbs:=\begin{pmatrix} \langle V_{\cdot,m}, V_{\cdot,1} \rangle & \cdots & \langle V_{\cdot,m}, V_{\cdot,m-1} \rangle \end{pmatrix}\]
satisfy
\begin{equation}\label{eq:H_1}
H_1=\br^T\mbs+\mbs^T\br.
\end{equation}
Indeed, for $1\leq i,i' \leq m-1$ we have
\[
\begin{split}
(H_1)_{i,i'}&=\left.\frac{d}{dt}\langle (V\exp(tR))_{\cdot,i}, (V\exp(tR))_{\cdot,i'} \rangle \right|_{t=0}
\\ & = \left. \frac{d}{dt} \langle V_{\cdot,i}+ V_{\cdot,m} r_{i} t + O(t^2), V_{\cdot,i'} + V_{\cdot,m} r_{i'} t + O(t^2) \rangle \right|_{t=0}
\\ & = r_{i} \langle V_{\cdot,m}, V_{\cdot,i'} \rangle + r_{i'} \langle V_{\cdot,i}, V_{\cdot,m} \rangle.
\end{split}
\]
Now substituting \eqref{eq:H_1} into \eqref{eq:C_1}, using the additivity and the cyclic invariance of the trace and finally the symmetry of $H_0=V_{(m-1)}^T V_{(m-1)}$, we see that
\begin{equation}\label{eq:C_1-result}
C_1=\tr(\adj(H_0) (\br^T\mbs+\mbs^T\br)) = 2 (\mbs \adj(H_0)) \br^T.
\end{equation}

Similarly, we can compute the second derivative using Jacobi's formula \eqref{eq:jacobi-formula}:
\begin{align}
\notag \det(H(t))''|_{t=0} & = \big(\det(H(t)) \cdot \tr(H(t)^{-1} H'(t))\big)'|_{t=0}
\\ \notag & = \det(H(0))\cdot\left(\tr(H(0)^{-1}H'(0))^2 + \tr\big(H(0)^{-1}H''(0)-(H(0)^{-1}H'(0))^2\big)\right)
\\ \label{eq:C_2} & = 2 \det(H_0) \cdot \tr(H_0^{-1} H_2) + O(\|\mbs\|{}^2).
\end{align}
Here we used \eqref{eq:C_1}, \eqref{eq:H_1}, $H'(0)=H_1=O(\|\mbs\|)$, and the identity
\[
(H(t)^{-1})' = -H(t)^{-1} H'(t) H(t)^{-1}
\]
that follows from the Leibniz rule applied to $H(t)H(t)^{-1}$. We claim that
\begin{equation}\label{eq:H_2}
H_2=\langle V_{\cdot,m}, V_{\cdot,m} \rangle \br^T \br - \frac{1}{2} H_0 \br^T \br - \frac{1}{2} \br^T \br H_0.
\end{equation}
Indeed, for any $1\leq i,i'\leq m-1$, we have 
\begin{align*}
(H_2)_{i,i'}&=\frac{1}{2}\left.\frac{d^2}{dt^2}\langle (V\exp(tR))_{\cdot,i}, (V\exp(tR))_{\cdot,i'} \rangle \right|_{t=0}
\\ &= \langle (VR)_{\cdot,i}, (VR)_{\cdot,i'} \rangle + \frac{1}{2} \langle V_{\cdot,i}, (VR^2)_{\cdot,i'} \rangle + \frac{1}{2} \langle (VR^2)_{\cdot,i}, V_{\cdot,i'} \rangle
\\ & = \langle r_{i} V_{\cdot, m}, r_{i'} V_{\cdot,m} \rangle + \frac{1}{2} \langle V_{\cdot,i}, - V r_{i'} \begin{pmatrix}\br & 0 \end{pmatrix}^T \rangle + \frac{1}{2} \langle - V r_{i} \begin{pmatrix}\br & 0 \end{pmatrix}^T, V_{\cdot,i'} \rangle
\\ & = r_{i} r_{i'} \langle V_{\cdot,m},V_{\cdot,m}\rangle - \frac{1}{2} V_{\cdot,i}^T V r_{i'} \begin{pmatrix}\br & 0 \end{pmatrix}^T - \frac{1}{2} r_{i} \begin{pmatrix}\br & 0 \end{pmatrix} V^T V_{\cdot,i'} 
\\ & = \langle V_{\cdot,m},V_{\cdot,m}\rangle r_{i} r_{i'} - \frac{1}{2} ((H_0)_{i,\cdot}) (\br^T \br)_{\cdot,i'} - \frac{1}{2} (\br^T \br)_{i,\cdot} ((H_0)_{\cdot,i'})
\\ & = (\langle V_{\cdot,m},V_{\cdot,m}\rangle \br^T \br)_{i,i'} - \frac{1}{2} ( H_0 \br^T \br)_{i,i'} - \frac{1}{2} (\br^T \br H_0)_{i,i'}.
\end{align*}
Substituting \eqref{eq:H_2} into \eqref{eq:C_2}, we get that
\begin{align}
\notag C_2 & = \det(H_0) \cdot \br (\langle V_{\cdot,m},V_{\cdot,m}\rangle H_0^{-1} -\id_{m-1})\br^T + O(\|\mbs\|{}^2)
\\ \label{eq:C_2-result} & = \det(H_0) \cdot \br H_0^{-1/2} \left( \langle V_{\cdot,m},V_{\cdot,m} \rangle \id_{m-1} - H_0\right) H_0^{-1/2} \br^T + O(\|\mbs\|{}^2).
\end{align}

Let $\delta>0$ be a sufficiently small constant depending only on $\Omega$.
First assume that
\[
|\langle V_{\cdot,m}, V_{\cdot,i'} \rangle|>\delta^2 E_{m}(x) \qquad \text{for some} \qquad 1\leq i'\leq m-1.
\]
Then the unit vector $\br:=\mbs\adj(H_0)/\|\mbs\adj(H_0)\|$ yields \eqref{eq:maxC1C2} via \eqref{eq:C_1-result}:
\[
C_1=2\|\mbs\adj(H_0)\|\gg \delta^2 E_{m}(x).
\]
Now assume that
\[
|\langle V_{\cdot,m}, V_{\cdot,i'} \rangle| \leq \delta^2 E_{m}(x), \qquad 1\leq i'\leq m-1,
\]
and
\[
|\langle V_{\cdot,i}, V_{\cdot,i'} \rangle| > \delta E_{m}(x) \qquad \text{for some} \qquad 1\leq i<i'\leq m-1.
\]
Then the symmetric matrix $\langle V_{\cdot,m},V_{\cdot,m} \rangle \id_{m-1} - H_0$ inside \eqref{eq:C_2-result} has a (non-diagonal) entry of absolute value $\gg \delta E_{m}(x)$, hence \eqref{eq:C_2-result} gives for a suitable unit row vector $\br$ that 
\[
|C_2|\geq \alpha_1 \delta E_{m}(x) - \alpha_2 \|\mbs\|_\infty \geq (\alpha_1 \delta - \alpha_2 \delta^2) E_{m}(x)
\]
with some $\alpha_1,\alpha_2>0$ depending only on $\Omega$. 
Hence if $\delta>0$ is small enough in terms of $\Omega$, we arrive at \eqref{eq:maxC1C2} again. Finally, assume that
\[
|\langle V_{\cdot,i},V_{\cdot,i'} \rangle| \leq \delta E_{m}(x), \qquad 1\leq i<i'\leq m.
\]
This implies that
\[
\det(V^T V)=\langle V_{\cdot,1},V_{\cdot,1} \rangle \dotsb \langle V_{\cdot,m},V_{\cdot,m} \rangle + O(\delta E_{m}(x)).
\]
We can use these two bounds to estimate (cf. \eqref{eq:Em-def})
\[
E_{m}(x)=\|V^TV - \det(V^T V)^{1/m}\id_m\|_{\infty}.
\]
The result is a kind of converse to \eqref{eq:E_m-related-to-column-lengths}, in the current situation:
\begin{align*}
E_{m}(x) & \ll \max_{1\leq i\leq m} \left| \langle V_{\cdot,i}, V_{\cdot,i} \rangle - (\langle V_{\cdot,1},V_{\cdot,1} \rangle \dotsb \langle V_{\cdot,m},V_{\cdot,m} \rangle)^{1/m} \right|
\\ & \ll \max_{1\leq i\leq m} \left| \langle V_{\cdot,i}, V_{\cdot,i} \rangle - \langle V_{\cdot,m}, V_{\cdot,m} \rangle \right|.
\end{align*}
Then the symmetric matrix $\langle V_{\cdot,m},V_{\cdot,m} \rangle \id_{m-1} - H_0$ inside \eqref{eq:C_2-result} has a (diagonal) entry of absolute value $\gg E_{m}(x)$, hence \eqref{eq:C_2-result} gives for a suitable unit row vector $\br$ that 
\[
|C_2|\geq \alpha_3 E_{m}(x) - \alpha_4 \|\mbs\|_\infty \geq (\alpha_3 - \alpha_4 \delta) E_{m}(x)
\]
with some $\alpha_3,\alpha_4>0$ depending only on $\Omega$. 
Hence if $\delta>0$ is small enough in terms of $\Omega$, the inequality \eqref{eq:maxC1C2} holds as in the earlier two cases.
\end{proof}

\subsection{A support-based bound}\label{sec43}
We now turn our attention to the task of analyzing the size of the spherical trace function $\psi_{\tau}^U(x)$, which is given by the integral representation \eqref{eq:formula_nonminimal_spherical_trace_function} in terms of the trace function $T_\tau^{\sigma'}(x)$. The key result in this subsection is Lemma~\ref{lemma:thm3-big-ell}, which gives a bound for $\psi_{\tau}^{U}(x)$ by sacrificing cancellation in the integral over $K$ (as in the bound \eqref{tracefunctionbound} in terms of the $(2,1)$-norm of $\hat F_{\tau}^{\sigma'}(m'(xk))$) and using the fact that the integrand is often small. This is useful for large $\bm \ell$. Recall our earlier convention \eqref{eq:ordering-ell-s} and the definition \eqref{eq:D-def}. 

\begin{lemma} If $r\geq 1$ in \eqref{Mprimedecomposition}, then
\begin{equation}\label{eq:lemma5bound}
\|\hat F_{\tau}^{\sigma'} (m'(x))\|_{2,1}\leq \dim \tau \cdot \big(1 - D(x)^2/4 \big)^{\ell_1/2},\qquad x\in G.
\end{equation}
\end{lemma}

\begin{remark} By \eqref{tracebound}, the left-hand side equals $\|F_{\tau}^{\sigma'}(x)\|_{2,1}$, hence it provides an upper bound for the trace $T_{\tau}^{\sigma'}(x)$ of $F_{\tau}^{\sigma'}(x)$.
\end{remark}

\begin{proof} By the discussion below \eqref{tracebound}, the function $\|\hat F_{\tau}^{\sigma'} (m'(\cdot))\|_{2,1}$ is left $K$-invariant and right $A'N'$-invariant. The same is obviously true for the function $D$, whence we may assume that
\[
x=m'(x)=\diag\left(\begin{pmatrix}a_1 & b_1 \\ c_1 & d_1 \end{pmatrix},\dotsc,\begin{pmatrix}a_r & b_r \\ c_r & d_r \end{pmatrix},\pm 1,\dotsc,\pm 1\right) \in M'.
\]
Using \eqref{hatFtausigma2}, the multiplicativity of the $\|\cdot\|_{2,1}$ norm under tensor products, and \eqref{basicphi}, we see
\[
\|\hat F_{\tau}^{\sigma'} (x)\|_{2,1}\leq \dim\tau \cdot \prod_{p=1}^r \left(\frac{4}{a_p^2+b_p^2+c_p^2+d_p^2+2}\right)^{\ell_p/2}.
\]
The denominators here are at least $4$ by the determinant condition $a_pd_p-b_pc_p=\pm 1$. Therefore, dropping all but the first factor, it follows that
\[
\|\hat F_{\tau}^{\sigma'} (x)\|_{2,1} \leq \dim \tau \cdot \left(\frac{4}{a_1^2+b_1^2+c_1^2+d_1^2+2}\right)^{\ell_1/2}.
\]
Hence it suffices to show that
\[\frac{4}{a_1^2+b_1^2+c_1^2+d_1^2+2}\leq 1-\frac{1}{4}\left(\frac{a_1^2+c_1^2-b_1^2-d_1^2}{a_1^2+b_1^2+c_1^2+d_1^2}\right)^2.\]

Let us introduce the notation
\[u:=a_1^2+c_1^2,\qquad v:=b_1^2+d_1^2.\]
Then
\[uv=(a_1b_1+c_1d_1)^2+(a_1d_1-b_1c_1)^2\geq 1,\]
and we are left with proving that
\begin{equation}
\label{conclusion}
\frac{1}{4}\left(\frac{u-v}{u+v}\right)^2\leq\frac{u+v-2}{u+v+2}.
\end{equation}
Let $\lambda\in[0,1)$ be the number on the right-hand side. Then from $uv\geq 1$ we obtain that
\[\frac{u+u^{-1}-2}{u+u^{-1}+2}\leq\lambda\qquad\text{and}\qquad \frac{v+v^{-1}-2}{v+v^{-1}+2}\leq\lambda,\]
which in turn implies that
\[\frac{1-\sqrt{\lambda}}{1+\sqrt{\lambda}}\leq u,v\leq\frac{1+\sqrt{\lambda}}{1-\sqrt{\lambda}}.\]
Therefore,
\[\left|\frac{u-v}{u+v}\right|\leq\frac{(1+\sqrt{\lambda})^2-(1-\sqrt{\lambda})^2}{(1+\sqrt{\lambda})^2+(1-\sqrt{\lambda})^2}\
=\frac{2\sqrt{\lambda}}{1+\lambda}\leq 2\sqrt{\lambda},\]
and the required conclusion \eqref{conclusion} follows.
\end{proof}

We will also make use of the following real-analytic fact.
\begin{lemma}\label{lemma:from-real-analysis} Assume $f:I\to\RR$ is a $j$ times differentiable real-valued function on a real interval $I$ such that $|f^{(j)}(t)|\geq \Delta$ with some $\Delta>0$. Then for any $h\geq 0$ we have 
\[
\lambda(\{t\in I: |f(t)| < h \}) \leq (2^{j+1}-2) h^{1/j} \Delta^{-1/j},
\]
where $\lambda$ denotes the Lebesgue measure.
\end{lemma}
\begin{proof} If $g$ is a differentiable function on an interval such that either $g'\geq B$ or $g'\leq -B$ with some $B>0$, then $g$ is increasing or decreasing, respectively. Given $C > 0$, one can easily see (e.g., via Lagrange's mean value theorem) that $|g|\geq CB$ holds outside a subinterval of length at most $2C$.

We apply this with $C:=h^{1/j} \Delta^{-1/j}$ successively for the sequence
\[
g:=f^{(i)},\qquad B:=h C^{-i-1}, \qquad i=j-1,\dotsc,0,
\]
on the intervals where $|g'|\geq B$ holds (which is $I$ for $i=j-1$). By induction,
\[
|f^{(i)}|\geq h C^{-i}, \qquad i=j-1,\dotsc,0,
\]
holds outside the union of at most $2^{j-i}-1$ intervals, each of length at most $2C$. For $i=0$, this implies the conclusion of the lemma.
\end{proof}

Let $\mu_K$ denote the Haar probability measure of $K$.
\begin{lemma}\label{lemma:D-is-often-large} For any $x\in\Omega\setminus K$, and any $h>0$, we have
\[ \mu_K\left(\{k\in K: |D(xk)| < h\}\right) \ll h^{1/3}\cdot\dist(x,K)^{-1/3}. \]
\end{lemma}

\begin{proof} We shall assume that $\dist(x,K)\geq h$, for otherwise the statement is obvious. Let $c_1$ be as in Lemma~\ref{lemma:D-derivatives}. By the compactness of $K$, we may find $k_1,\dotsc,k_m\in K$ with $m:=m(\Omega)$ such that (cf. \eqref{defU})
\[
K=\bigcup_{i=1}^m k_i U(c_1).
\]
Let us fix $1\leq i\leq m$ for a moment. We shall apply Lemma~\ref{lemma:D-derivatives} with $xk_i$ in the role of $x$, and we shall abbreviate
\[
D_i(\bt):=D\left(xk_i \exp\left(\sum_{l=1}^{\dim K} t_l Y_l\right)\right),\qquad \bt:=(t_1,\dotsc,t_{\dim K}).
\]
By \eqref{eq:lebesgue-haar} and Fubini's theorem,
\begin{align*}
\int_{k_i U(c_1)} \mathbf{1}_{|D(xk)|<h}\,dk & \ll \int_{[-c_1,c_1]^{\dim K}} \mathbf{1}_{|D_i(\bt)|<h}\, d\bt \\ & = \int_{-c_1}^{c_1} \dotsc \int_{-c_1}^{c_1} \lambda(\{t_1\in[-c_1,c_1]:|D_i(\bt)|<h\})\, dt_2 \dotsc \, dt_{\dim K}.
\end{align*}
Then, combining Lemmata~\ref{lemma:D-derivatives}~and~\ref{lemma:from-real-analysis}, and using that $\dist(x,K)\geq h$,
\[
\int_{k_i U(c_1)} \mathbf{1}_{|D(xk)|<h}\,dk \ll \max_{1\leq j\leq 3} \left(h^{1/j} \cdot \dist(x,K)^{-1/j}\right)=h^{1/3}\cdot\dist(x,K)^{-1/3}.
\]
Finally, summing up these bounds, the result follows.
\end{proof}

\begin{lemma}\label{lemma:thm3-big-ell} For any $x\in\Omega$, we have
\begin{equation}\label{eq:thm3-big-ell}
\psi_{\tau}^U(x)\ll_{\Re\mu'}\dim\tau\cdot\left(1+\|\bm\ell\|_{\infty}^{1/2}\cdot\dist(x,K)\right)^{-1/3}.
\end{equation}
\end{lemma}
\begin{proof} Let $r$ be as in \eqref{Mprimedecomposition}. We shall assume that 
\[r\geq 1\qquad\text{and}\qquad\|\bm\ell\|_{\infty}^{1/2}\cdot\dist(x,K)\geq 1,\]
for otherwise we are done by the baseline bound \eqref{eq:psi-trivial-bound}. Combining \eqref{tracefunctionbound} and \eqref{eq:lemma5bound}, we see that
\[\psi_\tau^{U}(x)\ll_{\Re\mu'} \dim\tau \cdot \int_K \left(1-D(xk)^2/4\right)^{\ell_1/2}\,dk.\]
We decompose the integral according to dyadic ranges for $D(xk)$. More precisely, we apply Lemma~\ref{lemma:D-is-often-large} to conclude for every $0<h\leq 1$ that
\[
\int_{h/2\leq |D(xk)|<h}\left(1-D(xk)^2/4\right)^{\ell_1/2}\,dk
\ll \exp(-h^2\ell_1/32)\cdot h^{1/3} \cdot \dist(x,K)^{-1/3}.
\]
We add up these bounds for all $h=2^{-i}$ with $i$ a nonnegative integer. The maximum of the right-hand side is achieved at $h\asymp\ell_1^{-1/2}$, while the terms decay at least exponentially as we move away from this value. Upon noting that $\ell_1 = \|\bm\ell\|_{\infty}$, the desired bound follows.
\end{proof}

\subsection{Preliminaries to the stationary phase analysis} 
We shall assume that $\Im\mu'\neq 0$, for otherwise the conclusion of Theorem~\ref{thm3} is clear by Lemma~\ref{lemma:thm3-big-ell}. We start by removing the ambiguity in the generalized Iwasawa decomposition \eqref{iw3} by requiring from now on that   $m'(x)$ is upper-triangular with positive diagonal entries.
\begin{lemma}\label{lemma:iwasawa-coordinates-are-analytic} The functions $\kappa'$, $m'$, $\exp(H')$, $n'$ are real-analytic.
\end{lemma}
\begin{proof} It is well-known that the standard Iwasawa coordinates in \eqref{iw1} are analytic in a broad generality (see, e.g., \cite[p.~181]{Helgason-groups-and-geometric-analysis}). In our particular situation $G=\SL_n^{\pm}(\RR)$, one can easily check the geometric interpretations
\[
\exp(H(x))=\diag\left(\sqrt{\Delta_1(x)}, \sqrt{\frac{\Delta_2(x)}{\Delta_1(x)}}, \dotsc, \sqrt{\frac{\Delta_n(x)}{\Delta_{n-1}(x)}} \right),
\]
where $\Delta_j(x)$, defined in \eqref{deltam}, 
stands for the square of the $j$-dimensional volume of the $j$-dimensional parallelepiped spanned by the first $j$ columns of $x$; and
\[
(n(x))_{i,j}= \frac{\det\bigl(x_{(i)}^T x_{(i-1),j}\bigr)}{\Delta_i(x)},\qquad 1\leq i<j\leq n,
\]
where $x_{(i)}$ and $x_{(i-1),j}$ stand for the $n\times i$ matrices formed by the first $i$ columns and the first $i-1$ plus the $j$-th columns of $x$. Indeed, the first equation holds for $x\in A$, and it is left $K$- and right $N$-invariant; while the second equation holds for $x\in AN$ and is left $K$-invariant. Together they prove the analyticity of $\exp(H)$ and $n$, which is then inherited to $\kappa(x)=x n(x)^{-1} \exp(H(x))^{-1}$.

In order to switch to $x=\kappa'(x)m'(x)\exp(H'(x))n'(x)$, we first observe that our convention implies $\kappa'(x)=\kappa(x)$. Then, writing $\exp(H(x))n(x)$ in the form
\[
\exp(H(x))n(x)=\begin{pmatrix} \boxed{\begin{matrix} a_1 & b_1 \\ & d_1 \end{matrix}} & * & \cdots & * \\ & \ddots & & & & \\ & & \boxed{\begin{matrix} a_r & b_r \\ & d_r \end{matrix}} & & \\ & & & \ddots \end{pmatrix}
\]
implies that
\[
n'(x)=\underbrace{\diag \left(\begin{pmatrix} a_1 & b_1 \\ & d_1 \end{pmatrix}^{-1},\dotsc,\begin{pmatrix} a_r & b_r \\ & d_r \end{pmatrix}^{-1},\dotsc\right)}_{\in M'A'} \exp(H(x)) n(x).
\]
Here $\diag(\dotsc)$ has $r$ $2\times 2$ blocks and $s$ $1\times 1$ blocks. Also, $\exp(H'(x))$ is obtained from $\exp(H(x))$ by replacing each diagonal entry $a_j,d_j$ with their geometric mean $\sqrt{a_jd_j}$; for later reference, we record that
\begin{equation}\label{eq:explicit-H'-0}
\begin{split}
\exp(H'(x)) = \diag \Biggl( \sqrt[4]{\Delta_2(x)}, \sqrt[4]{\Delta_2(x)}, \dots, \sqrt[4]{\frac{\Delta_{2r}(x)}{\Delta_{2r-2}(x)}}, \sqrt[4]{\frac{\Delta_{2r}(x)}{\Delta_{2r-2}(x)}} , & \\  \sqrt{\frac{\Delta_{2r+1}(x)}{\Delta_{2r}(x)}}, \dots, \sqrt{\frac{\Delta_n(x)}{\Delta_{n-1}(x)}} & \Biggr).
\end{split}
\end{equation}
Then $m'(x)=\kappa'(x)^{-1} x n'(x)^{-1} \exp(H'(x))^{-1}$.
\end{proof}

Our next lemma shows that each matrix coefficient of $\tau$ is a real-analytic function on $K$ whose derivatives can be bounded conveniently. We introduce the standard notation
\[
d\tau(X)\bv:= \left.\frac{d}{dt}\tau(\exp(tX))\bv \right|_{t=0},\qquad X\in\mfk,\quad \bv\in V_{\tau},
\]
for the derived action of $\mfk$ on $V_{\tau}$.

\begin{lemma}\label{lemma:matrix-coefficients-derivatives} For any direction $X\in\mfk$, we have
\[
\|d\tau(X) \|_{\mathrm{op}} \ll 1+\|\bm\ell\|_{\infty},
\]
where $\|\cdot\|_{\mathrm{op}}$ stands for the operator norm.
\end{lemma}
\begin{proof} Since $K$ is compact, $X$ lies in a maximal abelian subalgebra $\mft$ of $\mfk$. More specifically, since $\tr(X^2)=-2$, there are real numbers $\theta_1,\dotsc,
\theta_{\lfloor n/2 \rfloor}$ satisfying $\sum_{1\leq j\leq n/2}|\theta_j|^2=1$ such that for all $t\in\RR$, $\exp(tX)$ is conjugate to
\[
\begin{pmatrix}
k(t\theta_1) & & \\ & \ddots & \\ & & k(t\theta_{\frac{n}{2}})
\end{pmatrix}
\qquad \text{or} \qquad
\begin{pmatrix}
k(t\theta_1) & & & \\ & \ddots & & \\ & & k(t\theta_{\frac{n-1}{2}}) & \\ & & & 1
\end{pmatrix},
\]
for $n$ even or odd, respectively, where $k(t)$ is as in \eqref{rotationmatrix}.

Let $\bv_1,\dotsc,\bv_{\dim \tau}\in V_{\tau}$ be an orthonormal basis of weight vectors corresponding to $\mft$.
Then for any $1\leq u\leq \dim\tau$, there is a weight $(m_1,\dotsc,m_{\lfloor n/2 \rfloor})\in\ZZ^{\lfloor n/2 \rfloor}$ of $\tau$ such that
\[
\tau(\exp(tX)) \bv_u = \exp\Bigg(it\sum_{j=1}^{\lfloor n/2 \rfloor} m_j \theta_j \Bigg) \bv_u,\qquad t\in\RR.
\]
Therefore,
\[d\tau(X) \bv_u = i\Bigg(\sum_{j=1}^{\lfloor n/2 \rfloor} m_j \theta_j \Bigg) \bv_u,\]
and $\|d\tau(X)\|_{\mathrm{op}}$ is the size of the largest coefficient (eigenvalue) that occurs on the right-hand side. However, the
highest weight of $\tau$ is \eqref{eq:highest-weight}, hence $|m_j|\leq 1+\|\bm\ell\|_{\infty}$ for all $1\leq j\leq \lfloor n/2 \rfloor$. The desired bound follows.
\end{proof}
In fact, we shall need a more general version of Lemma~\ref{lemma:matrix-coefficients-derivatives}, to be stated below. It will be helpful to use the following standard notations:
\[\bt = (t_1, \dotsc, t_j),\qquad \alpha=(\alpha_1,\dotsc,\alpha_j), \qquad |\alpha|:=\alpha_1+\dotsb+\alpha_j,\]
and
\[
\partial^{\alpha}:=\frac{\partial^{|\alpha|}}{\partial t_1^{\alpha_1}\dotsb \partial t_j^{\alpha_j}}.
\]
\begin{lemma}\label{lemma:matrix-coefficients-derivatives-2}
Let $k:\RR^j\to K$ be a smooth function, and let $C\subseteq\RR^j$ be a compact set. Then for any multi-index $\alpha\in\ZZ_{\geq 0}^j$,
\[
\sup_{\bt\in C} \| \partial^\alpha\tau(k(\bt)) \|_{\mathrm{op}} \ll_{M,\alpha} (1+\|\bm\ell\|_{\infty})^{|\alpha|},
\]
where
\[
M: = \sup\left\{|\partial^{\beta}k(\bt)|:\bt\in C,\,|\beta|\leq|\alpha|\right\}.
\]
\end{lemma}
\begin{proof}
Applying the operator $\partial_i:=\partial/\partial t_i$ to the identity $k(\bt)^T k(\bt)=\id_n$, we see that the $i$-th left logarithmic derivative
\[L_i(\bt):=k(\bt)^{-1}\partial_i k(\bt)\]
takes its values from the Lie algebra of $K$, hence it is a smooth function $\RR^j\to\mfk$. The identity $\partial_i k(\bt)=k(\bt)L_i(\bt)$ then implies that the following two curves in $K$ are tangent to each other at $h=0$:
\[h\mapsto k(\bt+h\mathbf{e}_i)\qquad\text{and}\qquad h\mapsto k(\bt)\exp(hL_i(\bt)).\]
Hence the images of these curves under $\tau$,
\[h\mapsto \tau(k(\bt+h\mathbf{e}_i))\qquad\text{and}\qquad h\mapsto \tau(k(\bt))\tau(\exp(hL_i(\bt))),\]
are also tangent to each other at $h=0$. This means that
\[\partial_i\tau(k(\bt))=\tau(k(\bt))\,d\tau(L_i(\bt)),\qquad 1\leq i\leq j.\]
It follows by induction on $|\alpha|$ that $\partial^\alpha\tau(k(\bt))$ is a finite linear combination of functions of the form
\[\tau(k(\bt))\,d\tau(X_1(\bt))\dotsb\,d\tau(X_r(\bt)),\]
where $0\leq r\leq|\alpha|$ and each $X_i:\RR^j\to\mfk$ is smooth with
\[\sup\bigl\{\|X_i(\bt)\|:\bt\in C\bigr\}\ll_{M,\alpha}1.\]
Hence Lemma~\ref{lemma:matrix-coefficients-derivatives} implies the stated estimate.
\end{proof}
For convenience, we rewrite \eqref{eq:formula_nonminimal_spherical_trace_function} with the help of \eqref{eq:T-sigma'} as
\begin{equation}\label{eq:formula_nonminimal_spherical_trace_function-rewritten}
\psi_\tau^{U}(x)
=\int_K T_\tau^{\sigma'}(k^{-1}xk)\,e^{(\Re\mu'-\rho')(H'(xk))}\,e^{i\Im\mu'(H'(xk))}\,dk.
\end{equation}
We introduce the notation
\begin{equation}\label{eq:A(x,k,mu')}
A(x,k,\mu'):= T_\tau^{\sigma'}(k^{-1}xk)\,e^{(\Re\mu'-\rho')(H'(xk))}.
\end{equation}
In the current and the next subsections we apply a stationary phase argument to estimate $\psi_\tau^{U}(x)$ for large $\Im\mu'$.

\begin{lemma}\label{lemma:sobolev} For any $x\in\Omega$ and any nonnegative integer $j$ we have
\begin{equation}\label{eq:lemma12bound}
\left\|k\mapsto A(x,k,\mu')\right\|_{S_j}\ll_{j,\Re\mu'} \dim \tau \cdot (1+\|\bm\ell\|_{\infty})^{j},
\end{equation}
where $S_j$ stands for the Sobolev norm of order $j$ with respect to any orthonormal basis of $\mfk$.
\end{lemma}
\begin{proof} 
To bound the $S_j$ Sobolev norm of $k \mapsto A(x, k, \mu')$, it suffices to bound its mixed derivatives along any sequence of directions $X_1,\dotsc,X_j \in \mfk$ (chosen from an orthonormal basis). Since $\Omega$ is bi-$K$-invariant and $H'$ is left $K$-invariant, replacing $x$ with $k^{-1}xk \in \Omega$ allows us to evaluate the derivatives at the identity element $k = \id$. For the proof, we set
\[
E(\bt):=\exp(t_1X_1)\dotsb\exp(t_jX_j),
\]
and consider the following differential operators acting on $A$:
\[ \left. \partial^{\alpha} A(x, E(\bt), \mu') \right|_{\bt=\mathbf{0}}, \qquad |\alpha|\leq j.\]
Using the definitions \eqref{eq:A(x,k,mu')} and \eqref{Fdef}, we can write the function as
\[
A(x, E(\bt),\mu') = \tr \left(e^{(\Re\mu' - \rho')H'(x E(\bt))} \times \tau(\kappa'(E(\bt)^{-1} x E(\bt))) \times \hat{F}_\tau^{\sigma'}(m'(x E(\bt))) \right).
\]
By the Leibniz rule, it suffices to show that the relevant derivatives of each factor (those separated by ``$\times$'') are appropriately bounded in suitable matrix norms at the origin $\bt=\mathbf{0}$.

For the initial factor, Lemma~\ref{lemma:iwasawa-coordinates-are-analytic} and repeated application of the chain rule yield
\begin{equation}\label{eq:sobolev-proof-1}
\left. \partial^{\beta} e^{(\Re\mu' - \rho')H'(x E(\bt) )} \right|_{\bt=\mathbf{0}} \ll_{\beta,\Re\mu'} 1.
\end{equation}
For the middle factor, Lemma~\ref{lemma:iwasawa-coordinates-are-analytic} and Lemma~\ref{lemma:matrix-coefficients-derivatives-2} immediately give that
\begin{equation}\label{eq:sobolev-proof-2}
\left\| \left. \partial^{\beta} \tau(\kappa'(E(\bt)^{-1} x E(\bt))) \right|_{\bt=\mathbf{0}} \right\|_{\mathrm{op}} 
\ll_{\beta} (1+\|\bm\ell\|_\infty)^{|\beta|}.
\end{equation}
Regarding the third factor, we observe first that $m'(xE (\bt))$ is of the shape
\[
m'(x E(\bt)) = \diag\left( \begin{pmatrix} a_1(\bt) & b_1(\bt) \\ 0 & d_1(\bt) \end{pmatrix},\dotsc,\begin{pmatrix} a_r(\bt) & b_r(\bt) \\ 0 & d_r(\bt) \end{pmatrix},1,\dotsc,1 \right), 
\]
by our convention on $m'$. Each $2\times 2$ block has a positive diagonal and determinant $1$. All the matrix entries here are real-analytic by Lemma~\ref{lemma:iwasawa-coordinates-are-analytic}, and then \eqref{hatFtausigma2} gives
\[
\hat F_\tau^{\sigma'}(m'(xE(\bt))) = \frac{\dim \tau}{2^r} \bigotimes_{p=1}^r \varphi_{\omega_{\ell_p}}^{\delta_{\ell_p}} \left(\begin{pmatrix} a_p(\bt) & b_p(\bt) \\ 0 & d_p(\bt) \end{pmatrix}\right),
\]
where the right-hand side is understood as a diagonal $2^r\times 2^r$ block of a $\dim \tau \times \dim \tau$ matrix (cf. the comment after \eqref{hatFtausigma2}). In fact, by \eqref{basicphi}, the entire matrix is diagonal, and the first $2^r$ diagonal entries are all the expressions of the shape
$$ \frac{\dim \tau}{2^r} \prod_{p=1}^r \left( \frac{2}{a_p(\bt) \pm i b_p(\bt) + d_p(\bt)} \right)^{\ell_p}. $$
Hence by the Leibniz rule, we obtain readily that
\begin{equation}\label{eq:sobolev-proof-3}
\left\| \left. \partial^{\beta} \hat{F}_\tau^{\sigma'}(m'(x E(\bt))) \right|_{\bt=\mathbf{0}} \right\|_{2,1} \ll_{\beta} \dim\tau \cdot (1+\|\bm\ell\|_\infty)^{|\beta|}.
\end{equation}

Finally we need a straightforward fact from linear algebra: if $T,U$ are complex square matrices of the same size, then
\[
|\tr(TU)| \leq \|TU\|_{2,1} \leq \|T\|_{\mathrm{op}} \cdot \|U\|_{2,1}.
\]
Combining this with \eqref{eq:sobolev-proof-1}, \eqref{eq:sobolev-proof-2}, \eqref{eq:sobolev-proof-3}, the proof is complete.
\end{proof}

Now we focus our attention on the function $k\mapsto e^{i\Im\mu'(H'(xk))}$. By \eqref{eq:explicit-H'-0},
\[
\begin{split}
H'(xk) = \log \diag \Biggl( \sqrt[4]{\Delta_2(xk)}, \sqrt[4]{\Delta_2(xk)}, \dots, \sqrt[4]{\frac{\Delta_{2r}(xk)}{\Delta_{2r-2}(xk)}}, \sqrt[4]{\frac{\Delta_{2r}(xk)}{\Delta_{2r-2}(xk)}} , & \\  \sqrt{\frac{\Delta_{2r+1}(xk)}{\Delta_{2r}(xk)}}, \dots, \sqrt{\frac{\Delta_n(xk)}{\Delta_{n-1}(xk)}} & \Biggr).
\end{split}
\]
Then, with the conventions $\Delta_0(xk):=1$ and $\mu'_{n-r+1}:=0$, we have 
\begin{align}
\notag \Im\mu'(H'(xk))&=\sum_{j=1}^{r} \frac{1}{2}\Im\mu'_j(\log \Delta_{2j}(xk) - \log\Delta_{2j-2}(xk)) \\ \notag & \qquad + \sum_{j=r+1}^{n-r} \frac{1}{2}\Im\mu_j'(\log \Delta_{r+j}(xk) - \log \Delta_{r+j-1}(xk))
\\ \notag & = \sum_{\substack{1\leq j\leq 2r \\ \text{$j$ even}}} \frac{1}{2}(\Im\mu'_{j/2}-\Im\mu'_{j/2+1}) \log \Delta_j(xk)
\\ \notag & \qquad + \sum_{2r<j\leq n} \frac{1}{2}(\Im\mu'_{j-r}-\Im\mu'_{j-r+1}) \log \Delta_j(xk)
\\ \label{eq:explicit-mu'-H'} & = \sum_{j=1}^n \tilde{\mu}_j \log \Delta_j(xk),
\end{align}
where we set
\begin{equation}\label{eq:mu-linear-transform}
\tilde{\mu}_j:=\begin{cases} 0, \qquad & 1\leq j\leq 2r,\ \text{$j$ odd}, \\ \frac{1}{2}(\Im\mu'_{j/2}-\Im\mu'_{j/2+1}), & 1\leq j\leq 2r,\ \text{$j$ even}, \\ \frac{1}{2}(\Im\mu'_{j-r}-\Im\mu'_{j-r+1}), & 2r<j\leq n. \end{cases}
\end{equation}
We record the following two consequences of \eqref{eq:mu'-sum-0} and \eqref{eq:mu-linear-transform}:
\begin{equation}\label{eq:mu-linear-transform-consequence1}
\sum_{j=1}^n w_j \tilde{\mu}_j = 0 \qquad \text{for some} \qquad w_j\asymp 1,
\end{equation}
and
\begin{equation}\label{eq:mu-linear-transform-consequence2}
\|\tilde{\mu}\|_{\infty}\asymp \|\Im \mu'\|_{\infty}.
\end{equation}
In particular, $\tilde{\mu}$ is nonzero, because $\Im \mu'$ was assumed to be nonzero.

We summarize \eqref{eq:formula_nonminimal_spherical_trace_function-rewritten}, \eqref{eq:A(x,k,mu')}, \eqref{eq:explicit-mu'-H'} and \eqref{eq:mu-linear-transform} as
\begin{equation}\label{eq:stationary-phase-prep-upshot}
\psi_\tau^{U}(x)=\int_K A(x,k,\mu') \cdot \exp\left(i\|\tilde{\mu}\|_{\infty}\cdot \sum_{j=1}^n \frac{\tilde{\mu}_j}{\|\tilde{\mu}\|_{\infty}} \log \Delta_j(xk)\right) \,dk.
\end{equation}
The coefficients $\tilde{\mu}_j/\|\tilde{\mu}\|_{\infty}$ lie in the interval $[-1,1]$.

\subsection{Stationary phase analysis}\label{sec45}
The guiding principle of the forthcoming argument is the following: unless $x$ is very close to $K$, we may typically find, for any $xk$ ($k\in K$), a direction $X\in\mfk$ in such a way that the $\Delta_j$'s are varying fast along $xk\exp(tX)$, a fact already encapsulated in Lemma~\ref{lemma:Delta-derivatives}. Hence there should be a substantial cancellation coming from the phase encoded in $\tilde{\mu}$.

However, there are several technical issues to overcome. First of all, moving in a certain direction $X$ modifies not only a certain $\Delta_j$, but potentially many others, and it is hard to guarantee that, weighted by $\tilde{\mu}$, these variations do not cancel each other. Secondly, we might have a small degree of freedom choosing $j$, potentially only a small number of $\tilde{\mu}_j$'s are on the scale of $\|\tilde{\mu}\|_{\infty}$, and we are essentially given which $\Delta_j$'s are to be perturbed. That is, we are limited in which directions we can use. A convenient idea to overcome these two difficulties is to choose the smallest possible $2\leq m\leq n$ with the property that $\tilde{\mu}_{m-1}$ is not too small in terms of $\|\tilde{\mu}\|_{\infty}$, and use directions from $\mfk_{m}$. Indeed, such a move only alters $\Delta_{m-1},\dotsc,\Delta_1$, and the $\tilde{\mu}$-weighted variation comes from $\tilde{\mu}_{m-1}\Delta_{m-1}+\dotsb+\tilde{\mu}_1\Delta_1$, where only the first term matters due to the dominating choice of $\tilde{\mu}_{m-1}$. In this first term, a big variation is guaranteed by Lemma~\ref{lemma:Delta-derivatives}, unless the first $m$ columns are about the same length and are close to orthogonal to each other. Then the third difficulty kicks in: this exceptional situation may arise. This issue will be handled using Lemma~\ref{lemma:D-is-often-large} by showing that the exceptional set, where we do not expect cancellation in the stationary phase analysis, is small on its own, hence a trivial estimate can be applied. Now let us see this in action.

Let $\alpha:=\alpha(\Omega)>1$ be a large parameter to be chosen later, and let $2\leq m\leq n$ be the smallest number with the property that
\begin{equation}\label{eq:m-choice}
|\tilde{\mu}_{m-1}|\geq \alpha^{m-n-1}\|\tilde{\mu}\|_{\infty}.
\end{equation}
Such an $m$ exists by \eqref{eq:mu-linear-transform-consequence1}, for otherwise $|\tilde\mu_1|,\dotsc,|\tilde\mu_{n-1}|$ would be much smaller than $|\tilde\mu_n|=\|\tilde{\mu}\|_{\infty}$. The usefulness of this definition will become clear in \eqref{stat-phase-lowerbound}. Let $dk'$ be the Haar probability measure on $K_m$, and let
\[ B(x,k,\mu'): = A(x,k,\mu') \cdot \exp\left(i\|\tilde{\mu}\|_{\infty}\cdot \sum_{j=1}^n \frac{\tilde{\mu}_j}{\|\tilde{\mu}\|_{\infty}} \log \Delta_j(xk)\right).\]
For any $L^1$-normalized smooth function $\omega:K_m\to\RR_{\geq 0}$, we can rewrite \eqref{eq:stationary-phase-prep-upshot} in a redundant fashion (with an extra integration over $K_m$) as
\begin{equation}\label{eq:psi-integral-Q_m-K_m}
\psi_\tau^{U}(x)=\int_K B(x,k,\mu')\,dk=\int_{K_m}\int_{K} B(x,kk',\mu')\,\omega(k')\,dk\,dk'.
\end{equation}
For the rest of the proof, we fix $\omega$ such that
\[
\supp(\omega)\subseteq\left\{\exp X: X\in\mfk_m,\,\|X\|<c_2\right\},
\]
where $c_2:=c_2(\Omega)$ is given by Lemma~\ref{lemma:Delta-derivatives}.

Now recall the definition \eqref{eq:Em-def}, and for $p>0$ consider the sets
\[
K^{(m)}_{\leq p}:= \left\{k\in K:\sup_{k'\in K_{m}}|E_{m}(xkk')|\leq p\right\},
\]
and
\[
K^{(m)}_{\geq p}:=\left \{k\in K:\inf_{k'\in K_{m}}|E_{m}(xkk')|\geq p\right\}.
\] 
By \eqref{eq:E_m-K_m-invariance}, the following bound holds with an implied constant depending only on $n$ (not on $\Omega$):
\[
\sup_{k'\in K_{m}}|E_{m}(xkk')| \ll \inf_{k'\in K_{m}}|E_{m}(xkk')|,\qquad k\in K.
\]
Hence there exists a constant $C$ 
depending only on $n$ (not on $\Omega$) such that 
\begin{equation}\label{eq:Q_m-union}
K=K^{(m)}_{\leq Ch} \cup K^{(m)}_{\geq h}, \qquad h>0.
\end{equation}

Let $h>0$ be a parameter to be chosen later. For the proof of Theorem~\ref{thm3} we may assume without loss of generality that $x \not\in K$, for otherwise the result follows from \eqref{eq:psi-trivial-bound}. In \eqref{eq:psi-integral-Q_m-K_m}, we switch the order of integration by Fubini. Then, by \eqref{eq:Q_m-union} and the triangle-inequality, we see that
\begin{equation}\label{eq:stationary-phase-splitted}
\left|\psi_\tau^{U}(x)\right| \leq I\left(K^{(m)}_{\leq Ch}\right) + I\left(K^{(m)}_{\geq h}\right),
\end{equation}
where
\begin{equation}\label{eq:ISintegral}
I(S):= \int_S \left|\int_{K_m} B(x,kk',\mu')\,\omega(k')\,dk'\right| dk,\qquad S\subseteq K.
\end{equation}
In order to bound $I\left(K^{(m)}_{\leq Ch}\right)$, we estimate trivially
\[
B(x,kk',\mu') \ll_{\Re \mu'} \dim\tau,
\]
using \eqref{eq:A(x,k,mu')}, \eqref{tracebound} and \eqref{eq:lemma5bound}. By \eqref{eq:D-def}, \eqref{eq:E_m-K_m-invariance} and \eqref{eq:E_m-related-to-column-lengths}, we see that
\[
D(xk) \ll h,\qquad k\in K^{(m)}_{\leq Ch}.
\]
Hence from Lemma~\ref{lemma:D-is-often-large} we infer that
\[
\mu_{K} \left( K^{(m)}_{\leq Ch} \right) \ll h^{1/3}\cdot\dist(x,K)^{-1/3}.
\]
In the end,
\begin{equation}\label{eq:I-E-small}
I\left(K^{(m)}_{\leq Ch}\right) \ll_{\Re \mu'} \dim \tau \cdot h^{1/3}\cdot\dist(x,K)^{-1/3}.
\end{equation}
We shall bound $I\left(K^{(m)}_{\geq h}\right)$ with the help of Lemma~\ref{lemma:Delta-derivatives}. As a preparation, we observe that the inner integral in \eqref{eq:ISintegral} depends continuously on $k$, hence we can fix a point $k\in K^{(m)}_{\geq h}$ such that
\[I\left(K^{(m)}_{\geq h}\right)\leq\left|\int_{K_m} B(x,kk',\mu')\,\omega(k')\,dk'\right|.\]
We consider the orthonormal basis $(Y_1,\dotsc,Y_{\dim K_{m}})$ of $\mfk_{m}$ provided by Lemma~\ref{lemma:Delta-derivatives} when applied to $xk$ in the role of $x$. We conclude that
\[
I\left(K^{(m)}_{\geq h}\right) \leq
\left|\int_{[-c_2,c_2]^{\dim K_m}} B\left(x,k\exp\left(\sum_{l=1}^{\dim K_{m}} t_l Y_l\right),\mu'\right)\varrho(\bt) \,d\bt \right|,
\]
where $\varrho$ is a smooth density function whose derivatives are bounded in terms of $\Omega$. To analyze the $\bt$-integral by a stationary phase argument, we need a lower bound for some derivatives of 
\[\sum_{i=1}^n \frac{\tilde{\mu}_i}{\|\tilde{\mu}\|_{\infty}}
\log \Delta_i\left(xk\exp\left(\sum_{l=1}^{\dim K_{m}} t_l Y_l\right)\right).\]
We note that the summands corresponding to $i\in\{m,\dotsc,n\}$ are constant in $\bt$. Let $k\in K^{(m)}_{\geq h}$. Using Lemma~\ref{lemma:E-m-derivatives-upper-bound}, Lemma~\ref{lemma:Delta-derivatives}, \eqref{eq:E_m-K_m-invariance} and \eqref{eq:m-choice}, we obtain $j\in\{1,2\}$ and $c_3:=c_3(\Omega)>0$ such that for any $\bt\in[-c_2,c_2]^{\dim K_m}$ we have
\begin{alignat}{2}\label{stat-phase-lowerbound}
\notag & & & \left| \frac{\partial^j}{\partial t_1^j} \sum_{i=1}^n \frac{\tilde{\mu}_i}{\|\tilde{\mu}\|_{\infty}} \log \Delta_i\left(x k \exp\left(\sum_{l=1}^{\dim K_{m}} t_l Y_l\right)\right)\right|\\
\notag & = & & \left| \frac{\partial^j}{\partial t_1^j} \sum_{i=1}^{m-1} \frac{\tilde{\mu}_i}{\|\tilde{\mu}\|_{\infty}} \log \Delta_i\left(x k \exp\left(\sum_{l=1}^{\dim K_{m}} t_l Y_l\right)\right)\right|\\
& \geq & &\ (c_2 \alpha^{m-n-1}-c_3 \alpha^{m-n-2}) E_{m}(x k) \gg h .
\end{alignat}
We choose $\alpha>c_3/c_2$ to justify the last inequality. We then apply \cite[Prop.~5 on p.~342]{St} in combination with Lemmata~\ref{lemma:E-m-derivatives-upper-bound} and \ref{lemma:sobolev} to infer that
\begin{equation}\label{eq:I-E-large}
I\left(K^{(m)}_{\geq h}\right) \ll_{\Re\mu'} \dim \tau \cdot (1+\|\bm{\ell}\|_{\infty})
\left(\|\tilde{\mu}\|_{\infty}^{-1}h^{-1} + \|\tilde{\mu}\|_{\infty}^{-1/2}h^{-1/2}\right).
\end{equation}
Finally, substituting \eqref{eq:I-E-small} and \eqref{eq:I-E-large} into \eqref{eq:stationary-phase-splitted}, we conclude that
\begin{equation}\label{eq:stationary-phase-upshot}
\begin{split}
\psi_\tau^{U}(x) \ll_{\Re\mu'}
& \dim \tau \cdot h^{1/3}\cdot\dist(x,K)^{-1/3}\\
& + \dim \tau \cdot (1+\|\bm{\ell}\|_{\infty})\left(\|\tilde{\mu}\|_{\infty}^{-1}h^{-1} + \|\tilde{\mu}\|_{\infty}^{-1/2}h^{-1/2}\right).
\end{split}
\end{equation}

\subsection{Proof of Theorem~\ref{thm3}}\label{thm3endgame}
We shall assume that $\|U\|_{\infty}>1$ and $x\not\in K$, for otherwise the conclusion follows from \eqref{eq:psi-trivial-bound}. First assume that $\|\bm{\ell}\|_{\infty}\leq\|\mu'\|_{\infty}^{6/17}$, which also implies that $\|\bm{\ell}\|_{\infty}\leq\|\mu'\|_{\infty}$. Then, \eqref{eq:mu-linear-transform-consequence2} and \eqref{eq:stationary-phase-upshot} yield for any $h\gg\|\mu'\|_{\infty}^{-1}$ that
\[\psi_\tau^{U}(x)
\ll_{\Re\mu'} \dim\tau\cdot\left(h^{1/3} + \|\mu'\|_{\infty}^{-5/34}h^{-1/2}\right)\cdot\dist(x,K)^{-1/3}.\]
We choose $h:=\|\mu'\|_{\infty}^{-3/17}$ and conclude that
\[\psi_\tau^{U}(x) \ll_{\Re\mu'} \dim\tau\cdot\|U\|_{\infty}^{-1/17}\cdot\dist(x,K)^{-1/3}.\]
Now assume that $\|\bm{\ell}\|_{\infty}>\|\mu'\|_{\infty}^{6/17}$, which also implies that $\|\bm{\ell}\|_{\infty}\geq 1$. Then Lemma~\ref{lemma:thm3-big-ell} yields that
\[\psi_{\tau}^{U}(x) \ll_{\Re\mu'} \dim \tau \cdot \|\bm\ell\|_{\infty}^{-1/6}\cdot\dist(x,K)^{-1/3},\]
hence also that
\[\psi_{\tau}^{U}(x) \ll_{\Re\mu'} \dim \tau \cdot \|U\|_{\infty}^{-1/17}\cdot\dist(x,K)^{-1/3}.\]

\subsection{A variant for nonminimal \texorpdfstring{$K$}{K}-types}\label{thm3bsubsection}
For the proof of Theorem~\ref{thm1}, we need to estimate the generalized spherical trace function $\psi_\tau^U$ for more general $K$-types $\tau$. This comes with some complications, the first of which is a potentially large multiplicity $[U:\tau]$. The next lemma provides a ``polynomial'' upper bound for this multiplicity. 

\begin{lemma}\label{lemma:multiplicitybound} Let $U=U^{\sigma',\mu'}$ be a generalized principal series representation. Let $\tilde\tau$ be the (unique) minimal $K$-type of $U$, and let $\tilde{\bm\lambda}$ denote its highest weight. If $\tau$ is an arbitrary $K$-type of highest weight $\bm\lambda$, then
\[[U:\tau]\ll\bigl(1+\|\bm\lambda-\tilde{\bm\lambda}\|_\infty\bigr)^{\lfloor n^2/4\rfloor}.\]
\end{lemma}

\begin{proof} We keep using the notations and findings of \S\ref{minimalKtype}. Without loss of generality, $\sigma'$ is of the form \eqref{sigma'}, where the bounds \eqref{eq:ordering-ell-s} and $u\leq t$ are in
place. Then, $\tilde{\bm\lambda}$ is given by \eqref{eq:highest-weight}:
\[\tilde{\bm\lambda}=(\ell_1,\dotsc,\ell_r,\underbrace{1,\dotsc,1}_{\text{$u$ times}},\underbrace{0,\dotsc,0}_{\text{$t-u$ times}}).\]
It follows from \eqref{schur-Hom} that $[U:\tau]$ equals the sum of $[\tau:\omega]$ over the representations $\omega\in\widehat{K'}$ that occur in $\sigma'$. These $\omega$'s are listed by \eqref{eq:omega-tensor}, where $m_p\in\{\ell_p,\ell_p+2,\dotsc\}$. For convenience, we rearrange the last $s$ factors of $\omega$ in \eqref{eq:omega-tensor} as follows:
\[\omega=\omega_{m_1}\boxtimes\dotsb\boxtimes\omega_{m_r}
\boxtimes(\sgn\boxtimes 1)^{\boxtimes u}\boxtimes(1\boxtimes 1)^{\boxtimes(t-u)}\boxtimes 1^{\boxtimes\epsilon},\]
where $\epsilon\in\{0,1\}$ is the parity of $n$. Restricting $\tau$ to the intermediate subgroup
\[K'':=\OO(2)^{r+t}\times O(1)^\epsilon,\]
and then further restricting it to $K'$ reveals that
\[[\tau:\omega]=\sum_{\rho\in\widehat{K''}}[\tau:\rho][\rho:\omega].\]
Let us write $\omega_0$ for the trivial representation of $\OO(2)$. Then it is straightforward to check that $[\rho:\omega]$ above is either zero or one, and it is one if and only if
\[\rho=\omega_{m_1}\boxtimes\dotsb\boxtimes\omega_{m_r}\boxtimes\omega_{m'_1}\boxtimes\dotsb\boxtimes\omega_{m'_t}\boxtimes 1^{\boxtimes\epsilon},\]
where $m'_1,\dotsc,m'_u\geq 1$ are odd, and $m'_{u+1},\dotsc,m'_t\geq 0$ are even. Hence $[\tau:\omega]$ is the sum of $[\tau:\rho]$ over these $\rho$'s. Summing these up over all $\omega$'s, we obtain the pretty formula
\[[U:\tau]=\sum_{\bm\mu}[\tau:\rho_{\bm\mu}],\]
where $\bm\mu=(\mu_1,\dotsc,\mu_{r+t})$ runs through $\tilde{\bm\lambda}+2\ZZ_{\geq 0}^{r+t}$, and
\begin{equation}\label{rhodecomp}
\rho_{\bm\mu}:=\omega_{\mu_1}\boxtimes\dotsb\boxtimes\omega_{\mu_{r+t}}\boxtimes 1^{\boxtimes\epsilon}.
\end{equation}
Of course $\rho_{\bm\mu}$ can only occur in $\tau$ when $\bm\mu\preceq\bm\lambda$, hence in fact
\begin{equation}\label{Utaubound}
[U:\tau]\leq\sum_{\tilde{\bm\lambda}\leq\bm\mu\preceq\bm\lambda}[\tau:\rho_{\bm\mu}].
\end{equation}
Here $\tilde{\bm\lambda}\leq\bm\mu$ means a pointwise bound.

We proceed to estimate $[\tau:\rho_{\bm\mu}]$ in terms of the multiplicity $[\bm\lambda:\bm\mu]$ of $\bm\mu$ in the irreducible Lie algebra representation of $\mathfrak{so}(n,\CC)$ of highest weight $\bm\lambda$. Let $(e_i)$ be the standard basis of $\ZZ^{r+t}=\ZZ^{\lfloor n/2\rfloor}$, and consider the set $\Phi^+$ of positive roots of $\mathfrak{so}(n,\CC)$:
\begin{alignat*}{2}
\Phi^+&=\{e_i\pm e_j:1\leq i<j\leq n/2\},\qquad&&\text{$n$ even},\\
\Phi^+&=\{e_i\pm e_j:1\leq i<j\leq n/2\}\cup\{e_i:1\leq i\leq n/2\},\qquad&&\text{$n$ odd}.
\end{alignat*}
For any weight $\bm\nu\in\ZZ^{\lfloor n/2\rfloor}$, let $\mcP(\bm\nu)$ be the number of decompositions
\[\bm\nu=\sum_{\bm\alpha\in\Phi^+}k_{\bm\alpha}\bm\alpha\]
with $k_{\bm\alpha}\in\ZZ_{\geq 0}$. Let us assume that $\mathbf{0}\preceq\bm\nu$, for otherwise $\mcP(\bm\nu)=0$. Then, each term $k_{\bm\alpha}\bm\alpha$ above is $\preceq\bm\nu$, whence passing to partial sums reveals that
$k_{\bm\alpha}\ll\|\bm\nu\|_\infty$. This yields the convenient upper bound
\begin{equation}\label{eq:Pnubound}
\mcP(\bm\nu)\ll(1+\|\bm\nu\|_{\infty})^{|\Phi^+|}.
\end{equation}

Suppose that $\tau$ is a Type I representation. Then $[\bm\lambda:\bm\mu]$ equals the sum of $[\tau:\rho]$ over those representations $\rho\in\widehat{K''}$ that can be obtained from \eqref{rhodecomp} by replacing some (or all) factors $\omega_0$ by the determinant representation of $\OO(2)$ and, when $n$ is odd, also possibly replacing the last factor $1$ by $\sgn$. In particular,
$[\tau:\rho_{\bm\mu}]\leq[\bm\lambda:\bm\mu]$. By Kostant's formula \cite[Cor.~5.83]{Knapp-Lie} and \eqref{eq:Pnubound}, we conclude that
\[[\tau:\rho_{\bm\mu}]\ll\max\left\{\mcP(\bm\nu):\bm\nu\preceq\bm\lambda-\bm\mu\right\}
\ll(1+\|\bm\lambda-\bm\mu\|_\infty)^{|\Phi^+|}.\]
Now suppose that $\tau$ is a Type II representation. So $n$ is even and $\lambda_{n/2}>0$. Then, by a similar argument as in the Type I case,
\[[\tau:\rho_{\bm\mu}]\leq[\bm\lambda:\bm\mu]+[\bm\lambda':\bm\mu],\]
where $[\bm\lambda':\bm\mu]$ is the multiplicity of $\bm\mu$ in the irreducible Lie algebra representation of $\mathfrak{so}(n,\CC)$ of highest weight
\[\bm\lambda':=(\lambda_1,\dotsc,\lambda_{n/2-1},-\lambda_{n/2}).\]
So Kostant's formula now yields that
\[[\tau:\rho_{\bm\mu}]\ll\max\left\{\mcP(\bm\nu):\text{$\bm\nu\preceq\bm\lambda-\bm\mu$ or $\bm\nu\preceq\bm\lambda'-\bm\mu$}\right\}.\]
For $\bm\nu\preceq\bm\lambda-\bm\mu$, we estimate $\mcP(\bm\nu)$ as before. Now let $\bm\nu\preceq\bm\lambda'-\bm\mu$ with $\mcP(\bm\nu)\neq 0$. Then $\mathbf{0}\preceq\bm\lambda'-\bm\mu$, hence by the definition of $\preceq$, it follows that
\[0\leq(\lambda_1-\mu_1)+\dotsb+(\lambda_{n/2-1}-\mu_{n/2-1})+(-\lambda_{n/2}-\mu_{n/2}).\]
Therefore,
\[0\leq\lambda_{n/2}+\mu_{n/2}\ll\|\bm\lambda-\bm\mu\|_\infty,\]
and so $\|\bm{\nu}\|_{\infty}\ll\|\bm\lambda'-\bm\mu\|_\infty\ll\|\bm\lambda-\bm\mu\|_\infty$, and finally by \eqref{eq:Pnubound},
\[\mcP(\bm\nu)\ll(1+\|\bm\lambda-\bm\mu\|_\infty)^{|\Phi^+|}.\]
To sum up, we proved in both the Type I and Type II cases that
\[[\tau:\rho_{\bm\mu}]\ll(1+\|\bm\lambda-\bm\mu\|_\infty)^{|\Phi^+|}.\]

Going back to \eqref{Utaubound}, we infer that
\[[U:\tau]\ll\sum_{\tilde{\bm\lambda}\leq\bm\mu\preceq\bm\lambda}(1+\|\bm\lambda-\bm\mu\|_\infty)^{|\Phi^+|}.\]
However, it is straightforward to verify that
\[\tilde{\bm\lambda}\leq\bm\mu\preceq\bm\lambda
\qquad\Longrightarrow\qquad
\|\bm\lambda-\bm\mu\|_\infty\ll\|\bm\lambda-\tilde{\bm\lambda}\|_\infty,\]
which also means that the number of $\bm\mu$'s occurring in the sum is 
$(1+\|\bm\lambda-\tilde{\bm\lambda}\|_\infty)^{\lfloor n/2\rfloor}$. As a result,
\[[U:\tau]\ll(1+\|\bm\lambda-\tilde{\bm\lambda}\|_\infty)^{|\Phi^+|+\lfloor n/2\rfloor}.\]
The exponent equals $|\Phi^+|+\lfloor n/2\rfloor=\lfloor n^2/4\rfloor$, so we are done.
\end{proof}

It is instructive to see now how the baseline bound \eqref{eq:psi-trivial-bound} generalizes to nonminimal $K$-types. In general, the function $\hat F_\tau^{\sigma'}$ defined by \eqref{Fdef2} is not as simple as in \eqref{hatFtausigma2}. Instead, the updated right-hand side of \eqref{hatFtausigma2} is composed of $[U:\tau]$ diagonal $2^r\times 2^r$ blocks of shape
\begin{equation}\label{superblocks-eq}
\frac{\dim\tau}{2^r}\bigotimes_{p=1}^r \varphi_{\omega_{m_p}}^{\delta_{\ell_p}} (x_p) \cdot \prod_{k=1}^u \epsilon_k,
\end{equation}
with certain combinations of $m_p\in\{\ell_p,\ell_p+2,\dotsc\}$ as also seen in the proof of Lemma~\ref{lemma:multiplicitybound}. Each block corresponds to an $\omega\in\widehat{K'}$ as in \eqref{eq:omega-tensor} that occurs both in $\sigma'$ and $\tau$, and each $\omega\in\widehat{K'}$ gives rise to $[\tau:\omega]$ blocks. As each tensor factor is a $2\times 2$ diagonal or antidiagonal matrix with entries of absolute value at most one, \eqref{tracefunctionbound} implies that
\begin{equation}\label{eq:psi-baseline-bound}
\psi^{U}_{\tau}(x) \ll_{\Re\mu',\Omega} [U:\tau]\cdot\dim \tau,\qquad x\in\Omega.
\end{equation}
For completeness, we remark that this bound is also valid when $U$ is a discrete series representation, simply because $U$ is unitary in that case.

Our next theorem improves upon this baseline bound, at least when $\tau$ is ``close'' to the minimal $K$-type of $U$. It extends Theorem~\ref{thm3} and is a crucial ingredient in the proof of Theorem~\ref{thm1}.

\begin{theorem}\label{thm3b} Let $U=U^{\sigma',\mu'}$ be a generalized principal series representation, and let $\tau$ be a $K$-type whose highest weight $\bm\lambda$ satisfies
\[\|\bm\lambda\|_\infty\leq\|\bm\ell\|_\infty+\|U\|_\infty^{1/52}.\]
Let $\Omega\subseteq G$ be a compact set. Then, for any $x\in\Omega$, we have
\[\psi^{U}_{\tau}(x)\ll_{\Re\mu',\Omega}[U:\tau]\cdot\dim \tau\cdot\left(1 +\|U\|_{\infty}^{1/6}\dist(x, K)\right)^{-1/3}.\]
\end{theorem}

\begin{proof}
The proof is similar to the proof of Theorem~\ref{thm3} given earlier in this section, so we shall only indicate the changes necessary to cover the more general $K$-types $\tau$. We shall assume that $\|U\|_{\infty}$ is sufficiently large and $x\not\in K$, for otherwise the conclusion of Theorem~\ref{thm3b} follows from \eqref{eq:psi-baseline-bound}. For convenience, we write
\[\Delta:=\|U\|_\infty^{1/52},\qquad\Delta':=\Delta\log(\|\bm\ell\|_\infty+\Delta).\]

We start from the decomposition of $\hat F_\tau^{\sigma'}$ into $[U:\tau]$ diagonal $2^r\times 2^r$ blocks of shape \eqref{superblocks-eq}, as described before Theorem~\ref{thm3b}. By our restriction on $\tau$, in \eqref{superblocks-eq} we have
\[\ell_p\leq m_p\leq\|\bm\lambda\|_\infty\leq\|\bm\ell\|_\infty+\Delta,\qquad p\in\{1,\dotsc,r\}.\]
In particular, if $r\geq 1$, then we can apply \eqref{eq:goodbound} to derive the following variant of \eqref{eq:lemma5bound}:
\[
\|\hat F_{\tau}^{\sigma'} (m'(x))\|_{2,1}\ll [U:\tau]\cdot\dim \tau \cdot \big(1 - D(x)^2/4 \big)^{\ell_1/\Delta'},\qquad x\in G.
\]
For $x\in\Omega$, this yields the following variant of \eqref{eq:thm3-big-ell}:
\begin{equation}\label{eq:thm3-big-ell-variant}
\psi_{\tau}^{U}(x) \ll_{\Re\mu'} 
[U:\tau]\cdot\dim \tau \cdot \left(1+\|\bm\ell/\Delta'\|_{\infty}^{1/2}\cdot\dist(x,K)\right)^{-1/3}.
\end{equation}

Lemmata~\ref{lemma:matrix-coefficients-derivatives} and \ref{lemma:matrix-coefficients-derivatives-2} remain valid if we replace $\|\bm\ell\|_\infty$ by $\|\bm\lambda\|_\infty$. Then, we can prove the following generalization of \eqref{eq:lemma12bound}:
\[\left\|k\mapsto A(x,k,\mu')\right\|_{S_j}\ll_{j,\Re\mu'} [U:\tau]\cdot\dim \tau \cdot (1+\|\bm\lambda\|_{\infty})^{j}.\]
The only nontrivial step is to extend \eqref{eq:sobolev-proof-3} to our current situation:
\[
\left\| \left. \partial^{\beta} \hat{F}_\tau^{\sigma'}(m'(x E(\bt))) \right|_{\bt=\mathbf{0}} \right\|_{2,1} \ll_{\beta} [U:\tau]\cdot\dim\tau \cdot (1+\|\bm\lambda\|_\infty)^{|\beta|}.
\]
For this, we only need to remark that each matrix entry of
\[\varphi_{\omega_{m_p}}^{\delta_{\ell_p}} \left(\begin{pmatrix} a_p(\bt) & b_p(\bt) \\ 0 & d_p(\bt) \end{pmatrix}\right)\]
has $\partial^\gamma$-derivative at $\bt=\mathbf{0}$ of size $O_\gamma(m_p^{|\gamma|})$, which in turn follows from the discussion below \eqref{eq:goodbound}. At the end of the day, we obtain the following extension of \eqref{eq:stationary-phase-upshot}:
\begin{equation}\label{eq:stationary-phase-upshot-extension}
\begin{split}
\psi_\tau^{U}(x) \ll_{\Re\mu'}
&[U:\tau]\cdot\dim \tau \cdot h^{1/3} \cdot \dist(x,K)^{-1/3}\\
& + [U:\tau]\cdot\dim \tau \cdot (1+\|\bm\lambda\|_{\infty})
\left(\|\tilde{\mu}\|_{\infty}^{-1}h^{-1} + \|\tilde{\mu}\|_{\infty}^{-1/2}h^{-1/2}\right).
\end{split}
\end{equation}

We finish the proof by updating the optimization argument in \S\ref{thm3endgame}. If $\|\bm{\ell}\|_{\infty}\leq\|\mu'\|_{\infty}^{6/17}$, then $\|\bm\lambda\|_\infty\ll\|\bm\ell\|_\infty+\|\mu'\|_{\infty}^{1/52}\ll\|\mu'\|_{\infty}^{6/17}$, hence 
choosing $h:=\|\mu'\|_{\infty}^{-3/17}$ in \eqref{eq:stationary-phase-upshot-extension} yields that
\[\psi_\tau^{U}(x) \ll_{\Re\mu'} [U:\tau]\cdot\dim\tau\cdot\|U\|_{\infty}^{-1/17}\cdot\dist(x,K)^{-1/3}.\]
If $\|\bm{\ell}\|_{\infty}>\|\mu'\|_{\infty}^{6/17}$, then \eqref{eq:thm3-big-ell-variant} yields that
\[\psi_{\tau}^{U}(x) \ll_{\Re\mu'} [U:\tau]\cdot\dim \tau \cdot \Delta'^{1/6}\|\bm\ell\|_{\infty}^{-1/6}\cdot\dist(x,K)^{-1/3},\]
hence also that
\[\psi_{\tau}^{U}(x) \ll_{\Re\mu'} [U:\tau]\cdot\dim \tau \cdot \Delta'^{1/6}\|U\|_{\infty}^{-1/17}\cdot\dist(x,K)^{-1/3}.\]
However,
\[\Delta'^{1/6}\ll\Delta^{1/6}\log\Delta\ll\|U\|_\infty^{1/306},\]
hence we have the required conclusion in all cases.
\end{proof}

\begin{remark}\label{rmk-to-thm3b} The bound of Theorem~\ref{thm3b} also holds when $U=\delta_\ell$ is a discrete series representation of $G=G(2)$, and we do not need to restrict $x$ to $\Omega$. This follows (in a much stronger form) from \eqref{eq:goodbound} upon noting that
\[
a^2+b^2+c^2+d^2 - 2 \asymp \dist(x,K)^2,\qquad x\in G.
\]
Alternatively, we can regard $U$ as a degenerate generalized principal series representation with $P'=M'=G$, and then Theorem~\ref{thm3b} a priori covers discrete series. It is instructive to see what our proof boils down to in this special case. We start from the following variant of \eqref{eq:lemma5bound}:
\[
\|\varphi_\tau^{U}(x)\|_{2,1} \ll (1-D(x)^2/4)^{\ell/\Delta'}, \qquad x\in G.
\]
The left-hand side is right $K$-invariant, hence following the proof of Lemma~\ref{lemma:thm3-big-ell}, we obtain \eqref{eq:thm3-big-ell-variant}.
\end{remark}

\section{Bounds for spherical functions II: the case \texorpdfstring{$n=3$}{n=3}}\label{section_for_n_equals3}
In this section, we prove Theorems~\ref{thm4} and \ref{thm4b}. Building upon \S\ref{oldsubsection}, we shall work with
\[G = \SL_3^{\pm}(\RR),\qquad K=\OO(3),\qquad L=\SO(3).\]
As in \S\ref{thm3-section}, $U=U^{\sigma',\mu'}$ is a generalized principal series representation, $\tau\in\widehat{K}$ is a $K$-type appearing in $U|_K$, and $\Omega\subseteq G$ is a bi-$K$-invariant compact set. Initially, $\tau$ will be the (unique) minimal $K$-type in $U$, while in \S\ref{n3-nonminimal-subsec} we shall allow more general $K$-types $\tau$.

\subsection{Setup and initial bounds}\label{n3-initial-subsec}
In this subsection, we set up an explicit version of the integral representation \eqref{psitau-integral} below, which will be used throughout \S\ref{section_for_n_equals3}. Then we prove Lemma~\ref{n3-essential-support-lemma}, which shows that the trace function $\psi_{\tau}^U(x)$ is very small unless $x$ is very close to the manifold $S$ defined in Theorem~\ref{thm4}.

From \eqref{eq:formula_nonminimal_spherical_trace_function} and \eqref{eq:T-sigma'}, we recall the integral representation
\begin{equation}
\label{psitau-integral}
\psi_{\tau}^{U}(x)=\int_KT_\tau^{\sigma'}(k^{-1}xk)\,e^{(\mu'-\rho')(H'(k^{-1}xk))}\,dk,
\end{equation}
which is valid for every $\tau\in\widehat{K}$ appearing in the given nonminimal principal series representation $U=U^{\sigma',\mu'}$. We keep the notations $\sigma'=\delta_{\ell}\boxtimes\sgn^u$ and $\tau=\tau_{\lambda}^{\iota}$, where $\lambda\geq \ell\geq 2$ are integers, $u\in\{0,1\}$, and $\iota=(-1)^{\ell+u}$. Here, the trace function
\[ T_\tau^{\sigma'}(x)=T_\tau^{\sigma'}(km'a'n')=\tr\big(\tau(k)\hat{F}^{\sigma'}_{\tau}(m')\big) \]
is given in terms of the $KM'A'N'$ decomposition and described explicitly
by \eqref{n-3-tr-explicit-nonminimal} and \eqref{n-3-tr-explicit-nonminimal-2} for
\begin{equation}\label{kmprime-parametrization}
\begin{alignedat}{3}
k&=\epsilon_1 l\in K,& m'&=\epsilon_2\begin{pmatrix}m''&0\\0&1\end{pmatrix}\in G(2)\times G(1),&\qquad & \epsilon_i\in\{\pm 1\}, \\
l&=k[\alpha,\beta,\gamma]\in L,\qquad& m''&=\begin{pmatrix} a&b\\c&d\end{pmatrix}\in G(2).&&
\end{alignedat}
\end{equation}
and the Euler angles $\alpha,\beta,\gamma$ are as in \eqref{euler-angles}. We can and we shall assume that $\epsilon_2=1$, because we can replace $(\epsilon_1,\epsilon_2)$ by $(\epsilon_1\epsilon_2,1)$. Moreover, by \eqref{transformation-harmony}, we can and we shall assume that $m''\in\SL_2(\RR)$.

In this section, we prove bounds on $\psi_{\tau}^U(x)$ using the integral representation \eqref{psitau-integral}. In \S\ref{n3-initial-subsec}--\S\ref{proof-thm4-subsection}, we address the minimal $K$-type case $\lambda=\ell$ and prove Theorem~\ref{thm4}. For the proof of Theorem~\ref{thm2}, we will also need a version of this estimate for the nonminimal $K$-types $\tau_{\lambda}^{\iota}$ in a suitably wider range of $\lambda-\ell>0$, which we prove in \S\ref{n3-nonminimal-subsec}.

Our starting point is \eqref{n-3-tr-explicit}, which in the present case $\lambda=\ell$ readily implies that
\begin{equation}
\label{n-3-tr-explicit-upper-bound}
|T_{\tau}^{\sigma'}(x)|\leq (2\ell+1)\cos\big(\tfrac12\beta\big)^{2\ell}\Theta(m'')^{\ell},\quad \text{where}\quad \Theta(m''):=\left|\frac{2}{a-ib+ic+d}\right|.
\end{equation}
Writing $m''=k_1 \diag\big(\alpha_{m''},\alpha_{m''}^{-1}\big)k_2$ with $k_1,k_2\in\SO(2)$ and $\alpha_{m''}\in\RR_{\geq 1}$, we have that
\[ \Theta(m'')=\left(\frac12+\frac14\tr\left(m''{}^Tm''\right)\right)^{-1/2}
=\Bigg(1+\left(\frac{\alpha_{m''}-\alpha_{m''}^{-1}}2\right)^2\Bigg)^{-1/2}. \]
From this it is obvious that $\Theta(m'') \leq 1$ as well as
\[ \Theta(m'')^{\ell}\leq\ell^{-A}\qquad\text{unless}\qquad \alpha_{m''}= 1+O_{A}\left(\frac{\mathfrak{L}}{\sqrt{\ell}}\right), \]
where we denote $\mathfrak{L}:=\sqrt{\log\ell}$. The latter condition is equivalent to
\[
m'\in L'+O_{A}\left(\frac{\mathfrak{L}}{\sqrt{\ell}}\right),\qquad \text{where} \qquad L':=\SO(2)\times\{1\}.
\]
Further, for $l=k[\alpha,\beta,\gamma]\in L$, we clearly have that
\[ \cos\left(\tfrac12\beta\right)^{2\ell}\leq\ell^{-A}\qquad\text{unless}\qquad \beta\in2\pi\ZZ+O_{A}\left(\frac{\mathfrak{L}}{\sqrt{\ell}}\right), \]
where the latter condition implies via \eqref{euler-angles} that
\[ l\in L' +O_{A}\left(\frac{\mathfrak{L}}{\sqrt{\ell}}\right). \]
Now let us assume that $x=km'a'n'\in \Omega$. Then, by Lemma~\ref{lemma:iwasawa-coordinates-are-analytic}, the $a'$ and $n'$ components lie in a compact set depending on $\Omega$. From these observations
and \eqref{n-3-tr-explicit-upper-bound} it follows that
\begin{equation}
\label{n3-essential-support-conditions}
\begin{alignedat}{2}
|T_\tau^{\sigma'}(x)|&\leq 2\ell+1,& &\text{and} \\
|T_\tau^{\sigma'}(x)|&\leq \ell^{-A}& \qquad &\text{unless}\qquad x\in \pm L'A'N'+O_{A}\left(\frac{\mathfrak{L}}{\sqrt{\ell}}\right).
\end{alignedat}
\end{equation}
Recall the notation $S$ from Theorem~\ref{thm4}, and observe that
\[
\pm L'A'N' = \pm A'L'N' = \pm A' \begin{pmatrix} \SO(2) & * \\ & 1 \end{pmatrix}.
\]
\begin{lemma}
\label{n3-essential-support-lemma}
For all $A>0$ and $x\in\Omega$, we have
\[
\psi_{\tau}^{U}(x)\ll_{\eps,A,\Re\mu'}\frac{\ell^{1+\eps}}{(1+\sqrt{\ell} \cdot \dist(x,S))^{A}}.
\]
\end{lemma}
\begin{proof} We start from the integral representation \eqref{psitau-integral}. Keeping in mind that
\[
e^{(\mu'-\rho')(H'(k^{-1}xk))}\ll_{\Re\mu'} 1,\qquad x\in\Omega,
\]
we apply \eqref{n3-essential-support-conditions} twice. In the first round, we conclude the baseline bound
\begin{equation}\label{eq:5.1-baseline}
\psi_{\tau}^{U}(x)\ll_{\Re\mu'} \ell,\qquad x\in \Omega.
\end{equation}
In the second round, we denote by $C_A$ the implied constant in \eqref{n3-essential-support-conditions}, and argue as follows. If $\dist(x,S)\leq C_A\mathfrak{L}/\sqrt{\ell}$, then the previous display immediately gives the statement. If $\dist(x,S)> C_A\mathfrak{L}/\sqrt{\ell}$, then we are done by the second line of \eqref{n3-essential-support-conditions}  upon noting that $\dist(x,S)$ is bounded and invariant under $K$-conjugation.
\end{proof}
\begin{remark}\label{rmk:supplement} The proof above also gives the following useful supplement. For $x\in \Omega$, we have
\[
\psi_{\tau}^{U}(x) \ll_{\Re\mu'} \ell^{-A} \qquad \text{unless} \qquad x\in S+O_A\left( \frac{\mathfrak{L}}{\sqrt{\ell}}\right).
\]
\end{remark}

\subsection{Localization within the essential support}\label{n3-localization-subsec}
In this subsection, we prove Lemma~\ref{subsection52lemma}, which shows that the trace function $\psi_{\tau}^U(x)$ exhibits significant decay (generically of size $\ell^{-1+o(1)}$ smaller than the baseline bound \eqref{eq:5.1-baseline}) unless $x$ is very close to $\{\pm\id\}$.

In the integral representation \eqref{psitau-integral}, clearly we may integrate over $L$ instead of $K$. The Haar probability measure on $L$ is given by $dl=(8\pi^2)^{-1}\sin\beta\,d\alpha\,d\beta\,d\gamma$ using the Euler angle parametrization $l=k[\alpha,\beta,\gamma]$ as in \eqref{euler-angles}; see \cite[\S2.2.1]{Buttcane2018}. We observe that, by \eqref{inv}, \eqref{eq:T-sigma'}, and \eqref{endgame},
\begin{alignat*}{2}
T_{\tau}^{\sigma'}(l'^{-1}xl') &= T_{\tau}^{\sigma'}(x),&&\qquad x\in G,\quad l'\in L',\\
H'(l'^{-1}xl') &= H'(x),&&\qquad x\in G,\quad l'\in L'.
\end{alignat*}
We apply this invariance property with $l'=k[0,0,\gamma]\in L'$ to
show that the integrand is constant in $\gamma$, and then divide the $\beta$-integral according to dyadic ranges for $\sin\beta$.
With a future stationary phase argument in mind (to be presented in \S\ref{n3-oscillation-subsec}), we apply a smooth partition of unity. We fix a smooth function $f:\RR_{>0}\to\RR_{\geq 0}$ supported in $[1,4]$ such that
\[
\sum_{j\in\ZZ} f(\xi/2^j)=1, \qquad \xi\in\RR_{>0}.
\]
We use this identity for $\xi=\sqrt{\ell}\sin\beta$. Then, with the convenient notations
\[
x^{\alpha,\beta}:=k[\alpha,\beta,0]^{-1}xk[\alpha,\beta,0],\qquad\BB_j:=2^j/\sqrt{\ell},
\]
we can rewrite \eqref{psitau-integral} as
\begin{equation}\label{51-rewritten}
\psi_{\tau}^{U}(x)=\sum_{j\in\ZZ}\frac{1}{4\pi}\int_0^{2\pi}\int_0^\pi
T_\tau^{\sigma'}(x^{\alpha,\beta})\,e^{(\mu'-\rho')(H'(x^{\alpha,\beta}))}f\left(\frac{\sin\beta}{\BB_j}\right)\sin\beta\,d\beta\,d\alpha.
\end{equation}
Note that the double integral is supported on the set
\begin{equation}\label{eq:def-S_j}
S_j:=\left\{(\alpha,\beta)\in[0,2\pi]\times [0,\pi]: \sin\beta\in [\BB_j,4\BB_j]\right\},
\end{equation}
Hence we can restrict to $\BB_j\leq 1$, and conclude that
\begin{equation}
\label{to-estimate-1}
\psi_{\tau}^{U}(x)=\sum_{j\leq\log\ell/\log 4}I_j(x),\qquad x\in G,
\end{equation}
where
\begin{equation}\label{to-estimate}
I_j(x):=\frac{1}{4\pi}\iint\limits_{(\alpha,\beta) \in S_j}
T_\tau^{\sigma'}(x^{\alpha,\beta})\,e^{(\mu'-\rho')(H'(x^{\alpha,\beta}))}f\left(\frac{\sin\beta}{\BB_j}\right)\sin\beta\,d\alpha\,d\beta.
\end{equation}

Let us restrict to $x\in\Omega$ from now on. By the first line of \eqref{n3-essential-support-conditions}, we have
\begin{equation}\label{eq:trivialforIj}
I_j(x)\ll_{\Re\mu'}\ell\BB_j^2=4^j,\qquad x\in\Omega,
\end{equation}
hence from \eqref{to-estimate-1} we infer that
\[
\psi_{\tau}^U(x)\ll_{\Re\mu'}1+\sum_{0\leq j\leq \log\ell/\log 4}|I_j(x)|,\qquad x\in\Omega.
\]
By the second line of \eqref{n3-essential-support-conditions}, we can restrict the integration in \eqref{to-estimate} to
\begin{equation}\label{eq:def-S_jx}
S_j(x) := \left\{ (\alpha,\beta)\in S_j:x^{\alpha,\beta}\in \pm L'A'N'+O\left(\frac{\mathfrak{L}}{\sqrt{\ell}}\right)\right\}
\end{equation}
at the cost of an $O(\ell^{-10})$ error. Therefore, by a third application of \eqref{n3-essential-support-conditions} we conclude that
\begin{equation}
\label{to-estimate-2}
\psi_{\tau}^U(x)\ll_{\Re\mu'} 1+\sum_{\substack{0\leq j\leq \log\ell/\log 4\\S_j(x)\neq\emptyset}}|I_j(x)|,\qquad x\in\Omega.
\end{equation}

We turn to analyzing the condition $S_j(x)\neq\emptyset$ in \eqref{to-estimate-2}. In fact, by Remark~\ref{rmk:supplement}, we only need to carry out this analysis when
\begin{equation}
\label{K-conj}
x=\pm (x_0+x_1),\quad x_0=\begin{pmatrix} y\cos\theta&-y\sin\theta&y_{13}\\y\sin\theta&y\cos\theta&y_{23}\\0&0&1/y^2\end{pmatrix},\quad x_1\ll \frac{\mathfrak{L}}{\sqrt{\ell}}
\end{equation}
for $y\asymp 1$, $y_{13},y_{23}\ll 1$ (cf. Lemma~\ref{lemma:iwasawa-coordinates-are-analytic}).
We fix a pair $(\alpha,\beta)\in S_j(x)$. In light of \eqref{K-conj}, the condition that $x^{\alpha,\beta}\in \pm L'A'N'+O(\mathfrak{L}/\sqrt{\ell})$ is equivalent to
\begin{equation}\label{eq:x0conjugate}
k[\alpha,\beta,0]^{-1}x_0k[\alpha,\beta,0]\in \pm L'A'N'+O\left(\frac{\mathfrak{L}}{\sqrt{\ell}}\right).
\end{equation}
Now we compute the left-hand side. First,
\[ k[\alpha,0,0]^Tx_0k[\alpha,0,0]=\begin{pmatrix} y\cos\theta&-y\sin\theta&y_{13\alpha}\\y\sin\theta&y\cos\theta&y_{23\alpha}\\0&0&1/y^2\end{pmatrix}, \]
where
\begin{equation}
\label{y-alpha}
\begin{pmatrix}y_{13\alpha}\\y_{23\alpha}\end{pmatrix}:=\begin{pmatrix}\cos\alpha&\sin\alpha\\-\sin\alpha&\cos\alpha\end{pmatrix}\begin{pmatrix}y_{13}\\y_{23}\end{pmatrix},
\end{equation}
and then
\begin{equation}
\label{conjugate-expanded}
\begin{aligned}
&k[\alpha,\beta,0]^Tx_0k[\alpha,\beta,0]=\\
&\qquad\begin{pmatrix} y\cos^2\beta\cos\theta-\sin\beta(y_{13\alpha}\cos\beta-(1/y^2)\sin\beta)&-y\cos\beta\sin\theta&\ast\\ y\cos\beta\sin\theta - y_{23\alpha}\sin\beta&y\cos\theta&\ast\\y\cos\beta\cos\theta\sin\beta-\sin\beta((1/y^2)\cos\beta+y_{13\alpha}\sin\beta)&-y\sin\beta\sin\theta&\ast\end{pmatrix}.
\end{aligned}
\end{equation}
Now we use the key observation that
\[\begin{alignedat}{3}
g_{31}        &={} & g_{32}        &={} & 0,\\
g_{22}-g_{11} &={} & g_{12}+g_{21} &={} & 0,
\end{alignedat}
\qquad
(g_{ij})_{1\leq i,j\leq 3}\in\pm L'A'N'.\]
These identities and the evaluation \eqref{conjugate-expanded} show that \eqref{eq:x0conjugate} implies the bounds
\begin{equation}
\label{n3-four-conditions-1}
\begin{alignedat}{3}
\sin\beta\cdot ((y\cos\theta-1/y^2)\cos\beta-y_{13\alpha}\sin\beta)&\ll\frac{\mathfrak{L}}{\sqrt{\ell}},&\qquad&&
\sin\beta \cdot y\sin\theta&\ll\frac{\mathfrak{L}}{\sqrt{\ell}},\\
\sin\beta\cdot ((y\cos\theta-1/y^2)\sin\beta+y_{13\alpha}\cos\beta)&\ll\frac{\mathfrak{L}}{\sqrt{\ell}},&\qquad&&
\sin\beta \cdot y_{23\alpha}&\ll\frac{\mathfrak{L}}{\sqrt{\ell}}.
\end{alignedat}
\end{equation}
By $\sin \beta\asymp \BB_j$, these four conditions are equivalent to
\[
\left\|\begin{pmatrix}\cos\beta&-\sin\beta\\\sin\beta&\cos\beta\end{pmatrix}\begin{pmatrix}y\cos\theta-1/y^2\\y_{13\alpha}\end{pmatrix}\right\|,\,\,y\sin\theta,\,\,y_{23\alpha}\ll\frac{\mathfrak{L}}{\BB_j\sqrt{\ell}}.
\]

Since rotation by $\beta$ is an isometry, the system of four conditions in \eqref{n3-four-conditions-1} is equivalent to
\[ y\cos\theta-1/y^2,\,\,y\sin\theta,\,\,y_{13\alpha},\,\,y_{23\alpha}\ll\frac{\mathfrak{L}}{\BB_j\sqrt{\ell}}. \]
Since rotation by $\alpha$ in \eqref{y-alpha} is also an isometry, \eqref{n3-four-conditions-1} may be equivalently phrased as
\begin{equation}\label{eq:y-theta}
y\cos\theta-1/y^2,\,\,y\sin\theta,\,\,y_{13},\,\,y_{23}\ll\frac{\mathfrak{L}}{\BB_j\sqrt{\ell}}.
\end{equation}
To sum up so far, the assumption $S_j(x)\neq\emptyset$ implies these four inequalities, which can be compactly wrapped up as
\begin{equation}\label{eq:wrap}
\|x_0-(1/y^2)\id\|\ll\frac{\mathfrak{L}}{\BB_j\sqrt{\ell}}.
\end{equation}
Recall that $y\asymp 1$. By the second inequality in \eqref{eq:y-theta}, we see that $\sin\theta\ll \mathfrak{L}/(\BB_j\sqrt{\ell})$. Then by the first inequality in \eqref{eq:y-theta}, we see that $y-1/y^2\ll \mathfrak{L}/(\BB_j\sqrt{\ell})$, which then yields $y-1\ll \mathfrak{L}/(\BB_j\sqrt{\ell})$. Combining this with \eqref{K-conj} and \eqref{eq:wrap}, we finally arrive at
\begin{equation}
\label{n3-final-condition}
\dist(x,\{\pm\id\})\ll\frac{\mathfrak{L}}{\BB_j\sqrt{\ell}}.
\end{equation}
In summary, we have proved that $S_j(x)\neq\emptyset$ implies \eqref{n3-final-condition}. It is not difficult to see that the entire chain of deductions from \eqref{eq:x0conjugate} to \eqref{n3-final-condition} can be reversed for any $(\alpha,\beta)\in S_j(x)$ (with adjustments in the values of the implied constants). Hence in fact:
\begin{enumerate}
\item \label{citable-implication} $S_j(x)\neq\emptyset$ implies \eqref{n3-final-condition}, which in turn implies \eqref{eq:x0conjugate} for all $(\alpha,\beta)\in S_j$;
\item \eqref{n3-final-condition} with a sufficiently small implied constant forces $S_j(x)=S_j$.
\end{enumerate}

Going back to \eqref{to-estimate-2}, we find by \eqref{eq:trivialforIj} and by summing the relevant geometric series that
\begin{equation}\label{eq:geometric-series}
\psi_{\tau}^U(x)\ll_{\Re\mu'} 1 + \sum_{\substack{0\leq j\leq \log \ell / \log 4\\ \text{$\BB_j$ satisfies \eqref{n3-final-condition}}}}\ell\BB_j^2  \ll\frac{\log\ell}{\dist(x,\{\pm\id\})^2}.
\end{equation}

We record this result in the following lemma.
\begin{lemma}\label{subsection52lemma}
For all $x\in\Omega$, we have
\[ \psi_{\tau}^U(x)\ll_{\Re\mu'}\frac{\log\ell}{\dist(x,\{\pm\id\})^2}. \]
\end{lemma}

\subsection{Profiting from oscillation}\label{subsec:profiting-from-oscillation}\label{n3-oscillation-subsec}
All of the estimates of \S\ref{n3-initial-subsec}--\S\ref{n3-localization-subsec} follow after applying absolute values in the integral representation \eqref{psitau-integral}. In this subsection, we prove Lemma~\ref{subsection53lemma}, which shows that $\psi_{\tau}^U(x)$ exhibits further decay (generically of size $\ell^{-3/2+o(1)}$ smaller than the baseline bound \eqref{eq:5.1-baseline}) unless $x$ is very close to the manifold $R$ described in Theorem~\ref{thm4}. We obtain these estimates by detecting the oscillation in the factor $T_{\tau}^{\sigma'}(l^{-1}xl)$, where as in \S\ref{n3-localization-subsec} we clearly may integrate over $l=k[\alpha,\beta,0]\in L=\SO(3)$, and the trace function $T_{\tau}^{\sigma'}(x)=\tr(\tau(k)\hat{F}_{\tau}^{\sigma'}(m'))$ is given in terms of the $KM'A'N'$ decomposition by \eqref{n-3-tr-explicit}. To do so, we first explicitly describe this decomposition in the ranges of interest to us.

Recall that we are primarily interested in the situation where $x$ (and all its $K$-conjugates) belong to a certain 
compact subset $\Omega \subseteq G(3)$. For an open subset $U\subseteq\RR^{3\times 3}$, let $\mcF(U)$ denote the set of complex-valued smooth functions defined on $U$. For $f,f_1,\dotsc,f_n\in\mcF(U)$, we write $f=O_2[f_1,\dotsc,f_n]$ if there exist $g_{ij}\in\mcF(U)$ such that
\[ f=\sum_{1\leq i,j\leq n}f_if_jg_{ij}\qquad\text{and}\qquad\partial^\gamma f_i,\partial^\gamma g_{ij}\ll_{\gamma,U} 1\]
for any multi-index $\gamma\in\ZZ_{\geq 0}^{3\times 3}$. We will typically use this in a situation where $f_i\ll\delta$ are suitably small, in which case it is clear that $f\ll\delta^2$ and $\|\nabla f\|\ll\delta$.
We will also use the same notation for functions $f,f_1,\dots,f_n:\Omega'\to\CC$ initially defined on some subset $\Omega'\subseteq\Omega$ if they admit smooth extensions to some open domain $U\supseteq\Omega'$ (which can in our use be described by some explicit algebraic expressions in the entries $x_{ij}$) satisfying the above conditions. In such a case, the domain $U$ will be given in terms of $\Omega'$, so the above bounds read as $\partial^{\gamma}f_i,\partial^{\gamma}g_{ij}\ll_{\gamma,\Omega'}1$. Finally, for any matrix $x=(x_{ij})_{1\leq i,j\leq 3}\in\RR^{3\times 3}$, we denote by $\bx_j=(x_{ij})_{1\leq i\leq 3}\in\RR^{3\times 1}$ its $j$-th column.

\begin{lemma}
\label{kmprime-explicit}
For every compact subset $\Omega\subseteq G(3)$, there exists a constant $c_{\Omega}>0$ such that every matrix $x=(x_{ij})_{1\leq i,j\leq 3}\in\Omega$ satisfying $|x_{31}|,|x_{32}|\leq c_{\Omega}$ admits a $KM'A'N'$ decomposition $x=km'a'n'$ with
\[ k=\begin{pmatrix}\dfrac{x_{11}}{A}&-\eta\dfrac{x_{21}}{A}&\ast\\[8pt]\dfrac{x_{21}}{A}&\eta\dfrac{x_{11}}{A}&\ast\\[8pt]\ast&\ast&\ast\end{pmatrix}+O_2[x_{31},x_{32}],\quad m'=\begin{pmatrix}\dfrac AB&\dfrac{x_{11}x_{12}+x_{21}x_{22}}{AB}&0\\[8pt]0&\dfrac BA&0\\[8pt]0&0&1\end{pmatrix}+O_2[x_{31},x_{32}], \]
where $\det m'=1$, and
\[ A:=\sqrt{x_{11}^2+x_{21}^2},\quad B:=\sqrt{|x_{11}x_{22}-x_{12}x_{21}|},\quad \eta:=\eta(x):=\sgn(x_{11}x_{22}-x_{12}x_{21}). \]
\end{lemma}

\begin{proof} The Gram--Schmidt process yields the classical $K(AN)$ decomposition
\[ x=kb,\qquad k:=\begin{pmatrix}\bk_1&\bk_2&\bk_3\end{pmatrix},\qquad b:=\begin{pmatrix}\|\bx_1\|&\frac{\langle\bx_1,\bx_2\rangle}{\|\bx_1\|}&\ast\\0&Y&\ast\\0&0&\frac1{\|\bx_1\|Y}\end{pmatrix}, \]
in which $k\in K=\OO(3)$ and
\[ \bk_1:=\frac{\bx_1}{\|\bx_1\|},\qquad \bk_2:=\frac1Y\left(\bx_2-\frac{\langle\bx_1,\bx_2\rangle}{\|\bx_1\|^2}\bx_1\right),\qquad Y:=\left\|\bx_2-\frac{\langle \bx_1,\bx_2\rangle}{\|\bx_1\|^2}\bx_1\right\|. \]
We compute that
\begin{align*}
\|\bx_1\|^2\bx_2-\langle\bx_1,\bx_2\rangle\bx_1
&=\begin{pmatrix}x_{12}x_{21}^2-x_{11}x_{21}x_{22}\\x_{22}x_{11}^2-x_{11}x_{12}x_{21}\\x_{32}\|\bx_1\|^2-x_{31}\langle\bx_1,\bx_2\rangle\end{pmatrix}+O_2[x_{31},x_{32}]\\
&=\begin{pmatrix}- x_{21}\Delta\\ x_{11}\Delta\\x_{32}\|\bx_1\|^2-x_{31}\langle\bx_1,\bx_2\rangle\end{pmatrix}+O_2[x_{31},x_{32}],
\end{align*}
where $\Delta:=x_{11}x_{22}-x_{12}x_{21}$. Since $\Omega$ is compact, we have that
\[ 1= |\det x| = |x_{33}\Delta| + O(|x_{32}|) + O(|x_{31}|). \]
For sufficiently small $c_{\Omega}>0$, this implies
\[1\gg A\gg|\Delta|\gg|x_{33}^{-1}|\gg 1.\]
In fact, for sufficiently small $c_{\Omega}>0$, the same conclusion is true on the open set
\[ U:=\left\{x\in\RR^{3\times 3}: \dist(x,\Omega),|x_{31}|,|x_{32}|<2c_{\Omega}\right\} \]
containing the set $\Omega':=\{x\in\Omega:|x_{31}|,|x_{32}|\leq c_{\Omega}\}$ on which our decompositions will be valid.

Thus, $A,|\Delta|\asymp 1$ holds on the open set $U\subseteq\RR^{3\times 3}$, and we further have that
\begin{align*}
\|\bx_1\|&=\sqrt{x_{11}^2+x_{21}^2+x_{31}^2}=A+O_2[x_{31},x_{32}],\\
\big\|\|\bx_1\|^2\bx_2-\langle \bx_1,\bx_2\rangle\bx_1\big\|&=A|\Delta|+O_2[x_{31},x_{32}],
\end{align*}
as well as
\[ Y=\frac{|\Delta|}A+O_2[x_{31},x_{32}],\qquad B=\sqrt{|\Delta|},\qquad \eta=\sgn(\Delta). \]
From this it is immediate that $k$ is of the announced form. We also have the $M'A'N'$ decomposition $b=m'a'n'$ with
\[ m'=\begin{pmatrix} \sqrt{\frac{\|\bx_1\|}{Y}}&\frac{\langle\bx_1,\bx_2\rangle}{\sqrt{\|\bx_1\|^3Y}}&\ast\\0&\sqrt{\frac{Y}{\|\bx_1\|}}&\ast\\0&0&1\end{pmatrix},\qquad a'=\begin{pmatrix}\sqrt{\|\bx_1\|Y}&&\\&\sqrt{\|\bx_1\|Y}&\\&&\dfrac{1}{\|\bx_1\|Y}\end{pmatrix}, \]
and the announced form of $m'$ follows as well.
\end{proof}

Now, let $x\in\Omega$ satisfy the condition $|x_{31}|,|x_{32}|\leq c_{\Omega}$, where we may assume that $c_{\Omega}>0$ is sufficiently small. We rewrite the decomposition $x=km'a'n'$ from Lemma~\ref{kmprime-explicit} as
\begin{equation}\label{eq:xdecomp}
x=\epsilon lm'a'n',\qquad \epsilon:=\det x=\det k\in\{\pm 1\},\qquad l:=\epsilon k\in L=\SO(3).
\end{equation}
Invoking \eqref{n-3-tr-explicit} and \eqref{n3-alphaplusgamma}, we conclude that $T_{\tau}^{\sigma'}(x)=(O(c_{\Omega}))^{\ell}=O(2^{-\ell})$ for sufficiently small $c_{\Omega}>0$ if $\eta(x)=-1$. If $\eta(x)=1$, then using also $\epsilon^{\ell+u}\epsilon^\ell=\epsilon^u$, we obtain that
\begin{align*}
 T_{\tau}^{\sigma'}(x)={}&(2\ell+1)\epsilon^u\\
 &\times\Re\left[\left(\frac{x_{11}+ix_{21}}{A}\cdot
 \frac{2AB}{A^2+B^2+i(x_{11}x_{12}+x_{21}x_{22})}
 +O_2[x_{31},x_{32}]\right)^{\ell}\right].
\end{align*}
Since
\[ A^2+B^2+i(x_{11}x_{12}+x_{21}x_{22})=(x_{11}+ix_{21})(x_{11}+x_{22}+i(x_{12}-x_{21})), \]
we conclude that
\begin{equation}\label{eta1-explicit-b}
T_{\tau}^{\sigma'}(x)=(2\ell+1)\epsilon^u
\cdot\Re\left[\left(\frac{x_{11}+x_{22}+i(x_{12}-x_{21})}{2B}+O_2[x_{31},x_{32}]\right)^{-\ell}\right].
\end{equation}

The bound in Lemma~\ref{subsection53lemma} follows immediately from Remark~\ref{rmk:supplement} unless $x$ is (after $K$-conjugation) of the form \eqref{K-conj}, so from now on we assume that is indeed the case. Moreover, if $\ell$ is sufficiently large in terms of $\Omega$, then we may write \eqref{K-conj} as $x=\epsilon(x_0+ x_1)$. We now denote 
\[ x^{\alpha,\beta}:=k[\alpha,\beta,0]^{-1}xk[\alpha,\beta,0],\qquad x_0^{\alpha,\beta}:=k[\alpha,\beta,0]^{-1}x_0k[\alpha,\beta,0], \]
and turn our attention to analyzing the oscillatory behavior of $T_{\tau}^{\sigma'}(x^{\alpha,\beta})$. We will perform this analysis when
\begin{equation}
\label{sinbeta-nottoosmall}
\beta\in\Omega'':=[0,\delta_{\Omega}]\cup[\pi-\delta_{\Omega},\pi].
\end{equation}
for a certain sufficiently small constant $\delta_{\Omega}>0$. In particular, $\eta(x^{\alpha,\beta})=1$ holds by \eqref{conjugate-expanded}. We will see that the contributions of the complementary ranges to the integral representation of $\psi_{\tau}^U(x)$ in \eqref{51-rewritten} are negligibly small unless $\dist(x,\{\pm\id\})$ is so small that the baseline bound is as good as the target bound in Lemma~\ref{subsection53lemma}.

Specifically, keeping in mind that $y\asymp_{\Omega}1$ and $y_{13},y_{23}\ll_{\Omega}1$ in \eqref{K-conj}, we can in light of \eqref{conjugate-expanded} choose a $\delta_{\Omega}>0$ sufficiently small that, for sufficiently large $\ell$, we have $|(x^{\alpha,\beta})_{31}|,|(x^{\alpha,\beta})_{32}|\leq c_{\Omega}$ for all $(\alpha,\beta)$ satisfying the condition \eqref{sinbeta-nottoosmall}, which we assume henceforth. (This is the initial assumption; at several points throughout the argument that follows, we will, as we may, additionally assume $\delta_{\Omega}>0$ to be suitably even smaller.)
Then, by \eqref{conjugate-expanded} and \eqref{eta1-explicit-b} we may write
\begin{equation}
\label{Ttausigma-explicit}
T_{\tau}^{\sigma'}(x^{\alpha,\beta})=(2\ell+1)\epsilon^{\ell+u}\cdot\Re\left[\phi(x^{\alpha,\beta})^{-\ell}\right]
\qquad \text{with} \qquad\phi(x^{\alpha,\beta})=\frac{C}{2B},
\end{equation}
where
\begin{align*}
C&=2y\cos\theta-y_{13\alpha}\sin\beta\cos\beta+(1/y^2-y)\sin^2\beta-2iy\cos\beta\sin\theta+iy_{23\alpha}\sin\beta\\
&\qquad +Q_1(\alpha,\beta,x)+O_2\left[(x_0^{\alpha,\beta})_{31},(x_0^{\alpha,\beta})_{32}\right],\\
B&=\sqrt{y^2\cos^2\beta-y\sin\beta\cos\beta(y_{13\alpha}\cos\theta+y_{23\alpha}\sin\theta)+(1/y)\sin^2\beta\cos\theta+
Q_2(\alpha,\beta,x)},
\end{align*}
and $Q_1(\alpha,\beta,x)$ and $Q_2(\alpha,\beta,x)$ are some trigonometric polynomials in $\alpha$ and $\beta$ whose coefficients depend on $x$ and are of size $O(\mathfrak{L}/\sqrt{\ell})$.

Now, it will be convenient to write $E=E(\alpha,\beta,x)$ for any function on $[0,2\pi]\times \Omega'' \times \Omega$ that is smooth in $\beta$ and satisfies
\begin{equation}
\label{E-derivatives}
\frac{\partial^k E}{\partial\beta^k}\ll_k \begin{cases} \sin^{2-k}\beta + \mathfrak{L}/\sqrt{\ell}, & k\in\{0,1\}, \\
1, & k\in\{2,3,\dotsc\}.
\end{cases}
\end{equation}
We allow this function to vary from line to line. Then
\begin{align*}
C&=2y(\cos\theta-i\varsigma\sin\theta)-\sin\beta\cdot(\varsigma y_{13\alpha}-iy_{23\alpha})+E\\
&=(\cos\theta-i\varsigma\sin\theta)\big(2y-\sin\beta\cdot(\varsigma y_{13\alpha}-iy_{23\alpha})(\cos\theta+i\varsigma\sin\theta)+E\big),\\
B&=\sqrt{y^2-y\sin\beta\cdot\varsigma(y_{13\alpha}\cos\theta+y_{23\alpha}\sin\theta)+E},
\end{align*}
where $\varsigma=\sgn\cos\beta\in\{\pm 1\}$.

In order to exploit the oscillations in the integrals involving the trace function expressed as in \eqref{Ttausigma-explicit}, we write $\phi(x^{\alpha,\beta})=re^{it}$, where in light of
\[ |C|,|2B|=2y+O\left(\sin\beta+\frac{\mathfrak{L}}{\sqrt{\ell}}\right) \]
we see that, for sufficiently small $\delta_{\Omega}>0$ and sufficiently large $\ell$, 
$r=r(\alpha,\beta,x)$ and $t=t(\alpha,\beta,x)$ are smooth real-valued functions satisfying
\begin{align*}
\frac{\partial r/\partial\beta}{r}+i\frac{\partial t}{\partial\beta}
=&\,\frac{\partial\phi(x^{\alpha,\beta})/\partial\beta}{\phi(x^{\alpha,\beta})}
=\frac{\partial C/\partial\beta}C-\frac{\partial B/\partial\beta}{B}\\
=&\,\frac{-\cos\beta\cdot(\varsigma y_{13\alpha}-iy_{23\alpha})(\cos\theta+i\varsigma\sin\theta)+\partial E/\partial\beta}{2y-\sin\beta(\varsigma y_{13\alpha}-iy_{23\alpha})(\cos\theta+i\varsigma\sin\theta)+E}\\
&\,-\frac{-y\cos\beta\cdot\varsigma(y_{13\alpha}\cos\theta+y_{23\alpha}\sin\theta)+\partial E/\partial\beta}{2y^2-2y\sin\beta\cdot\varsigma(y_{13\alpha}\cos\theta+y_{23\alpha}\sin\theta)+E}.
\end{align*}
By \eqref{E-derivatives}, this implies that
\begin{align*}
\frac{\partial r/\partial\beta}{r}+i\frac{\partial t}{\partial\beta}
=&\,-\frac{\cos\beta\cdot(\varsigma y_{13\alpha}-iy_{23\alpha})(\cos\theta+i\varsigma\sin\theta)}{2y}\\
&\,+\frac{\cos\beta\cdot\varsigma(y_{13\alpha}\cos\theta+y_{23\alpha}\sin\theta)}{2y}
+O\left(\sin\beta+\frac{\mathfrak{L}}{\sqrt{\ell}}\right)\\
=&\,-\frac{i\cos\beta}{2y}(y_{13\alpha}\sin\theta-y_{23\alpha}\cos\theta)
+O\left(\sin\beta+\frac{\mathfrak{L}}{\sqrt{\ell}}\right).
\end{align*}
Since $\cos\beta=\varsigma+O(\sin^2\beta)$, we conclude that
\begin{equation}
\label{uv-1}
\frac{\partial r}{\partial\beta}\ll\sin\beta+\frac{\mathfrak{L}}{\sqrt{\ell}},\quad \frac{\partial t}{\partial\beta}=-\frac{\varsigma}{2y}(y_{13\alpha}\sin\theta-y_{23\alpha}\cos\theta)+O\left(\sin\beta+\frac{\mathfrak{L}}{\sqrt{\ell}}\right),
\end{equation}
as well as
\begin{equation}
\label{uv-2}
\frac{\partial^kr}{\partial\beta^k},\frac{\partial^kt}{\partial\beta^k}\ll_k 1, \qquad k\geq 2.
\end{equation}
We record for future reference that $r\geq 1$ by \eqref{n-3-tr-explicit}, and hence \eqref{uv-1} and \eqref{uv-2} combined with Fa\`a di Bruno's formula \cite[Th.~2]{MR602839} imply that
\begin{equation}\label{eq:fdb}
\frac{\partial^k(r^{-\ell})}{\partial\beta^k} \ll_k r^{-\ell} \left(\ell\sin\beta + \sqrt{\ell} \mathfrak{L}\right)^k.
\end{equation}
We summarize these findings in the following lemma.
\begin{lemma}
\label{phase-explicated-lemma}
For every compact set $\Omega\subseteq G(3)$, there exist constants $\delta_{\Omega},\ell_{\Omega}>0$ with the following property. Let $(\alpha,\beta,x)\in [0,2\pi]\times \Omega'' \times \Omega$, with $\Omega''$ as in \eqref{sinbeta-nottoosmall} and $x$ satisfying \eqref{K-conj}. If $\ell\geq\ell_\Omega$, then
the trace function appearing in \eqref{51-rewritten} can be expressed as
\[T_{\tau}^{\sigma'}(x^{\alpha,\beta})=(2\ell+1)\epsilon^{\ell+u}\cdot\Re\left[\phi(x^{\alpha,\beta})^{-\ell}\right], \]
where $\phi(x^{\alpha,\beta})=re^{it}$ with explicit fixed smooth functions $r=r(\alpha,\beta,x)$ and $t=t(\alpha,\beta,x)$
satisfying \eqref{uv-1}--\eqref{uv-2} and 
$r=1+O(\sin\beta+\mathfrak{L}/\sqrt{\ell})$.
\end{lemma}

With Lemma~\ref{phase-explicated-lemma} at our disposal, we now turn our attention to the goal of proving a refined estimate for $I_j(x)$ of \eqref{to-estimate} and the corresponding refined estimate for the trace function $\psi_{\tau}^U(x)$; the plan is to use the now-explicit oscillatory behavior of $T_{\tau}^{\sigma'}(x^{\alpha,\beta})$ and \cite[Lemma~8.1]{BlomerKhanYoung2013} to show that, in suitable ranges, large portions of the integral defining $I_j(x)$ are negligibly small. With this in mind, until told otherwise, we impose the assumption
\begin{equation}\label{eq:BKY-input-condition1}
\dist(x,\{\pm\id\}) \gg \frac{\mathfrak{L}}{\sqrt{\ell}}
\end{equation}
with a sufficiently large implied constant (depending only on $\Omega$). Then, in view of \eqref{eq:def-S_j}, \eqref{eq:def-S_jx}, and \eqref{n3-final-condition},
\begin{equation}\label{eq:BKY-input-condition2}
S_j(x)\neq\emptyset\quad\Longrightarrow\quad S_j\subseteq[0,2\pi]\times\Omega''.
\end{equation}

As a warm-up, we record that
\begin{equation}\label{Ij-new-b}
I_j(x)\ll_{\Re\mu'} \min(\ell^{-10}, \ell\BB_j^2) \quad\text{if}\quad S_j(x)=\emptyset,
\end{equation}
an immediate consequence of \eqref{eq:def-S_jx} and the sentence containing it. Hence we shall focus on the case when $S_j(x)\neq \emptyset$. By \eqref{eq:BKY-input-condition2} and Lemma~\ref{phase-explicated-lemma}, we can decompose the integral \eqref{to-estimate} as
\begin{equation}
\label{Ij-Ij-alpha}
I_j(x)=\int_0^{2\pi}I_{j,\alpha}(x)\,d\alpha, \quad I_{j,\alpha}(x):=\sum_{\pm}\int_0^{\pi}w_j(\alpha,\beta,x)e^{\pm i\ell t(\alpha,\beta,x)}\,d\beta,
\end{equation}
where
\[ w_j(\alpha,\beta,x):=\frac1{8\pi}(2\ell+1)\epsilon^{\ell+u}\cdot e^{(\mu'-\rho')(H'(x^{\alpha,\beta}))}f\left(\frac{\sin\beta}{\BB_j}\right)\,\sin\beta\cdot r(\alpha,\beta,x)^{-\ell}. \]
Using \eqref{eq:fdb}, we find (with a bit of thought) for every $k\geq 0$ that
\begin{equation}
\label{w-conditions}
\frac{\partial^kw_j(\alpha,\beta,x)}{\partial\beta^k}\ll_{k} XU^{-k}\qquad\text{with}\qquad X:=\ell\BB_j,\,\,U:=(\ell\BB_j^{\ast})^{-1},
\end{equation}
where we denote shorthand
\begin{equation}
\label{bjast-def}
\BB_j^{\ast}:=\BB_j+\frac{\mathfrak{L}}{\sqrt{\ell}}+\frac1{\ell\BB_j}+\frac{\|\mu'-\rho'\|}{\ell},
\end{equation}
and we suppress the (continuous) dependence on $\Re\mu'$ (and also the dependence on $\Omega$ as before) from notation.

On the other hand, recalling the definition \eqref{y-alpha} and writing
\[ \by_3:=\begin{pmatrix} y_{13}\\y_{23}\end{pmatrix}=\|\by_3\|\begin{pmatrix}\cos\upsilon\\\sin\upsilon\end{pmatrix}, \]
we have that
\begin{align*}
y_{13\alpha}\sin\theta-y_{23\alpha}\cos\theta
&=\|\by_3\|\begin{pmatrix}\sin\theta&-\cos\theta\end{pmatrix}\begin{pmatrix}\cos\alpha&\sin\alpha\\-\sin\alpha&\cos\alpha\end{pmatrix}\begin{pmatrix}\cos\upsilon\\\sin\upsilon\end{pmatrix}\\
&=\|\by_3\|\sin(\alpha+\theta-\upsilon).
\end{align*}
Thus, using \eqref{uv-1}, we infer for all $\beta$ with $\sin\beta\leq 4\BB_j$ that
\[ \frac{\partial(\ell t(\alpha,\beta,x))}{\partial\beta}=\frac{\ell\varsigma}{2y}\|\by_3\|\sin(\upsilon-\alpha-\theta)+O(\ell\BB_j^{\ast}), \]
and, using \eqref{uv-2}, we infer for every $k\geq 2$ that
\begin{equation}
\label{v-condition-2}
\frac{\partial^k(\ell t(\alpha,\beta,x))}{\partial\beta^k}\ll_{k} YQ^{-k}\qquad\text{with}\qquad Y:=\ell,\,\,Q:=1.
\end{equation}

Now, fix an $\eta>0$, and suppose that
\begin{equation}\label{eq:BKY-condition}
\|\by_3\|\cdot|\sin(\upsilon-\alpha-\theta)|\gg\ell^{\eta}\BB_j^{\ast}
\end{equation}
with a sufficiently large implied constant depending on $\Omega$. Then, for $\ell$ sufficiently large, we have for all $\beta$ satisfying $\sin\beta\leq 4\BB_j$ that
\begin{equation}
\label{v-condition-1}
\left|\frac{\partial(\ell t(\alpha,\beta,x))}{\partial\beta}\right|\geq R
\qquad\text{with}\qquad R:=\ell^{1+\eta}\BB_j^{\ast}.
\end{equation}
Combining \eqref{w-conditions}, \eqref{v-condition-1}, and \eqref{v-condition-2}, we find that the integral $I_{j,\alpha}(x)$ satisfies all conditions of \cite[Lemma~8.1]{BlomerKhanYoung2013} and that, moreover,
\[ QR/\sqrt{Y}\geq\ell^{\eta},\qquad RU\geq\ell^{\eta}. \]
Therefore, \cite[Lemma~8.1]{BlomerKhanYoung2013} yields that
\begin{equation}\label{eq:BKY-alpha}
I_{j,\alpha}(x)\ll_{\eta}\ell\BB_j^2\cdot\ell^{-10}\text{\quad if\quad \eqref{eq:BKY-condition} holds}.
\end{equation}

Now we estimate $I_j(x)$ via
\[|I_j(x)| \leq \int_{\text{\eqref{eq:BKY-condition} fails}} |I_{j,\alpha}(x)|\,d\alpha + \int_{\text{\eqref{eq:BKY-condition} holds}} |I_{j,\alpha}(x)|\,d\alpha,\]
and observe that the $\alpha$'s for which \eqref{eq:BKY-condition} fails form at most three intervals of total length $O(\ell^{\eta}\BB_j^{\ast}/\|\by_3\|)$. Using also \eqref{eq:BKY-alpha}, we arrive at
\begin{equation}
\label{Ij-new}
I_j(x) \ll_{\eta} \ell\BB_j^2\left( \frac{\ell^{\eta}\BB_j^{\ast}}{\|\by_3\|}+\ell^{-10}\right)
\quad\text{if}\quad S_j(x)\neq\emptyset.
\end{equation}

We turn to the main prize, the estimation of $\psi_{\tau}^U(x)$, which as in \eqref{to-estimate-1} is expressed as the sum of $I_j(x)$ over all $j\in\ZZ$ with $j\leq\log\ell/\log 4$ (thus, $\BB_j\leq 1$). Combining \eqref{to-estimate-1}, \eqref{Ij-new-b}, \eqref{bjast-def}, and \eqref{Ij-new}, we conclude that
\begin{align}
\notag\psi_{\tau}^U(x)
&\ll_{\eta} \sum_{\substack{j\leq\log\ell/\log 4\\\ S_j(x)\neq\emptyset}}\frac{\ell^{1+\eta}\BB_j^2\BB_j^{\ast}}{\|\by_3\|}+ \sum_{j\leq \log\ell/\log 4} \min(\ell^{-10}, \ell\BB_j^2)\\
\notag&\ll\frac{\ell^{1+\eta}\mathfrak{L}}{\|\by_3\|}
\sum_{\substack{j\leq \log \ell / \log 4\\ \text{$\BB_j$ satisfies \eqref{n3-final-condition}}}}
\left(\BB_j^3+\frac{\BB_j}{\ell}+\frac{\BB_j^2\|\mu'-\rho'\|}{\ell}\right)+\ell^{-10}\log\ell\\
\label{muprime-dependence}&\ll_{\eta}\frac{\ell^{2\eta}}{\sqrt{\ell}\dist(x,\{\pm\id\})^3\|\by_3\|}\left(1+\frac{\|\mu'-\rho'\|\dist(x,\{\pm\id\})}{\sqrt{\ell}}\right).
\end{align}
Finally, we note that, in light of \eqref{K-conj},
\[ \dist(x,R)\leq\left\|(x_0+x_1)-\begin{pmatrix} y\cos\theta&-y\sin\theta&0\\y\sin\theta&y\cos\theta&0\\0&0&1/y^2\end{pmatrix}\right\|\ll\|\by_3\|+\frac{\mathfrak{L}}{\sqrt{\ell}}. \]
Thus, as long as $\|\by_3\|\geq\mathfrak{L}/\sqrt{\ell}$,
\[ \psi_{\tau}^U(x)\ll_{\Re\mu',\eta}\frac{\ell^{2\eta}}{\sqrt{\ell}\dist(x,\{\pm\id\})^3\dist(x,R)} \left(1+\frac{\|\mu'-\rho'\|\dist(x,\{\pm\id\})}{\sqrt{\ell}}\right). \]
The same estimate also holds if $\|\by_3\|<\mathfrak{L}/\sqrt{\ell}$, for then this follows directly from Lemma~\ref{subsection52lemma}.

Now we drop the condition \eqref{eq:BKY-input-condition1}, and observe that the last display remains true in the complementary range $\dist(x,\{\pm\id\})\ll \mathfrak{L}/\sqrt{\ell}$ by the baseline bound \eqref{eq:5.1-baseline}. Putting everything together and taking $\eta>0$ as small as desired, we obtain the following statement.

\begin{lemma}\label{subsection53lemma}
For every $\eps>0$, we have for every $x\in\Omega$ the bound
\[ \psi_{\tau}^U(x)\ll_{\Re\mu',\eps}\frac{\ell^{\eps}}{\sqrt{\ell}\dist(x,\{\pm\id\})^3\dist(x,R)} \left(1+\frac{\|\mu'-\rho'\|\dist(x,\{\pm\id\})}{\sqrt{\ell}}\right) . \]
\end{lemma}

\subsection{Proof of Theorem~\ref{thm4}}\label{proof-thm4-subsection}
Theorem~\ref{thm4} follows immediately by combining Lemmata~\ref{n3-essential-support-lemma}, \ref{subsection52lemma}, and \ref{subsection53lemma}.

\subsection{Nonminimal \texorpdfstring{$K$}{K}-types}\label{n3-nonminimal-subsec}
In this subsection, we indicate the adjustments to the arguments in the previous subsections which allow us to prove versions of the estimate on $\psi_{\tau}^U(x)$ in Theorem~\ref{thm4} with
$U=U^{\sigma',\mu'}$, $\sigma'=\delta_{\ell}\boxtimes\sgn^u$, $\tau=\tau_{\lambda}^{\iota}$, $\iota=(-1)^{\ell+u}$, and $\lambda\in[\ell,2\ell]$. For convenience, we denote 
\begin{equation}\label{eq:Deltadef}
\Delta:=\lambda-\ell+2.
\end{equation}
It follows from \eqref{schur-Hom} and the discussion below \eqref{n3-alphaplusgamma} that 
\[[U:\tau]=\sum_{\substack{\ell\leq m\leq\lambda\\m\equiv\ell\pmod*{2}}}1=\lfloor\Delta/2\rfloor,\]
hence the baseline bound \eqref{eq:psi-baseline-bound} gives
\begin{equation}\label{eq:psi-nonminimal-baseline}
\psi_{\tau}^U(x)\ll_{\Re\mu'}\Delta\ell,\qquad x\in\Omega.
\end{equation}

We begin our proof of an improved estimate by collecting the estimates on the trace function $T_{\tau}^{\sigma'}(x)=\tr\bigl(\tau(k)\hat{F}_{\tau}^{\sigma'}(m')\bigr)$ for $x=km'a'n'$ and $k=\pm l$, $m'=\big(\begin{smallmatrix}m''&0\\0&1\end{smallmatrix}\big)$, $l=k[\alpha,\beta,\gamma]\in L$, $m''\in \SL_2(\RR)$ parametrized as in \eqref{kmprime-parametrization} and the discussion below it. Now the discussion below \eqref{eq:goodbound-tau} can be succinctly summarized as follows: the matrix $\hat{F}_{\tau}^{\sigma'}(m')$ is diagonal or anti-diagonal whose entries are $(\dim\tau)/2$ times suitable matrix coefficients of the unitary representation $\sigma'$. Since $\tau$ is also a unitary representation, we have the baseline bound
\begin{equation}
\label{baseline-nonminimal-eq}
|T_{\tau}^{\sigma'}(x)|\leq [U:\tau] \cdot \dim \tau \ll \Delta \ell,\qquad x\in G,
\end{equation}
which also serves as the basis of \eqref{eq:psi-nonminimal-baseline}. Typically, this bound
is not sharp because the matrix coefficients of $\sigma'$ and $\tau$ are much smaller than $1$ (in absolute value); see \eqref{eq:goodbound} and \eqref{eq:goodbound-tau} in this regard. In fact, plugging the bounds \eqref{eq:goodbound} and \eqref{eq:goodbound-tau} into the identity \eqref{n-3-tr-explicit-nonminimal}, and keeping in mind the notation \eqref{eq:Deltadef}, we obtain
\[
|T_{\tau}^{\sigma'}(x)| \leq(2\lambda+1) e^2\sum_{\substack{\ell\leq m\leq\lambda\\m\equiv\ell\pmod*{2}}}\Theta(m'')^{2\ell/(\Delta\log(\lambda+2))}(\cos\tfrac12\beta)^{2m/(\Delta\log(2\lambda))}.
\]
where $0<\Theta(m'')\leq 1$ is as in \eqref{n-3-tr-explicit-upper-bound}. From this it follows that
\[ |T_{\tau}^{\sigma'}(x)|\leq\ell^{-A}\qquad\text{unless}\qquad \Theta(m''),\,\cos\tfrac12\beta=1+O_{A}\left(\frac{\mathfrak{L}_{\Delta}^2}{\ell}\right), \]
where we denote $\mathfrak{L}_{\Delta}:=\sqrt{\Delta}\log\ell$. As in \S\ref{n3-initial-subsec}, these latter two conditions in turn imply that $m',l\in L'+O_A(\mathfrak{L}_{\Delta}/\sqrt{\ell})$, whence
\begin{equation}\label{n3-essential-support-conditions-Delta}
|T_{\tau}^{\sigma'}(x)|\leq \ell^{-A}\qquad\text{unless}\qquad x\in \pm L'A'N'+O_A\left(\frac{\mathfrak{L}_{\Delta}}{\sqrt{\ell}}\right). \end{equation}
Then, similarly as in Remark~\ref{rmk:supplement}, we obtain for $x\in \Omega$ that
\begin{equation}\label{eq:psi-essential-support-nonminimal}
\psi_{\tau}^{U}(x) \ll_{\Re\mu'} \ell^{-A} \qquad \text{unless} \qquad x\in S+O_A\left( \frac{\mathfrak{L}_{\Delta}}{\sqrt{\ell}}\right).
\end{equation}

Recall the notation \eqref{eq:def-S_j}. By \eqref{n3-essential-support-conditions-Delta}, we can restrict the integration in \eqref{to-estimate} to
\begin{equation*}
S_{j,\Delta}(x) := \left\{ (\alpha,\beta)\in S_j:x^{\alpha,\beta}\in 
\pm L'A'N'+O\left(\frac{\mathfrak{L}_{\Delta}}{\sqrt{\ell}}\right)\right\}
\end{equation*}
at the cost of an $O(\ell^{-10})$ error. In particular, using also \eqref{baseline-nonminimal-eq}, we obtain the following analogue of \eqref{Ij-new-b}:
\begin{equation}\label{Ij-new-b-Delta}
I_j(x)\ll_{\Re\mu'} \min(\ell^{-10}, \Delta\ell\BB_j^2) \quad\text{if}\quad S_{j,\Delta}(x)=\emptyset.
\end{equation}
From now on, we shall focus on the case when $S_{j,\Delta}(x)\neq\emptyset$. Moreover, until further notice, we impose the assumption
\begin{equation}\label{eq:BKY-input-condition1-Delta}
\frac{\mathfrak{L}_{\Delta}}{\sqrt{\ell}} \ll \dist(x,\{\pm\id\})
\end{equation}
with a sufficiently small implied constant (depending only on $\Omega$). In particular, this guarantees
that the error terms $O_A(\mathfrak{L}_{\Delta}^2/\ell)$ and $O_A(\mathfrak{L}_{\Delta}/\sqrt{\ell})$ encountered so far and onward throughout this subsection are sufficiently small in terms of $\Omega$ (for example, for various applications of Lemma~\ref{kmprime-explicit}).

In the light of \eqref{eq:psi-essential-support-nonminimal}, since $\psi_{\tau}^{U}(x)$ is $K$-conjugation invariant, we can and we shall assume that \eqref{K-conj} holds true with $\mathfrak{L}_{\Delta}$ in place of $\mathfrak{L}$. Then, we have the following analogue of \eqref{n3-final-condition}:
\begin{equation}
\label{n3-final-condition-Delta}
\dist(x,\{\pm\id\})\ll\frac{\mathfrak{L}_\Delta}{\BB_j\sqrt{\ell}}.
\end{equation}
From here, the argument in \S\ref{n3-localization-subsec} proceeds as written, with $\mathfrak{L}_{\Delta}$ in place of $\mathfrak{L}$, and with \eqref{baseline-nonminimal-eq} as the baseline estimate (which generalizes the first line of \eqref{n3-essential-support-conditions}). This leads to the following variant of \eqref{eq:geometric-series}:
\begin{equation}\label{eq:geometric-series-variant}
\psi_{\tau}^U(x)\ll_{\Re\mu'} \Delta + \sum_{\substack{0\leq j\leq \log \ell / \log 4\\ \text{$\BB_j$ satisfies \eqref{n3-final-condition-Delta}}}}\Delta\ell\cdot\BB_j^2  \ll\frac{\Delta^2 \log^2\ell}{\dist(x,\{\pm\id\})^2}.
\end{equation}

We now proceed to generalize the arguments in \S\ref{subsec:profiting-from-oscillation} so that we profit from oscillation in the integral representation \eqref{psitau-integral}. The starting point is to generalize the explicit expression for $T_{\tau}^{\sigma'}(x)$ in Lemma~\ref{phase-explicated-lemma}. For $x=\epsilon lm'a'n'$ as in Lemma~\ref{kmprime-explicit} and \eqref{eq:xdecomp}, where as before $l=k[\alpha,\beta,\gamma]\in L$ and $m'=\big(\begin{smallmatrix}m''&\\&1\end{smallmatrix}\big)$ with $m''=\big(\begin{smallmatrix} a&b\\c&d\end{smallmatrix}\big)\in\SL_2(\RR)$, we begin by rewriting \eqref{n-3-tr-explicit-nonminimal-2} as follows:
\begin{equation}
\label{Ttau-explicit-first}
T_{\tau}^{\sigma'}(x)=(2\lambda+1)\epsilon^{\ell+u}\cdot\Re\sum_{\substack{\ell\leq m\leq\lambda\\m\equiv\ell\pmod*{2}}}
\left[e^{i(\alpha+\gamma)}\left(\frac2{a+ib-ic+d}\right)(\cos^2\tfrac12\beta)\right]^mP_{\ell,m,\lambda}(x),
\end{equation}
where
\[P_{\ell,m,\lambda}(x):=\tilde{R}_{\lambda,m}\left(\cos^2\tfrac12\beta\right)\tilde{Q}_{m,\ell}\left(\frac{a^2+b^2+c^2+d^2-2}4\right)\]
with the real polynomials
\begin{align*}
\tilde{R}_{\lambda,m}(v)&:=v^{-m}R_{\lambda,m}(v-1,v)=\sum_{\kappa=0}^{\lambda-m}\binom{\lambda-m}{\kappa}\binom{\lambda+m}{\kappa}(v-1)^{\kappa}v^{\lambda-m-\kappa},\\
\tilde{Q}_{m,\ell}(w)&:=\sum_{j=0}^{\frac{m-\ell}2}\binom{\frac{m+\ell}2-1}j\binom{\frac{m-\ell}2}{j}(-w)^j.
\end{align*}
The identity \eqref{Ttau-explicit-first} takes the place of \eqref{eta1-explicit-b} in this more general context $\lambda=\ell+\Delta-2\in[\ell,2\ell]$. We observe that $\deg\tilde{R}_{\lambda,m}+2\deg\tilde{Q}_{m,\ell}\leq\Delta$. Moreover, $|\tilde{R}_{\lambda,m}(v)|,|\tilde{Q}_{m,\ell}(w)|\leq e^{O(\ell)}$ for $v,w\ll 1$. Hence it is again clear by Lemma~\ref{kmprime-explicit} and \eqref{n3-alphaplusgamma} that $T_{\tau}^{\sigma'}(x)=O(2^{-\ell})$ for $\eta(x)=-1$ and sufficiently small $c_{\Omega}>0$. Therefore, from now on we work with the case $\eta(x)=1$.

Using the notation $l(x):=l$ and $m'(x):=m'$, and invoking \eqref{n3-alphaplusgamma} again, we see that
\begin{equation}\label{eq:pellmlambda-rewritten}
P_{\ell,m,\lambda}(x)=\tilde{R}_{\lambda,m}\left(\sqrt{\omega_1(l(x))}\right)\tilde{Q}_{m,\ell}\left(\omega_2(m'(x))\right),
\end{equation}
where for $y\in\RR^{3\times 3}$ we write
\begin{align}
\notag
\omega_1(y)&:=\big((y_{11}+y_{22})^2+(y_{12}-y_{21})^2\big)/4,\\
\label{omegax-def}
\omega_2(y)&:=\big((y_{11}-y_{22})^2+(y_{12}+y_{21})^2\big)/4=O_2[y_{11}-y_{22},y_{12}+y_{21}].
\end{align}
For subsequent calculations, it is good to keep in mind that $\omega_1$ and $\omega_2$ are bi-$L'$-invariant, as follows from
\[
\omega_1(y)+\omega_2(y)=\frac{y_{11}^2+y_{22}^2+y_{12}^2+y_{21}^2}{2},\qquad
\omega_1(y)-\omega_2(y)=y_{11}y_{22}-y_{12}y_{21}.
\]
Using Lemma~\ref{kmprime-explicit} and \eqref{eq:xdecomp}, we now find that
\begin{equation}
\label{omega1x-def}
\omega_1(l(x))=1+O_2[x_{31},x_{32}].
\end{equation}
To unify notation, we may write for $\delta\in\ZZ_{\geq 0}$ and $\Uplambda\in\ZZ_{\geq 1}$,
\[ H_{\delta,\Uplambda}(w):=\sum_{j=0}^{\delta}\binom{\delta+\Uplambda}j\binom{\delta}j(-w)^j={}_2F_1(-\delta-\Uplambda,-\delta;1;-w), \]
so that
\begin{equation}
\label{RQH}
\tilde{R}_{\lambda,m}(v)=v^{\lambda-m}H_{\lambda-m,2m}\left(\frac{1-v}v\right),\qquad\tilde{Q}_{m,\ell}(w)=H_{\frac{m-\ell}{2},\ell-1}(w).
\end{equation}
It is straightforward to check that
\begin{align*}
H'_{\delta,\Uplambda}(w)&=-(\delta+\Uplambda)\sum_{j=0}^{\delta-1}\binom{\delta+\Uplambda-1}j\binom{\delta}{j+1}(-w)^j\\
&=-(\delta+\Uplambda)\sum_{j=0}^{\delta-1}\sum_{i=1}^{\delta-j}\binom{\delta+\Uplambda-1}j\binom{\delta-i}j(-w)^j\\
&=-(\delta+\Uplambda)\sum_{i=1}^{\delta}H_{\delta-i,\Uplambda-1+i}(w),
\end{align*}
which implies for every $k\in\ZZ_{\geq 0}$ that
\begin{equation}
\label{Hk-derivatives}
H^{(k)}_{\delta,\Uplambda}(w)\ll_k(\delta(\delta+\Uplambda))^k
\max\left\{|H_{\delta',\Uplambda'}(w)|:0\leq\delta'\leq\delta,\ \Uplambda\leq\Uplambda'\leq\Uplambda+\delta-\delta'\right\}.
\end{equation}

The calculations in \eqref{eq:genmatrixcoeff} and \eqref{tau-coefficients-general} relate the polynomials 
$\tilde{Q}_{m,\ell}$ and $\tilde{R}_{\lambda,m}$ to certain matrix coefficients of unitary representations. Specifically, the calculation in \eqref{eq:genmatrixcoeff} combined with \eqref{eq:varphi_vs_U} shows that, for every integer $m\geq\ell$ of the same parity as $\ell$, a unit vector $\bv_m$ can be chosen in the copy of 
$V_{\omega_m}$ inside
$V_{\delta_{\ell}}$ such that, for every matrix $y=\bigl(\begin{smallmatrix} a&b\\c&d\end{smallmatrix}\bigr)\in\SL_2(\RR)$,
\[ \langle\delta_\ell(y)\bv_m,\bv_m\rangle=\left(\frac{2}{a-ib+ic+d}\right)^m\tilde{Q}_{m,\ell}\left(\frac{a^2+b^2+c^2+d^2-2}4\right). \]
In particular, the inequality $|\langle\delta_\ell(y)\bv_m,\bv_m\rangle|\leq 1$ shows that $w=(a^2+b^2+c^2+d^2-2)/4\geq 0$ satisfies
\[ |\tilde{Q}_{m,\ell}(w)|\leq (1+w)^{m/2}. \]
Given any $\delta\in\ZZ_{\geq 0}$, $\Uplambda\in\ZZ_{\geq 1}$, and $w\in\RR_{\geq 0}$
we may apply this conclusion to
\[ \ell=\Uplambda+1,\qquad m=2\delta+\Uplambda+1,\qquad y=\left(\begin{smallmatrix} \sqrt{1+w}+\sqrt{w}&0\\0&\sqrt{1+w}-\sqrt{w}\end{smallmatrix}\right)\in\SL_2(\RR), \]
so that $(\delta,\Uplambda)=(\frac{m-\ell}2,\ell-1)$ and $H_{\delta,\Uplambda}(w)=\tilde{Q}_{m,\ell}(w)$. It follows that
\[ |H_{\delta,\Uplambda}(w)|\leq (1+w)^{\delta+\frac{\Uplambda+1}2},
\qquad\delta\in\ZZ_{\geq 0},\quad\Uplambda\in\ZZ_{\geq 1},\quad w\in\RR_{\geq 0}. \]
Similarly, \eqref{tau-coefficients-general} shows that, writing $H_{\delta,\Uplambda}(w)=v^{m-\lambda}\tilde{R}_{\lambda,m}(v)$ with $v=\cos^2\beta$, $w=(1-v)/v$, and $(\delta,\Uplambda)=(\lambda-m,2m)$, we have the slightly stronger bound $|H_{\delta,\Uplambda}(w)|\leq (1+w)^{\delta+\frac{\Uplambda}2}$. This bound is subject to $2\mid\Uplambda$. Hence we can improve the previously displayed bound to
\[
|H_{\delta,\Uplambda}(w)|\leq (1+w)^{\delta+\left\lfloor\frac{\Uplambda+1}2\right\rfloor},
\qquad\delta\in\ZZ_{\geq 0},\quad\Uplambda\in\ZZ_{\geq 1},\quad w\in\RR_{\geq 0}.
\]
Combining this with \eqref{RQH} and \eqref{Hk-derivatives},
we conclude for every integer $m\in[\ell,\lambda]\subseteq[\ell,2\ell]$ of the same parity as $\ell$, every integer $k\geq 0$,
and every reals $0\leq v\leq 1$ and $w\geq 0$, that
\begin{equation}
\label{RQ-bounds}
\tilde{R}_{\lambda,m}^{(k)}(v)\ll_k (\Delta\ell)^k v^{-m-2k},\qquad \tilde{Q}_{m,\ell}^{(k)}(w)\ll_k (\Delta\ell)^k(1+w)^{m/2}.
\end{equation}

For $(\alpha,\beta)\in S_j$, we shall employ the expression \eqref{Ttau-explicit-first} with $x^{\alpha,\beta}$ in the role of $x$. Our assumption \eqref{eq:BKY-input-condition1-Delta} coupled with the consequence \eqref{n3-final-condition-Delta} of $S_{j,\Delta}(x)\neq\emptyset$ forces that $\BB_j$ is sufficiently small, hence $\beta\in\Omega''$ as in \eqref{sinbeta-nottoosmall}, and also $\eta(x^{\alpha,\beta})=1$. So we obtain the following generalization of \eqref{Ttausigma-explicit}:
\[ T_{\tau}^{\sigma'}(x^{\alpha,\beta})=(2\lambda+1)\epsilon^{\ell+u}\cdot\Re\sum_{\substack{\ell\leq m\leq\lambda\\m\equiv\ell\pmod*{2}}}\phi(x^{\alpha,\beta})^{-m}P_{\ell,m,\lambda}(x^{\alpha,\beta}), \]
where $\phi(x^{\alpha,\beta})$ is exactly the same as in \S\ref{subsec:profiting-from-oscillation}. On the other hand, in light of the analogue of item \eqref{citable-implication} below \eqref{n3-final-condition}, with $S_{j,\Delta}(x)$ in place $S_j(x)$, \eqref{n3-final-condition-Delta} also implies that $x^{\alpha,\beta}\in\pm L'A'N'+O(\mathfrak{L}_{\Delta}/\sqrt{\ell})$. Thence, it follows from Lemma~\ref{kmprime-explicit} that
\[l(x^{\alpha,\beta})_{31},\quad l(x^{\alpha,\beta})_{32},\quad
m'(x^{\alpha,\beta})_{11}-m'(x^{\alpha,\beta})_{22},\quad m'(x^{\alpha,\beta})_{12}+m'(x^{\alpha,\beta})_{21}\ll\frac{\mathfrak{L}_{\Delta}}{\sqrt{\ell}}.\]
Combining \eqref{eq:pellmlambda-rewritten}, \eqref{omegax-def}, \eqref{omega1x-def}, and \eqref{RQ-bounds}, we conclude via Fa\`a di Bruno's formula \cite[Th.~2]{MR602839} that
\begin{equation}\label{eq:fdb2}
\frac{\partial^k}{\partial\beta^k}P_{\ell,m,\lambda}(x^{\alpha,\beta})\ll_k (\Delta\mathfrak{L}_{\Delta}\sqrt{\ell})^k \cdot W(x^{\alpha,\beta})^m,
\end{equation}
where, for $x\in G$,
\[ W(x):=\omega_1(l(x))^{-1/2}(1+\omega_2(m'(x)))^{1/2}=\sqrt{\omega_1(m'(x)) / \omega_1(l(x))}. \]
Here we also use that $\omega_1(l(x))\leq 1\leq\omega_1(m'(x))$.

We are now ready for the adaptation of the application of \cite[Lemma~8.1]{BlomerKhanYoung2013}. In place of \eqref{Ij-Ij-alpha} we have
\begin{gather*}
I_j(x)=\int_0^{2\pi}I_{j,\alpha}'(x)\,d\alpha,\quad I_{j,\alpha}'(x):=\sum_{\substack{\ell\leq m\leq\lambda\\m\equiv\ell\pmod*{2}}} \sum_{\pm}\int_0^{\pi}w_{j,m}'(\alpha,\beta,x)e^{\pm im t(\alpha,\beta,x)}\,d\beta,\\
w_{j,m}'(\alpha,\beta,x):=\frac1{8\pi}(2\lambda+1)\epsilon^{\ell+u}\cdot e^{(\mu'-\rho')(H'(x^{\alpha,\beta}))}f\left(\frac{\sin\beta}{\BB_j}\right)\sin\beta\cdot r(\alpha,\beta,x)^{-m}P_{\ell,m,\lambda}(x^{\alpha,\beta}).
\end{gather*}
Note that $r(\alpha,\beta,x)={W}(x^{\alpha,\beta})$ by \eqref{Ttau-explicit-first} and \eqref{n3-alphaplusgamma}. Now we apply \eqref{eq:fdb} with $(m,\mathfrak{L}_{\Delta})$ in place of $(\ell,\mathfrak{L})$, and combine it with \eqref{eq:fdb2} to get the following analogue of \eqref{w-conditions}:
\[ \frac{\partial^k w_{j,m}'(\alpha,\beta,x)}{\partial\beta^k}\ll_{k,\Re\mu'}XU^{-k},\qquad X:=\ell\BB_j,\, U:=(\ell \BB_{j,\Delta}^{\ast})^{-1}, \]
where we denote shorthand
\[
\BB_{j,\Delta}^{\ast}:=\BB_j+\frac{\Delta \mathfrak{L}_{\Delta}}{\sqrt{\ell}}+\frac1{\ell\BB_j}+\frac{\|\mu'-\rho'\|}{\ell}.
\]
We proceed as in the calculation below \eqref{bjast-def} to get the counterpart of \eqref{Ij-new}:
\begin{equation}\label{Ij-new-Delta}
I_j(x) \ll_{\eta,\Re\mu'} \Delta\ell\BB_j^2\left( \frac{\ell^{\eta}\BB_{j,\Delta}^{\ast}}{\|\by_3\|}+\ell^{-10}\right)
\quad\text{if}\quad S_{j,\Delta}(x)\neq\emptyset.
\end{equation}
Armed with \eqref{Ij-new-b-Delta} and \eqref{Ij-new-Delta}, we obtain the analogue of \eqref{muprime-dependence}:
\begin{align*}
\psi_{\tau}^U(x)
&\ll_{\eta,\Re\mu'} \Delta\sum_{\substack{j\leq\log\ell/\log 4\\\ S_{j,\Delta}(x)\neq\emptyset}}\frac{\ell^{1+\eta}\BB_j^2\BB_{j,\Delta}^{\ast}}{\|\by_3\|}+ \Delta\sum_{j\leq \log\ell/\log 4} \min(\ell^{-10}, \ell\BB_j^2)\\
&\ll\frac{\Delta\ell^{1+\eta}\log\ell}{\|\by_3\|}
\sum_{\substack{j\leq \log \ell / \log 4\\ \text{$\BB_j$ satisfies \eqref{n3-final-condition-Delta}}}}
\left(\BB_j^3+\frac{\BB_j^2\Delta^{3/2}}{\sqrt{\ell}}+\frac{\BB_j}{\ell}+\frac{\BB_j^2\|\mu'-\rho'\|}{\ell}\right)+\ell^{-8}\\
&\ll_{\eta}\frac{\Delta^{7/2} \ell^{2\eta}}{\sqrt{\ell}\dist(x,\{\pm\id\})^3\|\by_3\|}\left(1+\frac{\|\mu'-\rho'\|\dist(x,\{\pm\id\})}{\Delta^{3/2}\sqrt{\ell}}\right).
\end{align*}
Arguing as below \eqref{muprime-dependence}, and using also \eqref{eq:geometric-series-variant}, we conclude that
\begin{equation}\label{psitau-m-bound3}
\psi_{\tau}^U(x)\ll_{\eps,\Re\mu'}\frac{\Delta^{7/2}\ell^{\eps}}{\sqrt{\ell}\dist(x,\{\pm\id\})^3\dist(x,R)}\left(1+\frac{\|\mu'-\rho'\|\dist(x,\{\pm\id\})}{\Delta^{3/2}\sqrt{\ell}}\right).
\end{equation}

To sum up, we have proved the bounds \eqref{eq:geometric-series-variant} and \eqref{psitau-m-bound3} assuming \eqref{eq:BKY-input-condition1-Delta} holds with a sufficiently small implied constant (depending only on $\Omega$). However, \eqref{eq:psi-nonminimal-baseline} shows that these bounds are also valid when this condition is not met. Putting everything together, we obtain the following generalization of Theorem~\ref{thm4}.

\begin{theorem}\label{thm4b}
Let $U=U^{\sigma',\mu'}$ and $\sigma'=\delta_{\ell}\boxtimes\sgn^u$ be as in Theorem~\ref{thm4}. Let $\tau=\tau_{\lambda}^{\iota}$,
where $\lambda\in[\ell,2\ell]$ and $\iota:=(-1)^{\ell+u}$. Put $\Delta:=\lambda-\ell+2$. Then, for all $\eps,A>0$ and $x\in\Omega$, we have
\[\psi^{U}_{\tau} (x)\ll_{\eps, A,\Re\mu',\Omega}
\frac{\ell^{1+\eps}\Delta}{1 + \ell/\Delta \cdot d_{\{\pm\id\}}^2+\ell^{3/2}/\Delta^{5/2}\cdot c_{\mu'}d_{\{\pm\id\}}^3d_R+(\sqrt{\ell/\Delta} \cdot d_S)^A}, \]
where $d_{\{\pm\id\}}$, $d_R$, $d_S$, and $c_{\mu'}$ are as in Theorem~\ref{thm4}.
\end{theorem}

\begin{proof}
Let us look at the four terms in the denominator. If the first term dominates, then we are done by \eqref{eq:psi-nonminimal-baseline}. If the second term dominates, then we are done by \eqref{eq:geometric-series-variant}. If the third term dominates, then we are done by \eqref{psitau-m-bound3}, with a bit of room to spare. Finally, let us assume that the fourth term dominates. Then, we need to prove that
\[\psi^{U}_{\tau} (x)\ll_{\eps, A,\Re\mu',\Omega}\ell^{1+\eps}\Delta(\sqrt{\ell/\Delta} \cdot d_S)^{-A}.\]
By \eqref{eq:psi-essential-support-nonminimal} and the boundedness of $d_S$, we are done unless
$\sqrt{\ell/\Delta}\cdot d_S\ll_A\log\ell$; in this remaining case, we are again done by \eqref{eq:psi-nonminimal-baseline}.
\end{proof}

\section{Paley--Wiener theory for \texorpdfstring{$\SL_n^{\pm}(\RR)$}{SLn(R)} and beyond}\label{PWsection}

\subsection{A non-smooth Paley--Wiener theorem}
In this subsection, $G$ is a semisimple Lie group belonging to the Harish-Chandra class. In particular, both $G$ and $G^0$ have finite center, and $G$ has finitely many connected components. We denote by $\Gtemp$ the tempered unitary dual of $G$, by $\Gadm$ the admissible dual of $G$, and by $\Gqsf$ the set of quasisimple representations\footnote{By a quasisimple representation we mean a continuous representation of $G$ possessing an infinitesimal character. We do not require the center of $G$ to act by scalars.} of $G$ of finite length up to infinitesimal equivalence. Thus
\[\Gtemp\subseteq\widehat{G}\subseteq\Gadm\subseteq\Gqsf.\]

Let $C_c(G)$ be the convolution algebra of continuous, compactly supported functions $f:G\to\CC$. Let $K$ be a maximal compact subgroup of $G$, and let $\tau$ be an irreducible unitary representation of $K$ with normalized character $\chi_{\tau}:=d_{\tau}\tr(\tau)$. We denote by $I_c(G,\tau)$ the convolution (sub)algebra of functions $f\in C_c(G)$ satisfying
\begin{equation}\label{fprop1}
f(k^{-1}xk)=f(x),\qquad x\in G,\quad k\in K,
\end{equation}
and
\begin{equation}\label{fprop2}
\ov{\chi_{\tau}} * f * \ov{\chi_{\tau}} = f.
\end{equation}
In the previous display, the convolution $*$ is meant over $K$ with respect to the Haar probability measure. If $(\varpi,V_\varpi)$ is a Hilbert space representation of $G$ whose restriction to $K$ is unitary (as we can assume without loss of generality), then for any $f\in C_c(G)$ the operator
\[
\varpi(f):=\int_G f(x)\varpi(x)\,dx\in\End(V_{\varpi})
\]
is well-defined. Similarly, the operator
\[
\varpi(\ov{\chi_\tau}):=\int_K \ov{\chi_\tau(k)}\varpi(k)\,dk\in\End(V_{\varpi})
\]
is well-defined, and it projects $V_\varpi$ onto its $\tau$-isotypic subspace $V_\varpi(\tau)$. It follows that if $f\in I_c(G,\tau)$, then $\varpi(f)$ commutes with the action of $K$, it leaves $V_\varpi(\tau)$ invariant, and it annihilates $V_\varpi(\tau)^\perp$.

The next theorem is a variant of \cite[Th.~2]{DFJ}, allowing disconnected groups and non-smooth test functions\footnote{We shall use that, for connected semisimple Lie groups with finite center, \cite[Th.~2]{DFJ} holds without any assumption on the central character. This is in agreement with a remark below \cite[Th.~1]{DFJ}.}. In order to formulate the theorem, we recall the notations $G=KAN$ and $P=MAN$ from \S\ref{Pseries}. Let $\theta$ be the Cartan involution of $G$ relative to $K$. Let $\mfg$ be the Lie algebra of $G$, and let $\mfb$ be the Lie algebra of a maximal torus of $M$. Then $\mfh:=\mfa+\mfb$ is a $\theta$-invariant Cartan subalgebra of $\mathfrak{g}$, and $\mfh^r:=\mfa+i\mfb$ is the real form of $\mfh_\CC$ generated by the co-root vectors. 
Accordingly, we shall decompose every $\Lambda\in\mfh^*_\CC$ as
\[\Lambda=\Re\Lambda+i\Im\Lambda\qquad\text{with}\qquad\Re\Lambda,\Im\Lambda\in(\mfh^r)^*.\]
We equip $(\mfh^r)^*$ with a Weyl-invariant Euclidean norm, and then define
\[
\|\Lambda\|^2:=\|\Re\Lambda\|^2+\|\Im\Lambda\|^2,\qquad\Lambda\in\mfh^*_\CC.
\]
Note that if $G$ is a split group (as in our application), then $\mfh=\mfh^r=\mfa$.

\begin{theorem}[Paley--Wiener theorem]\label{Thm5} There exists some $A>0$ depending only on $G$ such that the following holds. Let $R>0$ be arbitrary. Assume $h:\mfh^*_{\CC}\to\CC$ is an entire, Weyl-invariant function satisfying the Paley--Wiener growth condition
\[
h(\Lambda)\ll_h e^{R\|\Re\Lambda\|} (1+\|\Lambda\|)^{-A},\qquad\Lambda\in\mfh^*_\CC.
\]
Then there exists a unique function $f\in I_c(G,\tau)$ such that
\begin{equation}\label{scalarvalue}
\varpi(f)=h(\Lambda_{\varpi}) \varpi(\ov{\chi_{\tau}}),\qquad \varpi\in\Gqsf,
\end{equation}
where $\Lambda_{\varpi}$ is the infinitesimal character of $\varpi$, and $\varpi(\ov{\chi_{\tau}})$ is the projection of $V_{\varpi}$ onto $V_{\varpi}(\tau)$. This function is supported in a compact set depending only on $R$, and it obeys the inversion formula
\begin{equation}\label{inversion}
f(x)=\int_{\Gtemp(\tau)} h(\Lambda_{\varpi}) \psi_{\tau}^{\varpi}(x^{-1})\,d\varpi,
\end{equation}
where $d\varpi$ is the Plancherel measure of $G$, and $\Gtemp(\tau)$ is the set of those elements of $\Gtemp$ whose $\tau$-isotypic subspace is nontrivial.
\end{theorem}

\begin{proof} Throughout the proof, we suppress the dependence of the implied constants on $G$. Let $H:=G^0$ be the connected component of the identity, and let $L:=K\cap H$ be its maximal compact subgroup. We equip $H$ with the restriction of the Haar measure on $G$, and extend functions on $H$ by zero to $G$. We shall use below that if $\varpi\in\Gqsf$, then $\varpi|_H\in\GGqsf$ by Clifford theory. Moreover, $\varpi$ and $\varpi|_H$ have the same infinitesimal character, because $G$ and $H$ have the same Lie algebra. We shall also utilize the finite set
\[S:=\left\{\upsilon\in\widehat{L}:[\tau:\upsilon]\geq 1\right\}.\]

The key idea of the proof is that
\begin{equation}\label{pl4}
\left\{\Re\Lambda_\varpi:\varpi\in\Gtemp(\tau)\right\}
\end{equation}
is a finite set. Indeed, let $\varpi\in\Gtemp(\tau)$ be arbitrary. By \cite[Cor.~6.2]{Trombi}, $\varpi$ is unitarily equivalent to a subrepresentation of $U^{\sigma',\mu'}$ (cf. \S\ref{P'series}), where $\sigma'\in\widehat{M'}$ is in the discrete series and $\mu'\in i{\mfa'}^*$ is purely imaginary\footnote{Here $M'$, $\sigma'$, $\mu'$ are not the same as for the fixed representation $\pi\in\widehat{G}$. We stick to these symbols for notational simplicity.}. Then, by \eqref{indres} and Frobenius reciprocity,
\[1\leq[\varpi|_K:\tau]\leq[U^{\sigma',\mu'}|_K:\tau]=[\Ind_{K'}^K\sigma'|_{K'}:\tau]=[\sigma'|_{K'}:\tau|_{K'}].\]
Hence $\sigma'$ contains an irreducible component of $\tau|_{K'}$, and so by \cite[Cor.~4.10]{HW2} (which in turn is based on \cite[Lem.~70]{HC}), there are only finitely many choices for $\sigma'$. However, $\Re\Lambda_\varpi=\Re\Lambda_{\sigma'}$ by \cite[Prop.~11.43]{KV}, so \eqref{pl4} is a finite set.

We shall see in a similar fashion that
\begin{equation}\label{multiplicitybound}
1\leq[\varpi|_K:\tau]\leq\dim\tau,\qquad\varpi\in\Gtemp(\tau).\end{equation}
Indeed, by the subrepresentation theorem \cite[Th.~8.21]{CM}, there exists $\sigma\in\widehat{M}$ and $\mu\in\mfa_\CC^*$ such that $\varpi$ is infinitesimally equivalent with a subrepresentation of $U^{\sigma,\mu}$ (cf. \S\ref{Pseries}). However, this implies by \eqref{indres1} and Frobenius reciprocity that
\[1\leq[\varpi|_K:\tau]\leq[U^{\sigma,\mu}|_K:\tau]=[\Ind_M^K\sigma:\tau]=[\tau|_M:\sigma]\leq\dim\tau.\]

Now we fix an entire, Weyl-invariant function $k:\mfh_\CC^*\to\CC$ satisfying $k(0)=1$ and
\[k(\Lambda)\ll_N e^{\|\Re\Lambda\|}(1+\|\Lambda\|)^{-N},\qquad\Lambda\in\mfh_\CC^*,\quad N\in\NN.\]
For any $\eps\in(0,1)$, we define
\[h_\eps(\Lambda):=h(\Lambda)k(\eps\Lambda),\qquad\Lambda\in\mfh_\CC^*.\]
Hence $h_\eps:\mfh_\CC^*\to\CC$ is an entire, Weyl-invariant function satisfying
\begin{equation}\label{hdecay}
h_\eps(\Lambda)\ll_{N,h} e^{(R+\eps)\|\Re\Lambda\|}(1+\|\Lambda\|)^{-A}(1+\|\eps\Lambda\|)^{-N},\qquad\Lambda\in\mfh_\CC^*,\quad N\in\NN,
\end{equation}
and also
\[\lim_{\eps\to 0+}h_\eps(\Lambda)=h(\Lambda),\qquad\Lambda\in\mfh_\CC^*.\]
By \cite[Th.~2]{DFJ} and our initial remarks, for each $\upsilon\in S$, there exists $f_{\eps,\upsilon}\in C_c^\infty(H)$ such that
\[\varpi(f_{\eps,\upsilon})=h_\eps(\Lambda_\varpi)\varpi(\ov{\chi_{\upsilon}}),\qquad\varpi\in\Gqsf.\]
Here $\varpi(\ov{\chi_{\upsilon}})$ is the projection of $V_\varpi$ onto $V_\varpi(\upsilon)$. The sum of these projections can be compared with $\varpi(\ov{\chi_\tau})$, the projection of $V_\varpi$ onto $V_\varpi(\tau)$:
\[V_\varpi(\tau)\subseteq\bigoplus_{\upsilon\in S}V_\varpi(\upsilon)\qquad\Longrightarrow\qquad
\varpi(\ov{\chi_\tau})\sum_{\upsilon\in S}\varpi(\ov{\chi_{\upsilon}})=\varpi(\ov{\chi_\tau}).\]
Hence the function
\[f_\eps:=\ov{\chi_\tau}\ast\sum_{\upsilon\in S}f_{\eps,\upsilon}\in C_c^\infty(G)\]
satisfies
\[\varpi(f_\eps)=\varpi(\ov{\chi_\tau})\sum_{\upsilon\in S}h_\eps(\Lambda_\varpi)\varpi(\ov{\chi_{\upsilon}})=h_\eps(\Lambda_\varpi)\varpi(\ov{\chi_\tau}),\qquad\varpi\in\Gqsf.\]
By a remark following \cite[Th.~1]{DFJ}, each $f_{\eps,\upsilon}$ is supported in a compact set depending only on $R$, hence the same is true of $f_\eps$.

By \cite[Th.~27.3]{HC3} or the more general \cite[Th.~6.27]{HW},
\[f_\eps(x)=\int_{\Gtemp} \tr(\varpi(f_\eps)\varpi(x^{-1}))\,d\varpi=\int_{\Gtemp(\tau)} h_\eps(\Lambda_\varpi)\,\psi_\tau^\varpi(x^{-1})\,d\varpi.\]
By the bounds \eqref{multiplicitybound}--\eqref{hdecay} and the finiteness of the set \eqref{pl4}, the integrand on the right-hand side is $\ll_{\tau,R,h}(1+\|\Lambda_\varpi\|)^{-A}$. Now let $A>0$ be sufficiently large in terms of $G$. Then, by \cite[Th.~25.1]{HC3} or the more general \cite[Lem.~3.3]{HW2},
\[f_\eps(x)\ll_{\tau,R,h}1,\qquad x\in G.\]
Moreover, by the dominated convergence theorem, the pointwise limit
\[f(x):=\lim_{\eps\to 0+}f_\eps(x),\qquad x\in G,\]
exists, and it equals the right-hand side of \eqref{inversion}. It is clear that $f$ is supported in a compact set depending only on $R$, while the continuity of $f$ follows from \eqref{inversion} and the dominated convergence theorem. We can also deduce \eqref{scalarvalue} from the dominated convergence theorem. Indeed, let $(\varpi,V_\varpi)$ be any Hilbert space realization of an element of $\Gqsf$, and let $v\in V_\varpi$. Then
\[\|\varpi(f_\eps)v-\varpi(f)v\|\leq\int_G|f_\eps(x)-f(x)|\cdot\|\varpi(x)v\|\,dx\to 0\qquad\text{as}\qquad\eps\to 0+,\]
hence
\[\varpi(f)v=\lim_{\eps\to 0+}\varpi(f_\eps)v=\lim_{\eps\to 0+}h_\eps(\Lambda_\varpi)\varpi(\ov{\chi_\tau})v=h(\Lambda_\varpi)\varpi(\ov{\chi_\tau})v.\]
Finally, the uniqueness of $f$ and the validity of \eqref{fprop1}--\eqref{fprop2} follow from \eqref{scalarvalue} and the Plancherel theorem.
\end{proof}

\subsection{Harish-Chandra parameters}
Now we return to our concrete situation $G=\SL_n^{\pm}(\RR)$. We let the Weyl group $W=S_n$ act from the right on $\CC^n$ as follows:
\[(z_1,\dotsc,z_n)w:=(z_{w(1)},\dotsc,z_{w(n)}).\]
This induces an action on the trace $0$ subspace
\[
\CC^n_0:=\{(z_1,\dotsc,z_n)\in\CC^n:z_1+\dotsb+z_n=0\}.
\]
We identify $\CC^n$ with its dual via the standard scalar product on $\CC^n$. Correspondingly, we shall think of both $\mfa_{\CC}$ and $\mfa_{\CC}^*$ as $\CC^n_0$. Let $\Lambda_{\pi}$ be the infinitesimal character of $\pi$, the representation generated by $\Phi$. In light of the Harish-Chandra isomorphism, we shall think of $\Lambda_{\pi}$ as an element of $\mfa_{\CC}^*/W=\CC^n_0/W$. In particular, \eqref{eq:def-U-sigma'-mu'} and \cite[Prop.~11.43]{KV} yield that the infinitesimal character of $\pi=U^{\sigma',\mu'}$ is
\begin{equation}\label{picharacter}
\Lambda_{\pi} = \left(\frac{\ell_1-1}{2}+\mu'_1,\frac{1-\ell_1}{2}+\mu'_1,\dotsc,\frac{\ell_r-1}{2}+\mu'_r,\frac{1-\ell_r}{2}+\mu'_r,\mu'_{r+1},\dotsc,\mu'_{r+s}\right)W.
\end{equation}
In general, we denote the Harish-Chandra parameter of an arbitrary $\varpi\in\widehat{G}$ by
\[
\Lambda_{\varpi}:=(\bu+i\bv)W\in\CC^n/W,\qquad \bu+i\bv=(u_1+iv_1,\dotsc,u_n+iv_n).
\]
We record that
\[
\sum_{j=1}^n (u_j+i v_j)=0,\qquad -\ov{\Lambda_{\varpi}}=\Lambda_{\varpi},
\]
where the first equation is by the zero-trace assumption, and the second equation follows from unitarity.

We introduce a partial ordering on $\RR^n/W$ analogous to \eqref{eq:standard-dominance}. Assume
\begin{equation}\label{eq:aa'}
\ba:=(a_1,\dotsc,a_n)\in\RR^n \qquad \text{and} \qquad \ba':=(a'_1,\dotsc,a'_n)\in\RR^n.
\end{equation}
We say that $\ba'W\preceq \ba W$ if there exists $w\in W$ such that for any $w'\in W$ and $1\leq k\leq n$, we have that
\[
\sum_{j=1}^k a'_{w'(j)} \leq \sum_{j=1}^k a_{w(j)}.
\]
In practice, we rearrange $\ba$ by $W$ to a decreasing sequence $a_1\geq \dotsc \geq a_n$, and then we can use the identity permutation $w={\id}$.
With this notation, we have the following.
\begin{theorem}\label{theorem:K-type-barrier}
Let $\pi\in\widehat{G}$ be generic with (unique) minimal $K$-type $\tau$. Then
\[\Re\Lambda_{\varpi}\preceq\Re\Lambda_{\pi},\qquad\varpi\in\Gtemp(\tau).\]
\end{theorem}
\begin{proof}
First let $n=2$. If $\varpi$ is a principal series representation, then the conclusion reduces to the obvious statement $\{(0,0)\}\preceq\Re\Lambda_{\pi}$. Now let $\varpi$ be a discrete series representation $\delta_{\ell'}$. As $\tau$ occurs in $\varpi$, it is of the form $\omega_\ell$ with $\ell\geq\ell'$. Hence $\pi$ is the discrete series representation $\delta_{\ell}$, and the conclusion reduces to the obvious statement
\[\left(\frac{\ell'-1}{2},\frac{1-\ell'}{2}\right)S_2\preceq\left(\frac{\ell-1}{2},\frac{1-\ell}{2}\right)S_2.\]

Now let $n\geq 3$. By \cite[Th.~14.10.20]{GH}, we can assume that $\pi=U^{\sigma',\mu'}$, where $\sigma'$ satisfies \eqref{sigma'} and \eqref{eq:ordering-ell-s}. Hence
\begin{equation}\label{Lambdasigma'}
\Lambda_{\sigma'} = \left(\frac{\ell_1-1}{2},\frac{1-\ell_1}{2},\dotsc,\frac{\ell_{r}-1}{2},\frac{1-\ell_{r}}{2},
0,\dotsc,0\right)W,
\end{equation}
and (as mentioned before)
\[
\Lambda_{\pi} = \left(\frac{\ell_1-1}{2}+\mu'_1,\frac{1-\ell_1}{2}+\mu'_1,\dotsc,\frac{\ell_r-1}{2}+\mu'_r,\frac{1-\ell_r}{2}+\mu'_r,\mu'_{r+1},\dotsc,\mu'_{r+s}\right)W.
\]
We claim that $\Lambda_{\sigma'}\preceq\Re\Lambda_{\pi}$. We shall derive this from a refinement of the symmetry $-\ov{\Lambda_{\pi}}=\Lambda_{\pi}$, which refinement in turn is a consequence of \cite[Th.~14.10.20]{GH}. Namely, if we collect the entries of $\Lambda_\pi$ corresponding to the $\ell_j$'s that are equal to a given $\ell\in\ZZ_{\geq 2}$, say
\[\Lambda_{\pi,\ell}=\left(\frac{\ell-1}{2}+\mu'_{j+1},\frac{1-\ell}{2}+\mu'_{j+1},\dotsc,\frac{\ell-1}{2}+\mu'_{j+m},\frac{1-\ell}{2}+\mu'_{j+m}\right)S_{2m},\]
then in fact $-\ov{\Lambda_{\pi,\ell}}=\Lambda_{\pi,\ell}$. The same kind of symmetry also holds for the remaining $s$ entries of $\Lambda_{\pi}$ (which do not correspond to any $\ell_j$). From here it is clear that $\Re\Lambda_{\pi,\ell}$ dominates the corresponding part of $\Lambda_{\sigma'}$, and similarly $(\Re\mu'_{r+1},\dotsc,\Re\mu'_{r+s})S_s$ dominates $\{(0,\dotsc,0)\}$. The relation $\Lambda_{\sigma'}\preceq\Re\Lambda_{\pi}$ follows easily.

Now let $\varpi\in\Gtemp(\tau)$. As in the proof of Theorem~\ref{Thm5}, we embed $\varpi$ into a generalized principal series representations $U^{\sigma'',\mu''}$ with a purely imaginary parameter $\mu''$, so that $\Re\Lambda_\varpi=\Re\Lambda_{\sigma''}$. (In fact $U^{\sigma'',\mu''}$ is irreducible by the main theorem of \cite{Tadic}, but we shall not need this in the present proof.) Without loss of generality,
\[\sigma''=\delta_{\ell'_1}\boxtimes\dotsb\boxtimes\delta_{\ell'_{r'}}\boxtimes\sgn\boxtimes\dotsb\boxtimes\sgn\boxtimes1\boxtimes\dotsb\boxtimes1,\]
where $\ell'_1\geq\dotsb\geq\ell'_{r'}\geq 2$. Then
\begin{equation}\label{Lambdasigma''}
\Lambda_{\sigma''} = \left(\frac{\ell'_1-1}{2},\frac{1-\ell'_1}{2},\dotsc,\frac{\ell'_{r'}-1}{2},\frac{1-\ell'_{r'}}{2},
0,\dotsc,0\right)W,
\end{equation}
and it remains to verify that $\Lambda_{\sigma''}\preceq\Lambda_{\sigma'}$. This follows from the explicit knowledge of the minimal $K$-type $\tau$ of $U^{\sigma',\mu'}$ and the minimal $K$-type $\tau'$ of $U^{\sigma'',\mu''}$. By \eqref{eq:highest-weight}, the highest weights of $\tau$ and $\tau'$ are of shape
\[(\ell_1,\dotsc,\ell_r,1,\dotsc,1,0,\dotsc,0)\qquad\text{and}\qquad(\ell'_1,\dotsc,\ell'_{r'},1,\dotsc,1,0,\dotsc,0).\]
By assumption, $\tau'\preceq\tau$, hence we have a dominance relation in $\ZZ_{\geq 0}^{\lfloor n/2\rfloor}$ of shape
\[(\ell'_1,\dotsc,\ell'_{r'},1,\dotsc,1,0,\dotsc,0)\preceq(\ell_1,\dotsc,\ell_r,1,\dotsc,1,0,\dotsc,0).\]
We can change each zero entry on the right-hand side to one. After this, we can change each zero entry on the left-hand side to one:
\[(\ell'_1,\dotsc,\ell'_{r'},1,\dotsc,1)\preceq(\ell_1,\dotsc,\ell_r,1,\dotsc,1).\]
Finally, we apply the map $\lambda\mapsto(\lambda-1)/2$ to each entry, obtaining
\[
\left(\frac{\ell'_1-1}{2},\dotsc,\frac{\ell'_{r'}-1}{2},0,\dotsc,0\right)
\preceq
\left(\frac{\ell_1-1}{2},\dotsc,\frac{\ell_r-1}{2},0,\dotsc,0\right).
\]
We multiply both vectors by $-1$, reverse the order of the entries, and glue them to the original vectors. If $n$ is odd, we insert an extra zero into the middle. This way we obtain a dominance relation in $\left(\frac{1}{2}\ZZ\right)^n$ of shape
\[
\left(\frac{\ell'_1-1}{2},\frac{\ell'_2-1}{2},\dotsc,\frac{1-\ell'_2}{2},\frac{1-\ell'_1}{2}\right)
\preceq
\left(\frac{\ell_1-1}{2},\frac{\ell_2-1}{2},\dotsc,\frac{1-\ell_2}{2},\frac{1-\ell_1}{2}\right).
\]
By \eqref{Lambdasigma'} and \eqref{Lambdasigma''}, this is precisely the sought relation $\Lambda_{\sigma''}\preceq\Lambda_{\sigma'}$.

To sum up, we have shown that
\[\Re\Lambda_\varpi=\Lambda_{\sigma''}\preceq\Lambda_{\sigma'}\preceq\Re\Lambda_{\pi}.\]
The proof is complete.
\end{proof}

We conclude this section by introducing a metric on the set of Harish-Chandra parameters. For two parameters $\Lambda,\Lambda'\in\CC^n_0/W$ we set 
\[d(\Lambda,\Lambda'):= \min\left\{\|(\bu+i\bv)-(\bu'+i\bv')\|: \bu+i\bv \in \Lambda ,\, \bu'+i\bv' \in \Lambda'\right\},\]
where on the right-hand side we take the Euclidean norm. Note that
\begin{equation}\label{eq:parameter-distance-note}
\|(\bu+i\bv)-(\bu'+i\bv')\|\asymp \max_{1\leq j\leq n} \max\left(|u_j-u'_j|, |v_j-v'_j|\right).
\end{equation}
By the rearrangement inequality,
\begin{equation}\label{eq:parameter-distance-note2}
d(\Lambda,\Lambda') \geq d(\Re\Lambda,\Re\Lambda') = \|\bu-\bu'\|,
\end{equation}
where $\bu\in\Re\Lambda$ and $\bu'\in\Re\Lambda'$ are the unique decreasing vector representatives. With this metric at hand, we define the ``etalon ball'' used in \eqref{spectralcontent}:
\begin{equation}\label{eq:def-unit-ball-in-dual}
B(\pi):=\left\{\varpi\in\widehat{G}:d(\Lambda_{\pi},\Lambda_{\varpi}) < n\right\}.
\end{equation}

\subsection{Construction of a spectral localizer}
We first introduce
\[
\eta_j:=2(n-j)+1, \qquad 1\leq j\leq n.
\]

\begin{lemma}\label{lemma:eta} Let $\ba,\ba'\in\RR^n$ as in \eqref{eq:aa'} satisfy $a_1\geq \dotsb \geq a_n$ and $\ba'W\preceq \ba W$. Then we have
\[
\sum_{j=1}^n |a_j-a'_{w(j)}| \leq \sum_{j=1}^n \eta_j (a_j-a'_{w(j)}) ,\qquad w\in W.
\]
\end{lemma}
\begin{proof} We prove by induction on $n$. For $n=1$, the statement is obvious. So let us assume that $n\geq 2$ and the statement is true for $n-1$ in the role of $n$. That is, we have, for any 
\[
\bb:=(b_1,\dotsc,b_{n-1})\in\RR^{n-1} \qquad \text{and} \qquad \bb':=(b'_1,\dotsc,b'_{n-1})\in\RR^{n-1}
\]
satisfying $b_1\geq\dotsb\geq b_{n-1}$ and $\bb'S_{n-1}\preceq \bb S_{n-1}$,
that
\[
\sum_{j=1}^{n-1} |b_j-b'_{w(j)}| \leq \sum_{j=1}^{n-1} (\eta_j-2) (b_j-b'_{w(j)}),\qquad w\in S_{n-1}.
\]

Now let $\ba,\ba'\in\RR^n$ and $w\in W$ be as in the statement. Obviously $\preceq$ is inherited to the first $n-1$ coordinates of $\ba$ and $\ba'w$, hence by induction (applied to the identity element of $S_{n-1}$),
\begin{equation}\label{eq:in-eta-lemma-0}
\sum_{j=1}^{n-1} |a_j-a'_{w(j)}| \leq \sum_{j=1}^{n-1} (\eta_j-2) (a_j-a'_{w(j)}).
\end{equation}
We claim that
\begin{equation}\label{eq:in-eta-lemma}
|a_n-a'_{w(n)}| \leq (a_n-a'_{w(n)}) + 2 \sum_{j=1}^{n-1} (a_j-a'_{w(j)}).
\end{equation}
Indeed, if $a_n\geq a'_{w(n)}$, then the left-hand side equals the first term of the right-hand side, and we get \eqref{eq:in-eta-lemma} by the nonnegativity of the second term of the right-hand side; while if $a_n<a'_{w(n)}$, then the left-hand side is $a'_{w(n)}-a_n$, and \eqref{eq:in-eta-lemma} follows from $\sum a_j\geq \sum a'_{w(j)}$. Adding \eqref{eq:in-eta-lemma-0} and \eqref{eq:in-eta-lemma}, the desired conclusion follows.
\end{proof}

For $1\leq j\leq n$ we set
\begin{equation}\label{eq:E-def}
E_{j}(z):=e^{2\eta_j z} \cdot \frac{\sinh(z)}{z},\qquad z\in\CC^{\times},
\end{equation}
continued analytically as $E_j(0):=1$. Assume we are given a target Harish-Chandra parameter
\[
\Lambda_0=(\by+i\bt)W \in \CC^n_0/W
\]
such that
\[\by=(y_1,\dotsc,y_n)\in\RR^n,\qquad\bt=(t_1,\dotsc,t_n)\in\RR^n\]
satisfies $y_1\geq\dotsb\geq y_n$. Assume further that
\begin{equation}\label{eq:selfdual}
\Lambda_0=-\ov{\Lambda_0},
\end{equation}
as is the case with Harish-Chandra parameters of irreducible unitary representations. Then we set
\begin{equation}\label{eq:F_w-def}
F_w(\xi):=\prod_{j=1}^n E_{j}(\xi_{w(j)}-y_j-it_j),\qquad \xi\in\CC^n,\quad w\in W.
\end{equation}
Note that this is a function only on $\CC^n$, not on $\CC^n/W$. We immediately have
\begin{equation}\label{eq:F_id=1}
F_{{\id}}(\by+i\bt)=1.
\end{equation}

\begin{lemma}\label{lemma:F_w-bounds} Assume $\bu+i\bv\in\CC^n_0$ satisfies $\bu W\preceq\by W$. Then
\begin{equation}\label{eq:F_w-bound1}
F_w(\bu+i\bv)\ll \exp\left(-\sum_{j=1}^n |y_j-u_{w(j)}|\right),\qquad w\in W,
\end{equation}
and also
\begin{equation}\label{eq:F_w-bound2}
F_w(\bu+i\bv)\ll \|(\by+i\bt)-(\bu+i\bv)w\|^{-1}, \qquad w\in W.
\end{equation}
\end{lemma}

\begin{proof} First we prove \eqref{eq:F_w-bound1}. Using the standard bound $\left| \sinh(z)/z \right| \ll \exp(|\Re z|)$, we can estimate each factor on the right-hand side of \eqref{eq:F_w-def} as
\[
\begin{split}
|E_{j}(\xi_{w(j)}-y_j-it_j)| & = \left| e^{2\eta_j(\xi_{w(j)}-y_j-it_j)} \frac{\sinh(\xi_{w(j)}-y_j-it_j)}{\xi_{w(j)}-y_j-it_j} \right| \\ & \ll \exp\left( 2\eta_j (\Re\xi_{w(j)} - y_j) + |\Re\xi_{w(j)} - y_j| \right).
\end{split}
\]
Setting $\xi:=\bu+i\bv$, and taking the product over all $1 \leq j \leq n$, we obtain by Lemma~\ref{lemma:eta} that
\[
F_{w}(\bu+i\bv) \ll \exp \left( \sum_{j=1}^n \left( 2\eta_j(u_{w(j)}-y_j) + |y_j - u_{w(j)}| \right) \right) \leq \exp\left(-\sum_{j=1}^n |y_j - u_{w(j)}|\right).
\]
The proof of \eqref{eq:F_w-bound1} is complete.

Now we prove \eqref{eq:F_w-bound2}. Recall \eqref{eq:parameter-distance-note}, and the maximum there is attained by a difference $|y_{k}-u_{w(k)}|$ or a difference $|t_{k}-v_{w(k)}|$ for some fixed index $1\leq k\leq n$, and we go by cases accordingly.

\textbf{Case 1:} The maximal difference is attained in the real part, meaning that
\[
|y_{k} - u_{w(k)}| \gg \|(\by+i\bt)-(\bu+i\bv)w\|.
\]
By \eqref{eq:F_w-bound1}, we have
\[
F_w(\bu+i\bv) \ll \exp\left( - \sum_{j=1}^n |y_j - u_{w(j)}| \right) \leq \exp(-|y_{k} - u_{w(k)}|) \ll |y_{k}-u_{w(k)}|^{-1},
\]
proving \eqref{eq:F_w-bound2} in this case.

\textbf{Case 2:} The maximal difference is attained in the imaginary parts, meaning that
\[
|t_{k} - v_{w(k)}| \gg \|(\by+i\bt)-(\bu+i\bv)w\|.
\]
For the specific index $k$, we utilize $|\sinh(z)|\leq \exp(|\Re z|)$ to deduce the alternative bound
\[
\left|E_k(u_{w(k)}+iv_{w(k)}-y_{k}-it_{k})\right| \leq \frac{1}{|v_{w(k)}-t_{k}|} \exp\left( 2\eta_{k}(u_{w(k)} - y_{k}) + |u_{w(k)} - y_{k}| \right).
\]
Estimating all the other factors as in the proof of \eqref{eq:F_w-bound1}, we obtain
\[ F_w(\bu+i\bv) \ll \|(\by+i\bt)-(\bu+i\bv)w\|^{-1} \exp\left( - \sum_{j=1}^n \left( 2\eta_j(y_j - u_{w(j)}) - |y_j - u_{w(j)}| \right) \right). \]
By Lemma~\ref{lemma:eta}, the $j$-sum is nonnegative, and we arrive at \eqref{eq:F_w-bound2} in this case, too.
\end{proof}

Now we arrive at the main result of this subsection. 

\begin{theorem}\label{theorem:concrete-paley-wiener-function} For any $n\in\ZZ_{\geq 1}$ and $A\in\RR_{>0}$, there exists an $R\in\RR_{>0}$ with the following property. For any $\Lambda_0\in\CC^n_0/W$ with $\Lambda_0=-\ov{\Lambda_0}$, there exists an entire function $h:\CC^n_0/W\to\CC$ satisfying the Paley--Wiener condition
\begin{equation}\label{eq:paley-wiener-growth-condition}
h(\Lambda)\ll_{A,\Lambda_0} e^{R\|\Re\Lambda\|} (1+\|\Lambda\|)^{-A}
\end{equation}
and the bounds
\begin{alignat}{2}
\label{eq:paley-wiener-decay}
0\leq h(\Lambda)&\ll_A (1+d(\Lambda_0,\Lambda))^{-A},&&\qquad\Lambda=-\ov{\Lambda},\quad\Re\Lambda\preceq\Re\Lambda_0,\\
\label{eq:paley-wiener-decay-exp}
0\leq h(\Lambda)&\ll_A \exp(-A\cdot d(\Re\Lambda_0,\Re\Lambda)),&&\qquad\Lambda=-\ov{\Lambda},\quad\Re\Lambda\preceq\Re\Lambda_0,
\end{alignat}
\begin{equation}\label{eq:paley-wiener-pi-amplified}
h(\Lambda_0)\asymp_{A} 1.
\end{equation}
\end{theorem}

\begin{remark} The proof below makes it clear that the lower bounds in \eqref{eq:paley-wiener-decay}
and \eqref{eq:paley-wiener-decay-exp} only need the assumption $\Lambda=-\ov{\Lambda}$. We display them together with the upper bounds for aesthetic and practical reasons.
\end{remark}

\begin{proof} First we prove a ``non-unitary'' variant of the result, where instead of \eqref{eq:paley-wiener-decay} and \eqref{eq:paley-wiener-pi-amplified} we require
\begin{alignat}{2}
\label{eq:paley-wiener-decay-variant}
h(\Lambda)&\ll_A (1+d(\Lambda_0,\Lambda))^{-A},&&\qquad\Re\Lambda\preceq\Re\Lambda_0,\\
\label{eq:paley-wiener-decay-exp-variant}
h(\Lambda)&\ll_A \exp(-A\cdot d(\Re\Lambda_0,\Re\Lambda)),&&\qquad\Re\Lambda\preceq\Re\Lambda_0,
\end{alignat}
\begin{equation}\label{eq:paley-wiener-pi-amplified-variant}
|h(\Lambda_0)|\asymp_{A} 1.
\end{equation}
To this aim, we define the entire, Weyl-invariant functions
\[
h_k(z):= \sum_{w\in W} (F_w(z))^k, \qquad z\in\CC^n_0,\quad k\in\NN.
\]
By \eqref{eq:E-def} and \eqref{eq:F_w-def}, these functions satisfy the Paley--Wiener condition
\begin{equation}\label{eq:paley-wiener-growth-condition2}
h_k(\Lambda) \ll_{k,\Lambda_0} e^{(2n^2+n)k\|\Re\Lambda\|} (1+\|\Lambda\|)^{-k},\qquad \Lambda\in\CC^n_0/W.
\end{equation}
We quote a consequence of Turán's power sum theorem \cite[Cor. on p.~85]{Turan1984}: ``For integer $m\geq 0$ and complex $\eta_j$ there is an integer $\nu_0$ with
\[
m+1 \leq \nu_0 \leq m+n
\]
so that the inequality
\[
\bigl|\eta_1^{\nu_0}+\dotsb+\eta_n^{\nu_0}\bigr| \geq 2 \left(\frac{n}{8e(m+n)}\right)^n \left(\max_j |\eta_j|\right)^{\nu_0}
\]
holds.'' We apply this result with $(\lceil A\rceil,n!)$ in the role of $(m,n)$, and with the $n!$ complex numbers $F_w(\Lambda_0)$ in the role of the $\eta_j$'s. We observe that the maximum on the right-hand side is at least $1$ by \eqref{eq:F_id=1}. Therefore, denoting by $k\in\{\lceil A\rceil+1,\dotsc,\lceil A\rceil+n!\}$ the yielded $\nu_0$ and putting $h:=h_k$, the bounds \eqref{eq:paley-wiener-decay-variant}, \eqref{eq:paley-wiener-decay-exp-variant}, \eqref{eq:paley-wiener-pi-amplified-variant} follow from Lemma~\ref{lemma:F_w-bounds}, noting that in \eqref{eq:F_w-bound1} the $j$-sum is at least $\|\by-\bu w\|$.
In addition, putting $R:=(2n^2+n)(\lceil A\rceil+n!)$, the bound \eqref{eq:paley-wiener-growth-condition} follows from \eqref{eq:paley-wiener-growth-condition2}.

Now we modify our construction to fulfill the original requirements \eqref{eq:paley-wiener-growth-condition}, \eqref{eq:paley-wiener-decay}, \eqref{eq:paley-wiener-decay-exp}, \eqref{eq:paley-wiener-pi-amplified}. Simply, we double $R$ and replace $h$ with the function
\begin{equation}\label{eq:squared-h}
\Lambda \mapsto h(\Lambda)\ov{h(-\ov{\Lambda})},\qquad \Lambda\in\CC^n_0/W,
\end{equation}
which is entire by the reflection principle and nonnegative for ``unitary'' Harish-Chandra parameters:
\[h(\Lambda)\ov{h(-\ov{\Lambda})}=|h(\Lambda)|^2,\qquad\Lambda=-\ov{\Lambda}.\]
We can see that \eqref{eq:paley-wiener-growth-condition} remains true, while \eqref{eq:paley-wiener-decay}, \eqref{eq:paley-wiener-decay-exp}, \eqref{eq:paley-wiener-pi-amplified} follow from \eqref{eq:selfdual}, \eqref{eq:paley-wiener-decay-variant}, \eqref{eq:paley-wiener-decay-exp-variant}, \eqref{eq:paley-wiener-pi-amplified-variant}.
\end{proof}

\section{The amplified pretrace inequality}\label{sec6}
The rest of the paper is devoted to the proofs of Theorems~\ref{thm1} and~\ref{thm2}. Without loss of generality, we may assume $n \geq 3$, since Theorem~\ref{thm1} is known for $n=2$ \cite{IS,Xia}.

\subsection{Hecke operators}
We follow \cite[Ch.~3]{Shimura} and \cite[\S4]{BM-IMRN} in our exposition of the Hecke operators. To ease the notation, we shall identify
\begin{equation}\label{eq:id-SL-GL}
G \qquad \text{and} \qquad \GL_n(\RR) / \{\diag(a,\dotsc,a):a\in\RR^{\times}_{>0}\}.
\end{equation}
Accordingly, we shall lift any function defined on $G$ to a function on $\GL_n(\RR)$.

For an $n$-tuple
\[
\mbm=(m_1,\dotsc,m_n)\in\ZZ_{\geq 1}^n \qquad \text{with} \qquad m_n\mid\dotsb\mid m_1
\]
we define the Hecke operator
\[
T_\mbm(F):=\sum_j F \circ M_j,\qquad F\in L^2(\Gamma \bs G),
\]
where the matrices $M_j\in \ZZ^{n\times n}$ are chosen to satisfy
\[
\Gamma \diag(m_1,\dotsc,m_n) \Gamma = \bigcup_j \Gamma M_j
\]
with a disjoint union on the right-hand side. By \cite[Prop.~3.16]{Shimura} we have the decomposition
\[T_\mbm=\prod_p T_{v_p(\mbm)}(p),\qquad T_{\ba}(p):=T_{p^{a_1},\dotsc,p^{a_n}},\]
so that the (commutative) ring $R$ generated by the $T_\mbm$'s is the tensor product of the (commutative) rings $R_p$ generated by the $T_{\ba}(p)$'s. By \cite[Th.~3.20]{Shimura}, each $R_p$ is generated by the $n$ basic Hecke operators of the form $T(p,\dotsc,p,1,\dotsc,1)$.

For fixed $\ba$ and $p\to\infty$, the eigenvalue of $T_{\ba}(p)$ on constant functions is asymptotically $p^A$, where
\begin{equation}\label{Hecke-norm}
A: = \sum_{j=1}^n (n+1-2j) a_j.
\end{equation}
In contrast, the Ramanujan conjecture predicts that the eigenvalue $\lambda_{\varpi,\ba}(p)$ of $T_{\ba}(p)$ on each irreducible cuspidal subspace $V_\varpi\subseteq L^2(\Gamma \bs G)$ is $O_\ba(p^{A/2})$. This follows from \cite[(4.3)]{BM-IMRN} and the subsequent display in that paper.

Let $L$ be a large parameter to be chosen later, and let $\mcP$ be a set of primes in $[L,2L]$. Then, for $\ell\in\mcP$ and $1\leq j\leq n$ we set
\[
T_{[j]}(\ell):=T_{(j,0,\dotsc,0)}(\ell),\qquad
x_{[j]}(\ell):=\frac{|\lambda_{\pi,(j,0,\dotsc,0)}(\ell)|}{\lambda_{\pi,(j,0,\dotsc,0)}(\ell)},
\]
with the convention $0/0=0$ if the denominator vanishes. Finally, we introduce
\[
\mcT_j:= \sum_{\ell\in\mcP} \frac{x_{[j]}(\ell)}{\ell^{j(n-1)/2}} T_{[j]}(\ell),\qquad 1\leq j\leq n,
\]
and
\[
\mcT:=\sum_{j=1}^n \mcT_j \mcT_j^{\ast}.
\]
By \cite[Cor.~4.3]{BM-IMRN}, $V_{\pi}$ is an eigenspace of $\mcT$ with nonnegative eigenvalue
\begin{equation}\label{eigen-TT}
\lambda_{\mcT}(\pi)\gg (\#\mcP)^2.
\end{equation}
On the other hand, by \cite[Lem.~4.4]{BM-IMRN} we have
\begin{align*}
\mcT = & \sum_{j=1}^{n} \sum_{\substack{\ell_1,\ell_2\in\mcP \\ \ell_1\neq\ell_2}} \frac{x_{[j]}(\ell_1)\ov{{x}_{[j]}(\ell_2)} }{(\ell_1\ell_2)^{j(n-1)/2}}T_{[j]}(\ell_1)T^{\ast}_{[j]}(\ell_2) \\ & + \sum_{j=1}^{n} \sum_{i=0}^j \sum_{\ell \in \mcP} \frac{|x_{[j]}(\ell)|^2 c_{ij}(\ell)}{\ell^{(n-1)(j-i)}}T_{(2j-2i, j-i, \ldots, j-i, 0)}(\ell)
\end{align*}
for certain $c_{ij}(\ell) \ll 1$. 

We now argue as in \cite[\S6]{BM-IMRN}. For an integral matrix $\gamma \in \ZZ^{n\times n}$, let $\Delta_j$ denote the $j$-th determinantal divisor, i.e. the greatest common divisor of all $j\times j$ subdeterminants. For $m, l \in \ZZ_{\geq 1}$ let
\begin{equation}\label{Sml}
S(m, l) := \{\gamma \in \ZZ^{n\times n} : |\det \gamma| = m,\ \Delta_1 = 1,\ \Delta_2 = l \}.
\end{equation}
The group $\Gamma$ acts on this set from the left and the right. 
The significance of this definition is that all matrices in the coset decomposition of $$\Gamma \diag(\ell^{2j-2i}, \ell^{j-i}, \ldots, \ell^{j-i}, 1)\Gamma$$ are in $S(\ell^{n(j-i)}, \ell^{j-i})$ and all matrices in the coset decomposition of $$\Gamma \diag (\ell_1^j, 1, \ldots, 1) \Gamma \cdot \Gamma\diag(\ell_2^j, \ldots, \ell_2^j, 1)\Gamma = \Gamma\diag(\ell_1^j\ell_2^j, \ell_2^j, \ldots, \ell_2^j, 1)\Gamma$$ are in $S(\ell_1^j\ell_2^{j(n-1)}, \ell_2^j)$. We conclude that
\begin{equation*}
\begin{split}
\mcT(F)(x) = & \sum_{j=1}^n \sum_{\substack{\ell_1,\ell_2 \in \mcP \\ \ell_1\neq \ell_2}} \frac{1}{L^{(n-1)j}} \sum_{\gamma \in \Gamma\bs S(\ell_1^j\ell_2^{(n-1)j}, \ell_2^j)}c_j(\gamma, \ell_1, \ell_2) F(\gamma x) \\
&+ \sum_{0 \leq i \leq j \leq n} \sum_{\ell \in \mcP} \frac{1}{L^{(n-1)(j-i)}} \sum_{\gamma\in \Gamma\bs S(\ell^{n(j-i)}, \ell^{j-i})} c_{ij}(\gamma, \ell) F(\gamma x)
\end{split}
\end{equation*}
for certain $c_j(\gamma, \ell_1, \ell_2), c_{ij}(\gamma, \ell) \ll 1$. In the second line, we detach the terms $i = j$ (for which $\gamma=\id$), and include the terms $i<j$ in the first line with $j$ playing the role of $j-i$. We obtain
\begin{equation}\label{TT}
\begin{split}
\mcT(F)(x) = c_0 \cdot \# \mcP \cdot F( x) + \sum_{j=1}^n \sum_{\ell_1 ,\ell_2 \in \mcP} \frac{1}{L^{(n-1)j}} \sum_{\gamma \in \Gamma\bs S(\ell_1^j\ell_2^{(n-1)j}, \ell_2^j)}c_j(\gamma, \ell_1, \ell_2) F(\gamma x) 
\end{split}
\end{equation}
with coefficients $c_0, c_j(\gamma, \ell_1, \ell_2) \ll 1$. 

\subsection{Positivity}\label{subsec:positivity}
We follow \cite[\S3.2]{BHMM} to establish the amplified pre-trace inequality with as little technical effort as possible. In particular, we shall not use the spectral decomposition of $L^2(\Gamma \bs G)$.

We fix some large positive parameter $A>0$ to be specified later. In particular, let $A$ exceed the particular $A$ furnished by Theorem~\ref{Thm5}, and let $R>0$ be given by Theorem~\ref{theorem:concrete-paley-wiener-function} accordingly. We apply Theorem~\ref{theorem:concrete-paley-wiener-function} with these parameters and $\Lambda_0:=\Lambda_{\pi}$, obtaining a suitable $h:\CC^n_0/W\to\CC$. Finally, we let $f \in I_{c}(G,\tau)$ be the corresponding function provided by Theorem~\ref{Thm5}. Then, in particular, $\pi(f)$ acts by the scalar $h(\Lambda_{\pi})\geq 0$ on $V_{\pi}(\tau)$ and annihilates $V_{\pi}(\tau)^{\perp}$. We consider the convolution operator $R(f)$ on $L^2(\Gamma\bs G)$ given by
$$R(f)(\psi)(x): = \int_G f(y) \psi(xy)\, dy = \int_{\Gamma\bs G} k_f(x, y) \psi(y)\,dy$$
for $\psi \in L^2(\Gamma \bs G)$, $x \in G$, where
$$k_f(x, y): = \sum_{\gamma \in \Gamma} f(x^{-1} \gamma y).$$ 
Without loss of generality, $f$ factors as $g\ast\check g$, where $g\in I_c(G,\tau)$ is given by
\[\check g(x):=\ov{g(x^{-1})},\qquad x\in G.\]
For example, one can always replace $f$ by $f\ast\check f$ and $h$ by \eqref{eq:squared-h}. Then $R(f)=R(g)R(g)^\ast$ is a positive operator\footnote{The positivity of $R(f)$ follows more directly from the spectral decomposition of $L^2(\Gamma \bs G)$, without recourse to the factorization $f=g\ast\check g$.}, which commutes with the positive operator $\mcT$ from the previous subsection. We consider now the positive, bounded operator $\mcA := R(f) \mcT$ on $L^2(\Gamma \bs G)$. 

Let $\mfB$ be a finite orthonormal system of eigenfunctions $\phi$ of $\mcA$ with (not necessarily distinct) eigenvalues $(c_{\phi}(\mcA))_{\phi\in\mfB}$. Then, $\mcA$ preserves the orthodecomposition
$L^2(\Gamma\bs G)=\mathrm{Span}(\mfB)\oplus\mathrm{Span}(\mfB)^{\perp}$,
and for any $\psi\in L^2(\Gamma\bs G)$ the corresponding decomposition $\psi=\psi_1+\psi_2$ with
\[\psi_1:=\sum_{\phi\in\mfB}\langle\psi,\phi\rangle\phi\qquad \text{and} \qquad \psi_2:=\psi-\psi_1\]
gives
$$
\langle \mcA\psi,\psi\rangle=\langle \mcA\psi_1,\psi_1\rangle+\langle \mcA\psi_2,\psi_2\rangle
\geq\langle \mcA\psi_1,\psi_1\rangle=\sum_{\phi\in\mfB}c_{\phi}(A)|\langle\psi,\phi\rangle|^2.
$$
In particular, if $\mfB=\{\phi_1,\dotsc,\phi_{\dim\tau_{\pi}}\}$ denotes an orthonormal basis of $V_{\pi,\tau_{\pi}}$, we obtain as in \cite[(3.12)]{BHMM} the following amplified pre-trace inequality where we use \eqref{eq:paley-wiener-pi-amplified} and \eqref{eigen-TT} on the left-hand side and \eqref{TT} on the right-hand side:
\begin{align}
\notag (\#\mcP)^2 \sum_{\phi \in \mfB} | \phi(x)|^2 & \ll \#\mcP \sum_{\gamma \in \Gamma}| f(x^{-1}\gamma x)|+ \sum_{j=1}^n \sum_{\ell_1 ,\ell_2 \in \mcP} \frac{1}{L^{(n-1)j}} \sum_{\gamma\in S(\ell_1^j \ell_2^{(n-1)j}, \ell_2^j) } |f(x^{-1} \gamma x)| \\
& \label{finalpretrace} \ll \#\mcP \| f \|_{\infty}+ \sum_{j=1}^n \sum_{\ell_1 ,\ell_2 \in \mcP} \frac{1}{L^{(n-1)j}} \sum_{\gamma\in S(\ell_1^j \ell_2^{(n-1)j}, \ell_2^j) } |f(x^{-1} \gamma x)|.
\end{align}
Here, by our earlier convention \eqref{eq:id-SL-GL},
$f(x^{-1}\gamma x)$ really means $f(x^{-1}\tilde\gamma x)$, where
\begin{equation}\label{tildegamma}
\tilde\gamma:=\gamma/|\det \gamma|^{1/n},
\end{equation}
and we have also used that $f=g\ast\check{g}$ (as a function on $G$) is supported in a compact set depending only on $A$. 
Note that $\hat{f}(V)$ in \cite[(3.12)]{BHMM} is the trace of $\pi(f)$ (cf. \cite[(2.18)]{BHMM}) acting on a vector space of dimension $2\ell + 1$, so that $\hat{f}(V)/(2\ell + 1)$ in \cite[(3.12)]{BHMM} plays the role of $h(\Lambda_{\pi})$ in Theorem~\ref{Thm5}.

\subsection{Applying the Plancherel formula}
To use \eqref{finalpretrace} effectively, we need to estimate the function $f$. We accomplish this by combining \eqref{inversion} and \eqref{eq:paley-wiener-decay}. It is convenient to introduce
\[
\mcZ:=\Delta(\pi)\cdot \dim \tau_{\pi}.
\]
Here it is useful to comment on the size of these quantities. First,
the explicit descriptions of the Plancherel measure \cite[(6.3)]{KS} and the infinitesimal character \eqref{picharacter} reveal that $\Delta(\pi)$ is of the same size as the regularized Weyl discriminant of $\pi$. That is, for any $\bz\in\Lambda_\pi$, we have
\begin{equation}\label{eq:Delta-KS}
\Delta(\pi)\asymp\prod_{1 \leq i < j \leq n}(1+ |\bz_i-\bz_j|)\ll (1+\|\pi\|_\infty)^{n(n-1)/2}.
\end{equation}
Second, by the explicit highest weight \eqref{eq:highest-weight} and the dimension formulae \eqref{eq:dimformula_1}, \eqref{eq:dimformula_2}, \eqref{eq:dimformula_3},
\begin{equation}\label{eq:dim-tau-size}
\dim\tau_{\pi}\ll(1+\|\bm\ell\|_{\infty})^{\lfloor(n-1)^2/4\rfloor}.
\end{equation}
These immediately imply that
\begin{equation}\label{eq:Z-equiv}
\log \mcZ \asymp \log (1+\|\pi\|_{\infty}) \asymp \log (1+\|\Lambda_{\pi}\|).
\end{equation}

We start out from \eqref{inversion} to see that
\begin{equation}\label{eq:f-bound-0}
|f(x)| \leq \int_{\Gtemp(\tau_{\pi})} h(\Lambda_{\varpi}) |\psi_{\tau_{\pi}}^{\varpi}(x^{-1})| \,d\varpi.
\end{equation}
Every $\varpi$ here satisfies $\Re\Lambda_{\varpi}\preceq \Re\Lambda_{\pi}$ by Theorem~\ref{theorem:K-type-barrier}. We claim that
\begin{equation}\label{eq:multiplicitybound-final}
[\varpi:\tau_{\pi}] \ll (1+d(\Lambda_{\pi},\Lambda_{\varpi}))^{\lfloor n^2/4\rfloor}.
\end{equation}
For $n=2$, this is trivial as the left-hand side is $0$ or $1$. For $n\geq 3$, we use the discussion at the beginning of \S\ref{sec:generalized-spherical-trace-function-SL} to write $\pi=U^{\sigma',\mu'}$ and $\varpi=U^{\sigma'',\mu''}$ with $|\Re\mu'|<1/2$ and $\Re \mu'' = 0$. Then we apply Lemma~\ref{lemma:multiplicitybound} with $(\sigma'',\mu'',\varpi, \tau_{\varpi},\tau_{\pi})$ in the role of $(\sigma',\mu',U,\tilde{\tau},\tau)$. The bound \eqref{eq:multiplicitybound-final} now follows upon noting that (cf. \eqref{Lambdasigma'}, \eqref{Lambdasigma''}, \eqref{eq:parameter-distance-note2})
\begin{equation}\label{eq:lambda-diff-bound}
\|\bm\lambda - \tilde{\bm\lambda}\|_{\infty} \ll 1+d(\Lambda_{\sigma'},\Lambda_{\sigma''}) \ll 1+d(\Lambda_{\pi},\Lambda_{\varpi}).
\end{equation}
We note for future reference the following similar but simpler inequality:
\begin{equation}\label{eq:normdifference}
\|\pi\|_\infty-\|\varpi\|_\infty\ll 1+d(\Lambda_{\pi},\Lambda_{\varpi}).
\end{equation}
Indeed, this follows readily from the approximations (cf. \eqref{eq:unorm}, \eqref{picharacter})
\[\|\pi\|_\infty=2\|\Re\Lambda_\pi\|_\infty+\|\Im\Lambda_\pi\|_\infty+O(1),\qquad
\|\varpi\|_\infty=2\|\Re\Lambda_\varpi\|_\infty+\|\Im\Lambda_\varpi\|_\infty+O(1).\]
On the other hand, from \cite[(6.3)]{KS} and \eqref{eq:Delta-KS} we also see that
\begin{equation}\label{eq:RN-estimate}
d\varpi \ll \Delta(\varpi)\,d\sigma''\,d\mu''
\ll \Delta(\pi)\cdot(1+d(\Lambda_{\pi},\Lambda_{\varpi}))^{n(n-1)/2}\,d\sigma''\,d\mu'',
\end{equation}
where $d\sigma''$ is the counting measure and $d\mu''$ is an appropriate Lebesgue measure. Now we combine \eqref{eq:f-bound-0}, \eqref{eq:multiplicitybound-final}, \eqref{eq:RN-estimate} in two ways. We fix the parameter $A>0$ introduced in \S\ref{subsec:positivity} as $A:=100n^2$, and use that  $f$ is supported in a compact set depending only on $A$.

In the first round, we apply \eqref{eq:psi-baseline-bound} with $(\varpi,\tau_\pi,(\supp f)^{-1})$ in the role of $(U,\tau,\Omega)$ and \eqref{eq:paley-wiener-decay} with $(\Lambda_\pi,\Lambda_\varpi)$ in the role of $(\Lambda_0,\Lambda)$ to infer that
\begin{equation}\label{trivial-f}
\| f \|_{\infty} \ll \mcZ.
\end{equation}
We note at this point that with the unamplified version of \eqref{finalpretrace} (setting $\mcT:=\id$, $\mcA:=R(f)$ in \S\ref{subsec:positivity} without any Hecke operators\footnote{In order to use \eqref{finalpretrace} literally, take a fixed $L\geq 1$, and let $\mcP$ be the set of primes in $[L,2L]$.}), we arrive at
\[
\sum_{\phi \in \mfB} | \phi(x)|^2 \ll \sum_{\gamma\in \Gamma} |f(x^{-1}\gamma x)|.
\]
Applying \eqref{trivial-f} together with $x\in\Omega$, we finally obtain the trivial bound\footnote{It is remarkable that the so-called trivial bound has only been proven now.} \eqref{triv}.
For the rest of the argument, it is convenient to assume that $\|\pi\|_{\infty}$ is large enough, since Theorems~\ref{thm1} and~\ref{thm2} are immediate from \eqref{triv} if $\|\pi\|_{\infty}\leq C$ for any fixed $C>0$ (which may depend on $I$ in the case of Theorem~\ref{thm2}). Then $\mcZ$ and $\|\Lambda_{\pi}\|$ are also large by \eqref{eq:Z-equiv}.

In the second round, we cut the integral in \eqref{eq:f-bound-0} as follows. First, by \eqref{eq:psi-baseline-bound}, \eqref{eq:paley-wiener-decay}, \eqref{eq:Z-equiv}, we conclude for every $x\in\supp(f)$ that
\[
\int\limits_{\substack{\Gtemp(\tau_{\pi}) \\ d(\Lambda_{\pi},\Lambda_{\varpi}) \geq \|\pi\|_{\infty}^{1/100}}} h(\Lambda_{\varpi}) |\psi_{\tau_{\pi}}^{\varpi}(x^{-1})| \,d\varpi\ll 1.
\]
Secondly, in the range $d(\Lambda_{\pi},\Lambda_{\varpi}) < \|\pi\|_{\infty}^{1/100}$, we apply Theorem~\ref{thm3b} and Remark~\ref{rmk-to-thm3b} with $(\varpi,\tau_\pi,(\supp f)^{-1})$ in the role of $(U,\tau,\Omega)$. Using also \eqref{eq:paley-wiener-decay}, \eqref{eq:Z-equiv}, \eqref{eq:lambda-diff-bound}, \eqref{eq:normdifference}, we conclude for every $x\in\supp(f)$ that
\[
\int\limits_{\substack{\Gtemp(\tau_{\pi}) \\ d(\Lambda_{\pi},\Lambda_{\varpi}) < \|\pi\|_{\infty}^{1/100}}} h(\Lambda_{\varpi}) |\psi_{\tau_{\pi}}^{\varpi}(x^{-1})| \,d\varpi \ll \mcZ(1+\mcZ^c \dist(x,K))^{-1/3},
\]
where $c\in(0,3)$ is a constant depending only on $n$. The last two displays (valid on the whole support of $f$) then imply that
\begin{equation}\label{eq:f-decay}
f(x)\ll \mcZ \left(1+\mcZ^{c}\dist(x,K)\right)^{-1/3},\qquad x\in G.
\end{equation}
We shall use this bound in the proof of Theorem~\ref{thm1}. In the special case when $n=3$, $r=s=1$, and $\mu'$ lies in a bounded domain $I$, we can say more. First, we observe that $\mcZ\asymp_I \ell^4$. Then we argue as above, but instead of Theorem~\ref{thm3b} we apply Theorem~\ref{thm4b} with $(\varpi,\tau_\pi,(\supp f)^{-1})$ in the role of $(U,\tau,\Omega)$. The result is
\begin{equation}\label{eq:f-decay-2}
f(x) \ll_{\eps,I} \ell^{4+\eps} (1 + \ell \cdot \dist(x, \{\pm \id\})^2 )^{-1}, \qquad x\in G.
\end{equation}
We shall use this bound in the proof of Theorem~\ref{thm2}.

\section{Counting}\label{counting}
In this section we complete the proofs of Theorems~\ref{thm1} and~\ref{thm2}.

\subsection{Proof of Theorem~\ref{thm1}}
This is based on the results in \cite{BM}. Theorem~\ref{thm3b} provides decay of the spherical trace function as soon as $x$ moves away from $K$, which is the same input as in \cite[(2.4)]{BM} (with different constants). We therefore estimate the contribution of matrices $\gamma$ slightly away from $K$ directly by the decay bound in Theorems~\ref{thm3b}, and we use the argument in \cite{BM} to bound effectively the number of $\gamma$ very close to $K$. As in \cite[(2.7)]{BM} we define for $a, b, M\in\ZZ_{\geq 1}$ and a positive matrix $Q \in \text{Pos}_n(\RR)$ the quantity
\begin{equation*} 
\begin{split}
\mcS(Q, a, b, M) := \Big\{\gamma \in \ZZ^{n\times n} : \ &\gamma^T Q \gamma = (ab^{n-1})^{2/n} Q + O\Bigl((ab^{n-1})^{(2-M)/n}\Bigr), \\
& \Delta_1(\gamma) = 1, \, \Delta_2(\gamma) = b\Big\}.
\end{split}
\end{equation*}
We will apply this notation with
\[
Q: = x^{-T} x^{-1} \in \text{Pos}_n(\RR)
\]
lying in a fixed compact set depending on $\Omega$.
For any non-empty set $\mcP \subseteq [L, 2L]$ of primes, we group the $\gamma$'s in the rightmost expression in \eqref{finalpretrace} according to
\begin{equation}\label{eq:distance-cut}
\dist(x^{-1}\tilde\gamma x,K) > c_{\Omega,M} L^{-nM} \qquad \text{or} \qquad \dist(x^{-1}\tilde\gamma x,K) \leq c_{\Omega,M} L^{-nM},
\end{equation}
where the constant $c_{\Omega,M}>0$ is chosen (in terms of $\Omega$ and $M$) to satisfy the following:
\[
\left\{\gamma\in S(\ell_1^j\ell_2^{j(n-1)},\ell_2^j):\dist(x^{-1}\tilde{\gamma}x,K)\leq c_{\Omega,M} L^{-nM}\right\} \subseteq \mcS(Q, \ell_1^j, \ell_2^j, M).
\]
According to the dichotomy \eqref{eq:distance-cut},
in the first case, we estimate the number of matrices trivially by applying that each entry is $O(L^n)$ (using that $x\in\Omega$ and $f$ is supported in a fixed compact set) together with \eqref{eq:f-decay} for the contribution of an individual one. In the second case, we estimate an individual contribution trivially by \eqref{trivial-f}. All in all, we arrive at the basic inequality
\begin{equation}\label{basic}
\sum_{\phi\in\mfB}|\phi(x)|^2 \ll \mcZ\left(\frac{1}{\#\mcP} + \mcZ^{-c/3} L^{n^3 +nM/3} + \sum_{j=1}^n \frac{1}{(\# \mcP)^2} \sum_{\ell_1, \ell_2 \in \mcP} \frac{\#\mcS(Q, \ell_1^j, \ell_2^j, M)}{L^{j(n-1)}}\right),
\end{equation}
noting that the implied constant here also depends on $M$, suppressed from the notation.
As we explain below, it follows from the argument in \cite[\S5--\S7]{BM} that there exist some $M\in\ZZ_{\geq 1}$ (depending only on $n$), some $L=\mcZ^{\alpha}$ (depending on $Q$ but satisfying 
$1\ll \alpha\leq c/(4n^3+nM)$), and some prime set $\mcP\subseteq [L,2L]$ (also depending on $Q$) of size $\gg L^{1/2}$ such that
\[
\#\mcS(Q, \ell_1^j, \ell_2^{j}, M) \ll L^{j(n-1)-1/2},\qquad \ell_1,\ell_2\in\mcP,\quad 1 \leq j\leq n.
\]
In the light of \eqref{basic}, this proves Theorem~\ref{thm1} (we provide the details below).

To analyze $\mcS(Q, \ell_1^j, \ell_2^j, M)$, we combine \cite[Lem.~8 and~Prop.~1]{BM} into the following lemma. 
\begin{lemma}\label{BM-lemma} There exist constants $c_1, c_2, c_3 > 0$ depending at most on $n$ with the following property. Let $L_0 > 2$ and let $M, D_1, D_2 \geq 1$ satisfy 
\begin{equation}\label{cond}
D_2 \geq D_1^{n(n+1)/2}, \quad M \geq c_1D_1^{n(n+1)/2} D_2^{n(n+1)/2 + 1}.
\end{equation}
There exist $0 \leq i, k < n(n+1)/2$ and two sets $\mcD, \mcQ \subseteq \ZZ_{\geq 1}$ of cardinality at most $n(n-1)/2$ and $n$, respectively, with the following properties.

Put $\mcL := L_0^{(D_1D_2)^{i+1}}$. Then it holds that $\mcD, \mcQ \subseteq [1, O(\mcL^{c_2 D_1^k})]$. Let $\mcP$ be the set of all primes $p$ in $[\mcL^{D_1^{k+1}}, 2 \mcL^{D_1^{k+1}}]$ coprime to all elements in $\mcQ$ and such that $-d$ is a quadratic non-residue modulo $p$ for each $d \in \mcD$. Then 
\[
\#\mcS(Q, q^j, p^j, M) \ll p^{j(n-2) + \frac{c_3}{D_1}}
\]
holds for all $p, q \in \mcP$ and $1 \leq j \leq n$. 
\end{lemma}
We can now conclude the proof as in \cite[\S7]{BM}. By the Chinese remainder theorem and quadratic reciprocity, the set $\mcP$ in Lemma~\ref{BM-lemma} can be described by congruence conditions modulo a number $m \ll \mcL^{c_4D_1^k}$ for some $c_4>0$. By a Linnik-type bound for primes in arithmetic progressions \cite[Lem.~9]{BM}, we conclude that $D_1 \geq c_{5}$ implies that
$$\#\mcP \gg \mcL^{D_1^{k+1} - c_{6} D_1^k} \geq \mcL^{\frac{1}{2}D_1^{k+1}}.$$
With this choice of $D_1$ we now specify the other parameters in Lemma~\ref{BM-lemma}. We 
fix some $D_2$ and $M$ satisfying \eqref{cond}, and we put $L_0: = \mcZ^{\eta}$ 
for some small constant $\eta > 0$ to be specified in a moment and define $\mcL: = L_0^{(D_1D_2)^{i+1}}$ as in Lemma~\ref{BM-lemma}. Now we apply \eqref{basic} with
\[
L: = \mcL^{D_1^{k+1}} = \mcZ^{\eta (D_1D_2)^{i+1}D_1^{k+1} } = \mcZ^{\alpha},
\]
say, to obtain
$$\sum_{\phi\in\mfB}|\phi(x)|^2 \ll \mcZ \left(\frac{1}{L^{1/2}} + \mcZ^{-c/3 + \alpha (n^3+nM/3)} + \frac{1}{L^{1/2}}\right).$$
Choosing $\eta$ (hence $\alpha$) to be a sufficiently small positive number depending only on $n$, the proof is complete.

\subsection{Proof of Theorem~\ref{thm2}}
In this subsection we specialize to $n=3$ and recall the definition of $S(m, l)$ in \eqref{Sml}. These are matrices with Smith normal form $\diag(m/l, l, 1)$. In particular, we must have $l^2 \mid m$. We also recall our notation \eqref{tildegamma} for integral matrices of non-zero determinant. We start with two counting lemmata. The first one is elementary, the second one uses spectral theory. 

\begin{lemma}\label{count1} Let $m, l \in \ZZ_{\geq 1}$ and $ \Delta \geq0$. Then, for any $\eps>0$, we have
\[\#\{\gamma \in S(m, l) : \dist(\tilde{\gamma}, \{\pm\id\}) \leq \Delta \} \ll_\eps \left(1 + m^{1/3}\Delta\right)^{3} \left(1 + m^{1/3}\Delta + \frac{m^{2/3}\Delta^2}{l}\right)^{3+\eps}.\]
\end{lemma} 

\begin{proof} Assume that $\gamma\in S(m,l)$ satisfies $\dist(\tilde{\gamma}, \{\pm\id\}) \leq \Delta$. Then, $\gamma = \pm D \cdot \id + \delta $ with $D := m^{1/3}$, a matrix $\delta\in[-D\Delta,D\Delta]^{3\times 3}$, and an appropriate sign $\pm$. We can fix the three diagonal elements $\gamma_{i,i}$ in $\ll (1 + D\Delta)^3$ ways. By the definition of $S(m, l)$, we have the congruence
\[\gamma_{1, 2} \gamma_{2, 1}\equiv\gamma_{1,1}\gamma_{2,2}\pmod{l},\]
and we have $\ll 1 + D\Delta$ choices for $\gamma_{1, 2} \gamma_{2, 1}$ to be zero, and $\ll 1+(D\Delta)^2/l$ choices for $\gamma_{1, 2} \gamma_{2, 1}$ to be non-zero, which determines the pair $(\gamma_{1,2},\gamma_{2,1})$ up to a divisor function. The same argument applies to the products $\gamma_{1, 3} \gamma_{3, 1}$ and $\gamma_{2, 3} \gamma_{3, 2}$.
\end{proof}

Let $\theta \geq 0$ be an admissible exponent for the Ramanujan conjecture on $\GL_3$. (In particular, $\theta/2$ is admissible for $\GL_2$ by the theory of symmetric squares.) For our purposes, the value \cite[Cor. on p.~515]{MR618323}  \begin{equation}\label{1/2}
\theta = 1/2
\end{equation}
suffices.  The following lemma is an application of the (spherical) pretrace formula with Hecke operators.

\begin{lemma}\label{count2} Let $m, l \in \ZZ_{\geq 1}$ and $0 < \Delta \ll 1$. Then, for any $\eps>0$, we have
\[\#\{\gamma \in S(m, l) : \dist(\tilde{\gamma}, K) \leq \Delta\} \ll_\eps (m/l)^{2+\eps} \Delta^5 + (m/l)^{ 1 + \theta  + \eps}.\]
\end{lemma}

\begin{proof} Following the proof\footnote{We record a small inaccuracy in the statement and its proof: one needs some compatibility for the family of functions $\{\psi_\delta|_A:\delta>0\}$ such as $\psi_\delta(a)=c_\delta\cdot\psi_1(a^{1/\delta})$ with normalization factor $c_\delta\asymp\delta^{-5}$.} of \cite[Lem.~2]{Blomer-Lutsko}, we can fix a smooth, bi-$K$-invariant function $k_{\Delta}:G\to\RR_{\geq 0}$ with the following properties:
\begin{itemize}
\item $k_{\Delta}(x) = 1$ for $\dist(x, K) \leq \Delta$, and $k_{\Delta}(x) = 0$ for $\dist(x, K) \geq 2\Delta$;
\item the spherical transform of $k_\Delta$ as a function of $\mu \in \mfa^{\ast}_{\CC}/W$ is $\ll _A\Delta^5 (1 + \Delta \| \mu \|)^{-A}$ for any $A > 0$ in fixed vertical strips. 
\end{itemize}

We apply the (spherical) pretrace formula with the Hecke operator $T(m,l):=T_{m/l,l,1}$ associated to the double coset
\[S(m,l)=\Gamma\diag(m/l, l, 1)\Gamma.\]
We write $\lambda_\varpi(m,l)$ for the scalar by which this Hecke operator acts on an automorphic representation $(\varpi,V_\varpi)$ occurring in the spectral decomposition of $L^2(\Gamma\bs G\slash K)$. In light of the discussion below \eqref{Hecke-norm}, the Ramanujan conjecture predicts that $ \lambda_{\varpi}(m, l) \ll_\eps (m/l)^{1+\eps}$. This way we obtain
\begin{equation}\label{spec}
\sum_{\substack{\gamma \in S(m, l)\\ \dist(\tilde{\gamma}, K) \leq \Delta}} 1 \leq \sum_{\gamma \in S(m, l)} k_{\Delta}(\gamma) \ll_A \Delta^5 \int \frac{|\lambda_{\varpi}(m, l)|\cdot|\varpi(\id)|^2}{(1 + \Delta \| \mu_{\varpi} \|)^{A}} \,\tilde d\varpi,
\end{equation}
where $\mu_{\varpi}\in\mfa^{\ast}_{\CC}/W$ is the Langlands parameter of $\varpi$ at infinity and $\tilde d\varpi$ is the spectral measure (not to be confused with the Plancherel measure $d\varpi$ that we used before).

For a general linear reductive algebraic group over $\QQ$, the spectral decomposition of the automorphic $L^2$-space is given in \cite[Main Theorem]{MR546601}. For our special case of
\[L^2(\Gamma\bs G\slash K)\cong L^2(\SL_3(\ZZ)\bs\SL_3(\RR)\slash\SO(3)),\]
it can be found in \cite[\S6]{MR1823867} or \cite[\S10.13]{MR2254662}.
It reads
\begin{equation}\label{explicitspec}
\int \Phi(\varpi) \,\tilde d\varpi = \mcS + \mcE_{\min} + \mcE_{\max} + \mcE_{\deg} + \mathcal{C},
\end{equation} 
where up to normalizing factors
\begin{itemize}
\item $\displaystyle \mcS = \sum_{\varpi \text{ cuspidal}} \Phi(\varpi)$ is the contribution of an orthonormal Hecke basis of cusp forms for the group $\SL_3(\ZZ)$,
\item $\displaystyle \mcE_{\min} = \int_{\RR} \int_{\RR} \Phi(E_{\text{min}}(\ast; t_1, t_2)) \, dt_1\, dt_2$ is the contribution of the   minimal parabolic Eisenstein series $E_{\text{min}}(\ast; t_1, t_2)$
corresponding to the $1+1+1$ parabolic subgroup,
\item $\displaystyle \mcE_{\max} = \int_{\RR}\sum_{u \text{ cuspidal}} \Phi(E_{\text{max}}(\ast;u, t)) \,  dt$ is the contribution of the maximal parabolic Eisenstein series $E_{\text{max}}(\ast; u, t)$ corresponding to the $2+1$ parabolic subgroup with an $L^2$-normalized even Hecke--Maa{\ss} cusp form $u$ for the group $\SL_2(\ZZ)$ in the $2$-block,
\item $\displaystyle \mcE_{\deg} = \int_{\RR}  \Phi(E_{\deg}(\ast;  t)) \,  dt$  is the contribution of the  degenerate maximal Eisenstein series $E_{\deg}(\ast; t)$ containing the constant function in the $2$-block,
\item $\mathcal{C} = \Phi(\varpi_0)$ is the contribution of the $L^2$-normalized constant function $\varpi_0$. 
\end{itemize}

We collect the following two bounds for $\varpi(z)$, to be used at $z=\id$. On average we have (\cite[Lem.~1]{Blomer-Lutsko})
\[
\int_{B(\mu)} |\varpi(z)|^2 \,\tilde d\varpi\ll_z (1 + \| \mu \|)^3
\]
for a ball $B(\mu) \subseteq \mfa_{\CC}^{\ast}$ of radius $O(1)$ about any $\mu \in \mfa_{\CC}^{\ast}$. This is in fact another application of the (spherical) pretrace formula in the other direction and says that on average automorphic forms are bounded for $z$ in some fixed compact domain. As a consequence, we obtain that
\begin{equation}\label{avsup}
\int_{\|\mu_\varpi\|\leq T} |\varpi(z)|^2 \,\tilde d\varpi \ll_z (1 + T)^5.
\end{equation}
Alternatively, this statement is also the $n=3$ special case of \cite[Prop.~4.17]{MR4720461}.
For the maximal degenerate Eisenstein series $E_{\deg}(z, t)$ we simply use the convexity bound
\begin{equation}\label{maxEis}
E_{\deg}(z, t) \ll_{z,\eps} (1+|t|)^{3/4 +\eps}
\end{equation}
that follows from the functional equation for Epstein zeta functions \cite[Prop.~10.7.5]{MR2254662}. See \cite[Th.~1]{Blomer} for details as well as a subconvexity bound that replaces $3/4$ with $1/2$. 

In addition, we collect the following two bounds for $\lambda_{\varpi}(m,l)$. By the very feature of $\theta$, for $\varpi$ cuspidal we have 
\begin{equation}\label{ram}
\lambda_{\varpi}(m, l) \ll_{\eps} (m/l)^ {1 + \theta  + \eps}.
\end{equation}
This bound also holds (in stronger form) for minimal and maximal Eisenstein series. The degenerate Eisenstein series $E_{\deg}(z, t)$ has Satake parameters $(p^{1/2 + it}, p^{-1/2 + it}, p^{-2it})$ at a finite prime $p$, hence in this case we have the slightly weaker bound
\begin{equation}\label{ramEis}
\lambda_{E_{\deg}(z, t)}(m, l) \ll_{\eps} (m/l)^{3/2 + \eps}.
\end{equation}

After these preparations, we return to \eqref{spec}, decompose the right-hand side according to \eqref{explicitspec} as
$$\Delta^5 \int \frac{|\lambda_{\varpi}(m, l)|\cdot|\varpi(\id)|^2}{(1 + \Delta \| \mu_{\varpi} \|)^{A}} \,\tilde{d}\varpi =  \mathscr{S} + \mathscr{E}_{\min} + \mathscr{E}_{\max} + \mathscr{E}_{\deg} + \mathscr{C},$$
 and estimate each part   separately. For the cuspidal part and the minimal/maximal Eisenstein series we use \eqref{ram} and \eqref{avsup} to obtain 
$$ \mathscr{S} + \mathscr{E}_{\min} + \mathscr{E}_{\max}  \ll_{\eps} \Delta^5 (m/l)^{ 1+\theta +\eps} \Delta^{-5} = (m/l)^{ 1 + \theta + \eps} .$$
For the maximal degenerate Eisenstein series we use \eqref{ramEis} and \eqref{maxEis} to obtain
$$\mathscr{E}_{\deg}  \ll_{\eps} \Delta^5 (m/l)^{3/2 + \eps} \Delta^{-5/2 - \eps} = (m/l)^{3/2 + \eps} \Delta^{5/2 - \eps}.$$
For the constant function we estimate trivially
$$\mathscr{C} \ll_\eps (m/l)^{2+\eps} \Delta^5.$$
The bound for $\mathscr{E}_{\deg}$ above is dominated by our total bound for $\mathscr{S}+\mathscr{E}_{\min}+\mathscr{E}_{\max}+\mathscr{C}$, hence the proof is complete.
\end{proof}

With these two lemmata at hand, we can now complete the proof of Theorem~\ref{thm2}. We return to \eqref{finalpretrace} and use \eqref{eq:f-decay-2}. In this way we obtain for all $x\in\Omega$ that
\begin{equation}\label{n=3bound}
\sum_{\phi\in\mfB} |\phi(x)|^2 \ll_{\eps,I} \frac{\ell^{4+\eps}}{L^2} \left( L + \sum_{j=1}^3 \sum_{\ell_1, \ell_2 \in \mcP} \frac{1}{L^{2j}} \sum_{\substack{ \gamma\in S(\ell_1^j\ell_2^{2j}, \ell_2^j)\\ \| \tilde{\gamma} \| \ll 1}} \frac{1}{1+ \ell \cdot \dist(\tilde{\gamma}, \{\pm\id\})^2} \right),
\end{equation}
where
$\mcP$ is the set of primes in $[L, 2L]$. In the following we estimate the $\ell_1, \ell_2, \gamma$-sum, which we call $\Sigma_j$. We write 
$\delta := \dist(\tilde{\gamma}, \{\pm \id\})$ 
and split the $\gamma$-sum into dyadic intervals $\delta \asymp \Delta \ll 1$. 

By Lemma~\ref{count1}, the number of $\gamma$'s with $\delta\leq L^{-j}$ is $O(1)$, hence their contribution to $\Sigma_j$ is $O(L^{2-2j})=O(1)$. So we can restrict to the dyadic $\Delta$'s satisfying $L^{-j}\ll\Delta\ll 1$. The number of these $\Delta$'s is $O(\log L)$, so it suffices to treat one at a time. By Lemma~\ref{count1}, the contribution of a given $L^{-j}\ll\Delta\ll 1$ to $\Sigma_j$ is
\[\ll_\eps L^{2-2j+\eps} \frac{L^{6j}\Delta^6}{1 + \ell\Delta^2} < \frac{L^{2+4j+\eps}\Delta^4}{\ell}.\]
Let us have a closer look at the case when the right-hand side exceeds $L^{1+\eps}$, that is,
\[\Delta>\ell^{1/4}L^{-1/4-j}.\]
Then we use Lemma~\ref{count2} (which applies a fortiori if $\dist(\tilde{\gamma}, K)$ is replaced with $\dist(\tilde{\gamma}, \{\pm\id\})$) to bound the same contribution of $\Delta$ to $\Sigma_j$ by
\begin{align*}
& \ll_\eps L^{2-2j+\eps} \frac{L^{4j} \Delta^5 + L^{2j(1+\theta)}}{1 + \ell\Delta^2} 
= L^{\eps} \frac{L^{2+2j} \Delta^5 + L^{2+2j\theta}}{1 + \ell\Delta^2} \\
& \ll L^{\eps}\left( \frac{L^{2+2j}\Delta^3}{\ell} + \frac{L^{2+2j\theta}}{\ell^{3/2}L^{-1/2-2j}}\right) 
\ll L^\eps\left(\frac{L^8}{\ell}+\frac{L^{17/2+6\theta}}{\ell^{3/2}}\right).
\end{align*}
We choose $L$ in such a way that the right-hand side is $\ll L^{1+\eps}$; this will guarantee that the contribution of $\Delta$ to $\Sigma_j$ is always $\ll_\eps L^{1+\eps}$. So we are led to
\[\ell\geq\max(L^7,L^{5+4\theta}) = L^7\]
using \eqref{1/2}. 
We  choose   $L:=\ell^{1/7}$, 
substitute all of this back into \eqref{n=3bound} and conclude that
\[\sum_{\phi\in\mfB}|\phi(x)|^2 \ll_{\eps,I}\ell^{4+\eps} L^{-1} = \ell^{4 - \frac{1}{7} + \eps},\qquad x\in\Omega.\]
This completes the proof of Theorem~\ref{thm2}. 
 
\section*{Acknowledgements}
We are grateful to Gergely Zábrádi for sharing his knowledge with us. Important progress on this project was made during our visits to MFO (2023, 2026) and IAS (2024--2025). We thank both institutions for providing stimulating research environments and for their generous hospitality. We used Gemini, ChatGPT, and Claude for background/literature review and proofreading. In addition, we acknowledge that Gemini came up with a conceptual idea in the proof of Lemma~\ref{lemma:Delta-derivatives}.

\bibliographystyle{alphaurl}
\bibliography{references}

\end{document}